\documentclass{article}

\usepackage{microtype}
\usepackage{graphicx}
\usepackage{tabularx}
\usepackage{subcaption}
\usepackage{booktabs} 

\usepackage{enumerate}
\usepackage{enumitem} 

\usepackage{hyperref}

\usepackage[accepted]{icml2024}

\usepackage{amsmath}
\usepackage{amssymb}
\usepackage{mathtools}
\usepackage{amsthm}

\usepackage[capitalize,noabbrev]{cleveref}

\usepackage{tikz}
\usetikzlibrary{shapes.multipart,patterns,arrows}
\usetikzlibrary{graphs,graphs.standard}

\theoremstyle{plain}
\newtheorem{theorem}{Theorem}[section]

\newtheorem{lemma}[theorem]{Lemma}
\newtheorem{prop}[theorem]{Proposition}
\newtheorem{corollary}[theorem]{Corollary}
\newtheorem{example}[theorem]{Example}
\theoremstyle{definition}
\newtheorem{definition}[theorem]{Definition}
\newtheorem{assumption}[theorem]{Assumption}
\newtheorem{case}[theorem]{Case}
\theoremstyle{remark}
\newtheorem{remark}[theorem]{Remark}

\usepackage[textsize=tiny]{todonotes}

\def\liminf{\mathop{\rm lim\,inf}\limits}
\def\limsup{\mathop{\rm lim\,sup}\limits}

\def\R{\mathbb{R}}

\def\E{\mathbb{E}}
\def\P{\mathbb{P}}

\def\eps{\varepsilon}

\def\param{\boldsymbol{\theta}}
\def\Param{\boldsymbol{\Theta}}

\definecolor{hancolor}{rgb}{0.1, 0.0, 0.9}

\DeclareMathOperator*{\argmin}{arg\,min}

\usepackage{dsfont} 

\DeclarePairedDelimiter\norm{\lVert}{\rVert}

\newcommand{\cX}{\mathcal{X}}

\newcommand{\cF}{\mathcal{F}}
\def\1{\mathds{1}}
\newcommand\freefootnote[1]{%
  \let\thefootnote\relax%
  \footnotetext{#1}%
  \let\thefootnote\svthefootnote%
}

\definecolor{willcolor}{rgb}{0.9, 0.0, 0.1}

\icmltitlerunning{Stochastic Optimization with Arbitrary Recurrent Data Sampling}

\begin{document}

\twocolumn[
\icmltitle{Stochastic Optimization with Arbitrary Recurrent Data Sampling}



\icmlsetsymbol{equal}{*}

\begin{icmlauthorlist}
\icmlauthor{William Powell}{equal,uw}
\icmlauthor{Hanbaek Lyu}{equal,uw}
\end{icmlauthorlist}

\icmlaffiliation{uw}{Department of Mathematics, University of Wisconsin-Madison, WI, USA}

\icmlcorrespondingauthor{William Powell}{wgpowell@wisc.edu}
\icmlcorrespondingauthor{Hanbaek Lyu}{hlyu@math.wisc.edu}

\icmlkeywords{Machine Learning, ICML}

\vskip 0.3in
]



\printAffiliationsAndNotice{\icmlEqualContribution} 

\begin{abstract}
For obtaining optimal first-order convergence guarantees for stochastic optimization, it is necessary to use a recurrent data sampling algorithm that samples every data point with sufficient frequency. Most commonly used data sampling algorithms (e.g., i.i.d., MCMC, random reshuffling) are indeed recurrent under mild assumptions. In this work, we show that for a particular class of stochastic optimization algorithms, we do not need any further property (e.g., independence, exponential mixing, and reshuffling) beyond recurrence in data sampling to guarantee optimal rate of first-order convergence. Namely, using regularized versions of Minimization by Incremental Surrogate Optimization (MISO), we show that for non-convex and possibly non-smooth objective functions with constraints, the expected optimality gap converges at an optimal rate $O(n^{-1/2})$ under general recurrent sampling schemes. Furthermore, the implied constant depends explicitly on the `speed of recurrence', measured by the expected amount of time to visit a  data point, either averaged (`target time') or supremized (`hitting time') over the target 
locations. We discuss applications of our general framework to decentralized optimization and distributed non-negative matrix factorization. 
\end{abstract}

\section{Introduction}
In this paper we consider the minimization of a non-convex weighted finite-sum objective $f: \R^p \to \R$:
\begin{align}
    \param^* \in \argmin_{\param \in \Param} \Big\{f(\param):= \sum_{v \in \mathcal{V}} f^v(\param)\pi(v)\Big\}
    \label{eqn:main_problem}
\end{align}
where $\Param \subseteq \R^{p}$ is a convex, but not necessarily compact feasible set and $\param$ represents the parameters of a model to be optimized. Here $\mathcal{V}$ is a finite index set where one can view each index $v\in \mathcal{V}$ representing (a batch of) data that can be accessed at once. Then $f^v(\param)$ is the loss incurred using parameter $\param$ with respect to data at $v$, which is weighted by $\pi(v)\ge 0$ when forming the overall objective $f$ in \eqref{eqn:main_problem}. Without loss of generality, we assume the $\pi(v)$s sum to one.
When $\pi(v) \equiv \frac{1}{|\mathcal{V}|}$ the problem \eqref{eqn:main_problem} becomes the classical finite sum problem in the optimization literature. Instances of non-uniform $\pi$ arise when training a model with imbalanced data as has been studied in \cite{steininger2021density, wang2022imbanlanced, sow2024doubly}.

 We aim to solve this problem by developing an algorithm which produces iterative parameter updates $\param_n$ given only access to an \textit{arbitrary sequence of data samples} $(v_n)_{n \geq 1}$. In order to reach a first-order stationary point of \eqref{eqn:main_problem} for general objectives, it is necessary to use a sampling algorithm that is \textit{recurrent}, meaning that every data point is sampled infinitely often with `sufficient frequency'. Note that recurrence is satisfied by many common sampling schemes such as i.i.d. (independently and identically distributed) sampling, (irreducible) Markov Chain Monte-Carlo (MCMC), cyclic sampling \cite{bertsekas2011incremental}, and random-reshuffling  \cite{ying2017performance}. The main question we ask in this work is the following: 
\begin{description}
    \item $\bullet$ \textit{Is there any class of stochastic optimization algorithms for which recurrent sampling is enough to obtain optimal first-order convergence guarantee for \eqref{eqn:main_problem}?}
\end{description}
 In this paper, we show that for a class of suitable  extensions of stochastic optimization algorithms known as \emph{Minimization by Incremental Surrogate Optimization} (MISO) \cite{marial2015miso}, no additional property of a data sampling algorithm (e.g., independence, exponential mixing, reshuffling) other than recurrence is needed in order to guarantee convergence to first-order stationary points. Furthermore, we show that the rate of convergence depends crucially on the worst-case (over the initial locations) amount of times to reach a furthest target and/or an average target. 
 These correspond to the notion of `hitting time' and `random target time' in Markov chain theory, respectively.

With the original MISO algorithm in \cite{mairal2013stochastic}, even under the general recurrent data sampling, we are able to obtain asymptotically optimal iteration complexity if we can use strongly convex surrogate functions. However, there is a significant technical bottleneck in showing asymptotic convergence to stationary points, which was classically established in \cite{marial2015miso} in case of the i.i.d. data sampling. We find that using additional regularization helps with improving the convergence rate and allows us to prove asymptotic convergence to stationary points under arbitrary recurrent data sampling. For these reasons, we propose a slight extention of MISO that we call the \emph{Regularized Minimization by Incremental Surrogate Optimization} (RMISO), which takes the following form: 
\begin{description}
    \item[Step 1.] Sample $v_n$ according to a recurrent sampling algorithm
    \item[Step 2.] $g_n^{v_n} \leftarrow$ Convex majorizing surrogate of $f^{v_n}$ at $\param_{n-1}$; $g_n^v = g_{n-1}^v$ for $v \neq v_n$
    \item[Step 3.] $\bar{g}_n \leftarrow \sum_{v \in \mathcal{V}} g_n^v\pi(v)$; Compute
    \begin{align*}
        \param_n \in \argmin_{\param \in \Param} &\Big[ \bar{g}_n(\param) + \Psi(\norm{\param - \param_{n-1}}) \Big].
    \end{align*}
\end{description}

The algorithm maintains a list of majorizing surrogate functions for each data point $v$. At each step, a new data sample $v_{n}$ is drawn according to a recurrent data sampling algorithm. We then find a new majorizing convex surrogate $g_{n}^{v_{n}}$ that is tight at the current parameter $\param_{n-1}$. All other surrogates are unchanged. Then the new parameter $\param_{n}$ is found by minimizing the empirical mean of the current surrogates plus a regularization term $\Psi(\lVert \param-\param_{n-1} \rVert)$ that penalizes large values of $\norm{\param_n - \param_{n-1}}$. To handle dependent data, many algorithms use some form of projection or regularization to achieve this property \cite{lyu2023srmm, bhandari2018TDLearning, roy2023online}. This allows one to control the bias introduced by dependent sampling schemes as well as use the broader class of convex surrogates instead of requiring them to be strongly convex. The original MISO \cite{marial2015miso} is  recovered by omitting this regularization term. The particular choice of this term is crucial for the success of the analysis under the general recurrent data sampling setting. 

 Applications of our work include distributed optimization over networks 
 where $\mathcal{V}$ forms the vertex set of a connected graph $\mathcal{G} = (\mathcal{V}, \mathcal{E})$ and each vertex $v$ stores some data. Prior work \cite{JohanssonDistributed2010, Johansson2007ASP, ram2009incremental, Lopes2007incremental, Mao2020walkman, even2023stochastic, sun2022Adaptive} studies the performance of various optimization algorithms in this setting assuming the sequence $(v_n)_{n \geq 1}$ is a Markov chain 
 on the graph $\mathcal{G}$. Here $\pi$ is typically taken to be uniform and it is frequently assumed that the Markov chain is an MCMC sampling converging to $\pi$, see 
 \cite{sun2022Adaptive, JohanssonDistributed2010, wang2022stability}. 
 
  In this setting we find both theoretically and empirically that convergence of our algorithm can be accelerated by choosing sampling schemes that guarantee a higher frequency of visits to each $v \in \mathcal{V}$. Such schemes may be non-Markovian or not aperiodic and so not guaranteed to converge to a stationary distribution. Moreover, our analysis does not require $\pi$ to agree with the stationary measure of the sampling process if it exists. As remarked in \cite{even2023stochastic}, this additional flexibility may be advantageous as it allows one to opt for more efficient sampling schemes who's stationary measure may not agree with the data-weighting-distribution $\pi$. In our context, these are the schemes which minimize the measures of recurrence we define in the sequel.


\subsection{Contribution}\label{sec: contribution}

Our algorithms and analysis consider three cases which we briefly summarize in the following bullet points. 

\begin{itemize}
    \item[\textbullet] We show convergence rates of $O(n^{-1/2})$ for MISO with strongly convex surrogates or constant quadratic proximal regularization, matching the rate shown for SAG in \cite{even2023stochastic}. The implied constant depends on the potentially much smaller 'target time' rather than the hitting time.

    \item[\textbullet] The same convergence rates hold for MISO with dynamic quadratic proximal regularization where, inspired by the dynamic step size used for SAG in \cite{even2023stochastic}, the regularization parameter is adaptive to the state of the sampling process. Asymptotic convergence of stationarity measures in expectation is also proved.

    \item[\textbullet] Convergence rates of $O(n^{-1/2}\log n)$ are shown for MISO with diminishing search radius restriction, where averaged surrogates are minimized within a diminishing radius. We show almost sure convergence to stationarity for this method.

    \item[\textbullet]  We experimentally validate our results for the tasks of non-negative matrix factorization and logistic regression. We find that our method is robust to data heterogeneity as it produces stable iterate trajectories while still maintaining fast convergence (see Sec. \ref{sec: experiments}).
\end{itemize}

 \subsection{Related Work}

 MISO \cite{marial2015miso} was originally developed to solve finite sum problems under i.i.d sampling and proceeds by repeatedly minimizing a surrogate of the empirical loss function. In \cite{marial2015miso} it is shown that for MISO the expected objective optimality gap $\E[f(\param_n) - f(\param^*)]$ decays at rate $O(1/n)$ when the objective function is convex and exponentially fast when it is  strongly convex, just as batch gradient descent does \cite{Bottou2018optimization}. For non-convex $f$ it is shown that the iterates produced by MISO converge to the set of stationary points of $f$ over a convex constraint set, but no convergence rate analysis is given. Convergence rates for non-convex objectives were later provided for unconstrained problems in \cite{qian2019miso} where it was shown that the expected gradient norm $\E[\norm{\nabla f(\param_n)}]$ decays at rate $O(n^{-1/2})$. This rate was matched for the constrained setting in \cite{Karimi2022MISSO}. However, both papers only consider i.i.d sampling.  

(R)MISO may be compared with \emph{Stochastic Averaged Gradient} (SAG) \cite{schmidtSAG2017} as both store the most recent information computed using the data $v$ and output new parameter updates $\param_n$ depending on an average of this information over $\mathcal{V}$. Recently in \cite{even2023stochastic}, it was shown that for non-convex objectives SAG produces iterates such that the expected gradient norm decays at rate $O(n^{-1/2})$ under Markovian sampling. In comparison, the expected gradient norm converges at rate $O(n^{-1/4})$ for other stochastic first-order methods such as Stochastic Gradient Descent (SGD) \cite{SunOMCGD, AlacaogluNonconvex2023, even2023stochastic, karimi2019nonasymptotic}. Other works devoted to the study of first order optimization methods under Markovian sampling include \cite{beznosikov2023firstordermarkovnoise, bhandari2018TDLearning, wang2022stability, Xie2023BiasExtrap, lyu2023srmm}.


There has also been a recent focus on proving faster convergence for SGD using without-replacement sampling methods such as random-reshuffling \cite{Gurbuzbalaban2021reshufffling, ying2017performance}. This has been further extended to variance reduced algorithms \cite{huang2021improved, malinovsky2023rrwithvariancereduction, Beznosikov2023RRSARAH} and distributed optimization \cite{mishchenko2022FedRR, horvath2022fedshuffle}. New sampling algorithms that aim to improve over random reshuffling have been suggested in \cite{rajput2022permutationbased, lu2022grab, Mohtashami2022characterizing}. In particular, in \cite{lu2022exampleselection} the authors show that the convergence of SGD can be accelerated provided a certain concentration inequality holds and propose leveraging this using a greedy sample selection strategy.

 To obtain our results, we adopt a new analytical approach which is inspired in part by the analysis of SAG in \cite{even2023stochastic}. This strategy differs significantly from mixing rate arguments used in the analysis of stochastic optimization methods with Markovian data (e.g \cite{SunOMCGD, bhandari2018TDLearning, nagaraj2020least, lyu2020online, lyu2022online, lyu2023srmm, AlacaogluNonconvex2023}). We give a short sketch of our proofs in Section \ref{sec: sketch of proofs} and a brief overview of mixing rate techniques and the challenges of adapting them to the analysis of MISO in Appendix \ref{sec: mixing time analysis}. We believe that these techniques may be of interest in their own right and may further contribute to analyzing other stochastic optimization methods with recurrent data sampling.

\subsection{Notation}
In this paper, we let $\R^p$ denote the ambient space for the parameter space $\Param$ equipped with the standard inner product $\langle \cdot, \cdot \rangle$ and the induced Euclidean norm $\norm{\cdot}$. For $\param \in \R^{p}$ and $\eps > 0$, we let $B_{\eps}(\param)$ represent the closed Euclidean ball of radius $\eps$ centered at $\param$. We let $\1(A)$ be the indicator function of an event $A$ which takes value $1$ on $A$ and $0$ on $A^c$. We denote $\pi_{\text{min}} = \min_{v \in \mathcal{V}} \pi(v)$. We let $a \wedge b = \min\{a, b\}$ for real numbers $a$ and $b$. For a set $\cX$ we let $|\cX|$ denote its cardinality. 


\section{Preliminary Definitions and Algorithm Statement}\label{sec: algorithm statement}

In this section we state the two main algorithms used to solve \eqref{eqn:main_problem}. To do this we start by defining first-order surrogate functions and then define a few random variables that will be important in both implementation and analysis. First-order surrogates are defined by

\begin{definition}[First-order surrogates]\label{def: surrogates}
\normalfont
A convex function $g : \R^p \to \R$ is a first-order surrogate function of $f$ at $\param$ if 
\vspace{-0.2cm} 
\begin{itemize}
    \item[(i)] $g(\param') \geq f(\param')$ holds for all $\param' \in \Param$
    \item[(ii)] the approximation error $h:= g - f$ is differentiable and $\nabla h$ is $L$-Lipschitz continuous for some $L > 0$; moreover $h(\param) = 0$ and $\nabla h(\param) = 0$.
\end{itemize}
We denote by $\mathcal{S}_{L}(f, \param)$ the set of all first order surrogates of $f$ at $\param$ such that $\nabla h$ is $L$-Lipschitz. We further define $\mathcal{S}_{L, \mu}(f, \param)$ to be the set of all surrogates $g \in \mathcal{S}_{L}(f, \param)$ such that $g$ is $\mu$-strongly convex.
\end{definition}

Certain properties of the data sampling process are crucial in our analysis, especially in proving Lemmas \ref{lem: conv_proof_surrogate_error_grad_sum}, \ref{lem: gap sum dynamic prox reg}, and \ref{lem: gap sum dr}. Below we define the \emph{return time} and \emph{last passage time}. 
\begin{definition}[Return time]\label{def: return time}
\normalfont
    For $n \geq 0$ and $v \in \mathcal{V}$, the time to return to data $v$ starting from time $n$ is defined as 
    \begin{align}
        \tau_{n, v} = \inf\{j \geq 1: v_{n + j} = v\}.
    \end{align}
\end{definition}
\noindent
That is, $\tau_{n, v}$ is the amount of time which one has to wait after time $n$ for the process to return to $v$. The return time may be viewed as a generalization of the return times of a Markov chain. Indeed, $\tau_{0, v} = \inf\{n \geq 1: v_n = v\}$ agrees with the classical notion of return time from the Markov chain literature. This is closely related to the last passage time defined below.

\begin{definition}[Last passage time]\label{def: passage times}
\normalfont
    For $n \geq 1$ and $v \in \mathcal{V}$ we define the \textit{last passage time} of $v$ before time $n$ as
    \begin{align}
        k^v(n) = \sup\{j \leq n: v_j = v\}.
    \end{align}
    If the process has not yet visited $v$, i.e. $\{j \leq n : v_j = v\} = \emptyset$, then we set $k^v(n) = 1$.
\end{definition}


\begin{algorithm}[ht]
    \small
    \caption{Incremental Majorization Minimization with Dynamic Proximal Regularization}
    \label{RMISO}
    \begin{algorithmic}[1]
    \STATE \textbf{Input:} Initialize $\param_0 \in \Param$; $N > 0$; $\rho \geq 0$
    \STATE \textbf{Option: } $Regularization \in \{Dynamic, Constant\}$
    \STATE  Initialize surrogates $g_0^v \in \mathcal{S}_{L, \mu}(f^v, \param_0)$. 
    \FOR {$n = 1$ {\bfseries to} $N$}
    \STATE sample a data point $v_n$
    \STATE choose $g_n^{v_n} \in \mathcal{S}_{L, \mu}(f^{v_n}, \param_{n-1})$; $g_n^v = g_{n-1}^v \hspace{0.1cm} \forall v \neq v_n$
    \STATE $\bar{g}_n \leftarrow \sum_{v \in \mathcal{V}} g_n^v \pi(v)$
    \IF{$Regularization=Dynamic$}
    \STATE $\rho_n \leftarrow \rho + \max_{v \in \mathcal{V}} (n - k^v(n))$
    \ELSIF{$Regularization=Constant$}
    \STATE $\rho_n \leftarrow \rho$
    \ENDIF
    \STATE $\param_n \leftarrow \argmin_{\param \in \Param} \Big[\bar{g}_n(\param) + \frac{\rho_n}{2}\norm{\param - \param_{n-1}}^2\Big]$
    \ENDFOR
    \STATE \textbf{output:} $\param_N$
    \end{algorithmic}
\end{algorithm}

\begin{algorithm}[ht]
    \small
    \caption{Incremental Majorization Minimization with Diminishing Radius}
    \label{RMISO_DR}
    \begin{algorithmic}[1]
    \STATE \textbf{Input:} Initialize $\param_0 \in \Param$; $N > 0$; $(r_n)_{n \geq 1}$
    \STATE Initialize surrogates $g_0^{v} \in \mathcal{S}_{L}(f^{v}, \param_0)$. 
    \FOR{$n = 1$ {\bfseries to} $N$}
    \STATE sample a data point $v_n$
    \STATE choose $g_n^{v_n} \in \mathcal{S}_{L}(f^{v_n}, \param_{n-1})$; $g_n^v = g_{n-1}^v$ for all $v \neq v_n$
    \STATE $\bar{g}_n \leftarrow \sum_{v \in \mathcal{V}} g_n^v\pi(v)$
    \STATE $\param_n \leftarrow \argmin_{\param \in \Param \cap B_{r_n}(\param_{n-1})} \bar{g}_n(\param)$
    \ENDFOR
    \STATE \textbf{output:} $\param_N$
    \end{algorithmic}
\end{algorithm}

The last passage time $k^v(n)$ appears naturally as it is the last time the surrogate for data point $v$ has been updated during the execution of either Algorithm \ref{RMISO} or \ref{RMISO_DR}. Thus, $g_n^v$ is a surrogate of $f^v$ at $\param_{k^v(n) - 1}$ and the corresponding surrogate error at this point $h_n^v(\param_{k^v(n) - 1})$ and its gradient are equal to zero. We will use this fact crucially in the proof of the key lemma, Lemma \ref{lem: conv_proof_surrogate_error_grad_sum}. 


Our algorithms are stated formally in Algorithms \ref{RMISO} and \ref{RMISO_DR}. In Algorithm \ref{RMISO}, the regularization term uses \emph{Proximal Regularization} (PR) while Algorithm \ref{RMISO_DR} utilizes a \emph{Diminishing Radius} (DR) restriction. 


\section{Main Results}\label{sec: main_results}

\subsection{Optimality Conditions}
We now introduce the optimality conditions used in this paper and related quantities. Here we denote $f$ to be a general objective function $f: \Param \to \R$, but elsewhere $f$ will refer to the objective function in \eqref{eqn:main_problem} unless otherwise stated. 

For a given function $f$ and $\param^{*}, \param \in \Param$, we define its \emph{directional derivative} at $\param^{*}$ in the direction $\param - \param^{*}$ as
{\small
\begin{align}
    \nabla f(\param^{*}, \param - \param^{*}) := \lim_{\alpha \to 0^{+}} \frac{ f(\param^{*} + \alpha(\param - \param^{*})) - f(\param^{*})}{\alpha}
\end{align}
}

\vspace{-0.1cm}
\noindent A necessary first order condition for $\param^*$ to be a local minimum of $f$ is to require $\nabla f(\param^{*}, \param - \param^{*}) \geq 0$ for all $\param \in \Param$ (see \cite{marial2015miso}). Thus we define the optimality of $f$ at $\param^* \in \Param$ as 
\begin{align}
    O_f(\param^*) := \sup_{\param \in \Param, \norm{\param - \param^*} \leq 1} - \nabla f(\param^*, \param - \param^*).
\end{align}
Note that $O_f(\param^*)$ is non-negative (since we may take $\param = \param^*$) and only positive if there exists some $\param \in \Param$ with $\nabla f(\param^{*}, \param - \param^{*}) < 0$.
 Thus we say that $\param^{*} \in \Param$ is a \emph{stationary point} of $f$ over $\Param$ if $O_f(\param^*) = 0$.
If $f$ is differentiable and $\Param$ is convex, this is equivalent to $-\nabla f(\param^*)$ being in the normal cone of $\Param$ at $\param^*$. If $\param^*$ is in the interior of $\Param$ then it implies that $\norm{\nabla f(\param^*)} = 0$.

For iterative algorithms, this stationary point condition may hardly be satisfied in a finite number of iterations. A practically important question is how the worst case number of iterations required to achieve an $\eps$-approximate solution scales with the desired precision $\eps$. We say that $\param^* \in \Param$ is an $\eps$-\emph{approximate stationary point} of $f$ over $\Param$ if $O_f(\param^*) \leq \eps$.
This notion of $\eps$-approximate solution is consistent with the corresponding notion for unconstrained problems. In fact, if $f$ is differentiable, and if $\param^{*}$ is distance at least one away from the boundary $\partial \Param$, then it reduces to $\norm{\nabla f(\param^*)} \leq \eps$. For each $\eps > 0$, we then define the \emph{worst-case iteration complexity} of an algorithm for solving \eqref{eqn:main_problem} as 
\begin{align}
    N_{\eps}(\param_0):= \inf\{n \geq 1 : O_f(\param_n) \leq \eps\},
\end{align}
where $(\param_n)_{n \geq 0}$ is a sequence of iterates produced by the algorithm with initial estimate $\param_0$. 

\subsection{Assumptions}\label{sec: assumptions}
In this subsection, we state our assumptions for establishing the main results. Throughout this paper, we denote by $\cF_n$ the $\sigma$-algebra generated by the samples $v_1, \hdots, v_n$ and the parameters $\param_0, \hdots, \param_n$ produced by Algorithm \ref{RMISO} or \ref{RMISO_DR}. With this definition, $(\cF_n)_{n \geq 1}$ defines a filtration. 

In what follows we will also define some important quantities in terms of the measure theoretic definition of the $L_{\infty}$ norm for random variables:  
\begin{align}
    \norm{X}_{\infty} =  \inf \{ t > 0 : \P(|X| > t) = 0\}.
\end{align}
This is due to the technical consideration that the conditional expectation $\E[\tau_{n, v} | \cF_n]$ is random and hence so is $\sup_{n \geq 1} \E[\tau_{n, v} | \cF_n]$. Our analysis requires this supremum to be bounded by a non-random constant, while in the fully general case $\sup_{n \geq 1} \E[\tau_{n, v} | \cF_n]$ may be an unbounded. 

We first state our main assumption on the sampling scheme.

\begin{assumption}[Recurrent data sampling]\label{assumption: process properties}
\normalfont
    The sequence $(v_n)_{n \geq 1}$ of data samples defines a stochastic process which satisfies the following property: for each $v \in \mathcal{V}$, $ \sup_{n \geq 1} \left \lVert \E[\tau_{n, v} | \cF_n] \right \rVert_{\infty}  <  \infty$, i.e., the expected return time conditioned on $\cF_n$ is uniformly bounded.
\end{assumption}

Assumption \ref{assumption: process properties} states that the data $(v_n)_{n \geq 1}$ are sampled in such a way that the expected time between visits to a particular data point is finite and uniformly bounded. Generalizing the notion of positive recurrence in Markov chain theory, we say a sampling algorithm is \emph{recurrent} if Assumption \ref{assumption: process properties} is satisfied. We emphasize that recurrence is the \textit{only} requirement we make of the sampling process in order to prove the convergence rate guarantees in Theorem \ref{thm: convergence rates to stationarity} and the asymptotic convergence in Theorem \ref{thm: a.s. convergence} (We do not assume independence or Markovian dependence, etc.). We include below a list of some commonly used recurrent sampling algorithms.

\begin{description}
    \item[1.] (\textit{i.i.d. sampling}) Sampling data i.i.d from a fixed distribution is the most common assumption in the literature \cite{mairal2013stochastic, marial2015miso, Bottou2018optimization, schmidtSAG2017, Rie2013svrg}. Suppose we sample $v_n$ i.i.d from some distribution $\gamma$ on $\mathcal{V}$. Then the $\tau_{n, v}$ are independent geometric random variables taking values from $\{1,2,\dots,\}$ with success probability $\gamma(v)$, so $\E[\tau_{n, v} |\cF_n] = 1/\gamma(v)$. In particular, if $\gamma$ is uniform $\E[\tau_{n, v} |\cF_n] = |\mathcal{V}|$ for all $n$ and $v$. 
    \item[2.] (\textit{MCMC}) Markov chain Monte Carlo methods (see e.g. Ch.3 of \cite{levin2017markov}) produce a Markov chain $(v_n)$ on $\mathcal{V}$. If this chain is irreducible then $\max_{v, w} \E_w[\tau_{0, v}]$ is finite \cite{levin2017markov}. 
    For any $n$, $v$, and initial distribution $\nu$, the Markov property implies $\E_{\nu}[\tau_{n, v} |\cF_n] = \E_{v_n}[\tau_{0, v}]$. So any irreducible Markov chain satisfies \ref{assumption: process properties}. 
    
    \item[3.] (\textit{Cyclic sampling}) In cyclic sampling one samples data in order according to some enumeration until the dataset is exhausted. This process is then repeated until convergence. The authors of \cite{lu2022exampleselection} show that iteration complexity for SGD can be improved from $O(\eps^{-4})$ to $O(\eps^{-3})$ using such methods.
    To see that \ref{assumption: process properties} holds in this setting, we simply notice that $\tau_{n, v} \leq |\mathcal{V}|$ for all $n$ and $v$. 
    
    \item[4.] (\textit{Reshuffling}) Reshuffling is similar to cycling sampling except that the dataset is randomly permuted at the beginning of each epoch \cite{lu2022exampleselection}. It was observed empirically that random reshuffling performs better that i.i.d sampling in \cite{Bottou2012sgdtricks} and further studied in \cite{Gurbuzbalaban2021reshufffling, lu2022exampleselection}. The authors of \cite{lu2022exampleselection} show the same improvement in iteration complexity for SGD as for cyclic sampling. 
    In this case, \ref{assumption: process properties} is satisfied since $\tau_{n, v} \leq 2|\mathcal{V}|$. 
\end{description}

In Markov chain theory, the quantity $\max_{v, w} \E_w[\tau_{0,v}]$ is commonly denoted $t_{\text{hit}}$. Adapting this notion, we define
\begin{align}
    t_{\text{hit}}:= \max_{v \in \mathcal{V}} \, \sup_{n \geq 1} \,\, \left \lVert \E[\tau_{n, v} |\cF_n] \right \rVert_{\infty}
    \label{eq: thit}
\end{align}
for each $v$ when \ref{assumption: process properties} holds, for general sampling schemes. 

Continuing the connection with Markov chains, we also let
\begin{align}
    t_{\odot} := \sup_{n \geq 1} \,\,  \left \lVert \sum_{v \in \mathcal{V}} \E[\tau_{n, v} |\cF_n] \pi(v) \right \rVert_{\infty}
    \label{eq: target time}
\end{align}
where $\pi$ is as in \eqref{eqn:main_problem}. Note that if $v_n$ is an irreducible Markov chain then the Markov property implies 
\begin{align}
    t_{\odot} = \max_{w \in \mathcal{V}} \sum_{v \in \mathcal{V}} \E_w[\tau_{0, v}]\pi(v).
    \label{eq: target time plus}
\end{align}
This is closely related to the \emph{target time} define by 
\begin{align}
    t_{\odot}^w = \sum_{v \in V} \E_w[\tau_v]\pi(v)
    \label{eq: targ_time markov}
\end{align}
with the difference being that here $\tau_v = \inf\{n \geq 0 : v_n = v\}$ is the first hitting time of $v$ rather than the first return time to $v$. The \emph{random target lemma} (Lemma 10.1 in \cite{levin2017markov}) states that if $\pi$ is the stationary distribution of the Markov chain, then the target time is independent of the starting state $w$. In this case, the quantities \eqref{eq: target time plus} and \eqref{eq: targ_time markov} only differ by one. This can be seen by first noting that $\E_{w}[\tau_v] = \E_w[\tau_{0, v}]$ if $v \neq w$. This leaves only the difference $\E_w[\tau_{0, w}]\pi(w) - \E_w[\tau_w]\pi(w)$. The second term is equal to zero and the first equals one since if $\pi$ is the unique stationary distribution for the chain, $\pi(w) = \frac{1}{\E_w[\tau_{0,w}]}$ \cite{levin2017markov}.

\begin{figure}[h!]
\centering
\resizebox{0.23\textwidth}{!}{
\begin{tikzpicture}
\graph[simple, nodes={circle, draw, fill=black}, empty nodes] 
{ subgraph K_n [n=8,clockwise,radius=2cm];
10[x=4, y=0] -- 3
};
\end{tikzpicture}
}
\caption{Lonely graph}
\label{fig: lonely_node_graph}
\end{figure}
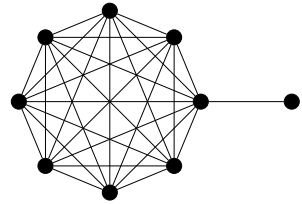

For transitive irreducible Markov chains, $t_{\text{hit}}$ and $t_{\odot}$ are comparable \cite{levin2017markov}.  However, in other situations $t_{\odot}$ may be much smaller that $t_{\text{hit}}$. 
For instance, consider the simple random walk on the graph in Figure \ref{fig: lonely_node_graph}, which we call the `lonely' graph, and let $\pi$ be its stationary distribution. In the lonely graph, $|\mathcal{V}| - 1$ vertices form a clique and the remaining vertex has degree one. A random walk on this graph has worst case hitting time $t_{\text{hit}} = O(|\mathcal{V}|^2)$: its value is tied to the lonely vertex with degree one which has low probability of being visited. On the other hand $t_{\odot}$ is only $O(|\mathcal{V}|)$ since it only depends on an \emph{average} hitting time instead of the worst case. See Example 10.4 in \cite{levin2017markov} for more details. 

We again emphasize that in our general setting, we do not require $(v_n)$ to be a Markov chain nor do we require $\pi$ to be a stationary measure for the process. The above analysis serves to motivate the definition given in \eqref{eq: target time} and provide an example of where $t_{\odot}$ may be much smaller than $t_{\text{hit}}$.

We next list some assumptions of the functions $f^v$. 
\begin{assumption}[Lower-bounded objective and directional derivatives]\label{assumption: directional derivatives}
\normalfont
    For all $v \in \mathcal{V}$ and $\param, \param' \in \Param$, the function $f^v$ is bounded below, i.e. $\inf_{\param \in \Param} f^v(\param) > -\infty$. Moreover, the directional derivative $\nabla f^v(\param, \param' - \param)$ exists.
\end{assumption}
Assumption \ref{assumption: directional derivatives} implies that the objective $f$ is bounded below. For the remainder of this paper we will denote $\Delta_0 := \bar{g}_0(\param_0) - \inf_{\param \in \Param} f(\param)$. It is important to note that if the initial surrogates $g_0^v$ are in $\mathcal{S}_{L}(f^v, \param_0)$ (as is the case in both Algorithms \ref{RMISO} and \ref{RMISO_DR}) then $g_0(\param_0) = f(\param_0)$ so $\Delta_0 = f(\param_0) - \inf_{\param \in \Param} f(\param)$. The regularity assumption in \ref{assumption: directional derivatives} was used in \cite{marial2015miso} and is necessary in analyzing our algorithms using our definition of approximate stationarity.  For Algorithm \ref{RMISO_DR} we make the following stronger but common assumption which is crucial to our analysis:
\begin{assumption}
\label{assumption: Lsmooth}
\normalfont
    For each $v \in \mathcal{V}$, the function $f^v$ is continuously differentiable and $\nabla f^v$ is $L$-Lipschitz continuous.
\end{assumption}
\noindent
For simplicity, we assume that if Assumption \ref{assumption: Lsmooth} holds then the Lipschitz constant of $f^v$ agrees with that of the corresponding approximation error $h^v$. 

Finally, Assumption \ref{assumption: sequences} states that the radii in Algorithm \ref{RMISO_DR} decrease slowly, but not too slowly. This is analogous to square summability of step sizes in gradient descent. 

\begin{assumption}[Square-summable and non-summable radii]\label{assumption: sequences}
\normalfont
    The sequence $(r_n)_{n \geq 1}$ is non-increasing, $\sum_{n = 1}^{\infty} r_n = \infty$, and $\sum_{n = 1}^{\infty} r_n^2 < \infty$. 
\end{assumption}

\subsection{Statement of main results}


In this section we state the two main results of this work. We consider the following three cases corresponding to the three variants of our main algorithm:

\begin{case}\label{case: prox_regularization-iid}
    Assumptions \ref{assumption: process properties}-\ref{assumption: directional derivatives} hold. Use Algorithm \ref{RMISO} with $Regularization Schedule = Constant$. 
\end{case}


\begin{case}\label{case: prox_regularization-markovian} 
    Assumptions \ref{assumption: process properties}-\ref{assumption: directional derivatives} hold. Use Algorithm \ref{RMISO} with $Regularization Schedule = Dynamic$.
\end{case}

\begin{case}\label{case: diminishing_radius-markovian}
    Assumptions \ref{assumption: process properties}-\ref{assumption: sequences} hold. Use Algorithm \ref{RMISO_DR}.
\end{case}

Notice that in Case \ref{case: prox_regularization-iid}, if one chooses $\rho = 0$ and the surrogates are $g_n^v$ are in $\mathcal{S}_{L, \mu}(f^v, \param_{n-1})$ for some $\mu > 0$, then Algorithm \ref{RMISO} reduces to the classical MISO algorithm in \cite{marial2015miso}. 

Our first main result, Theorem \ref{thm: convergence rates to stationarity}, gives worst case upper-bounds on the expected rate of convergence to optimality. For each of the cases \ref{case: prox_regularization-iid}-\ref{case: diminishing_radius-markovian} we give rates of convergence for the objective function $f$. 

\begin{theorem}[Rate of Convergence to Stationarity]\label{thm: convergence rates to stationarity}
Algorithms \ref{RMISO} and \ref{RMISO_DR} satisfy the following for any $N \geq 1$:

\vspace{-0.2cm}
\begin{enumerate}[label=(\roman{*}), itemsep=-0.1cm]
\item \label{item: conv to stationarity-miso}
Assume Case \ref{case: prox_regularization-iid}. Further assume $\rho = 0$ and Algorithm \ref{RMISO} is run with $\mu$-strongly convex surrogates. Then
\begin{align}
    \min_{1 \leq n \leq N} \E[O_f(\param_n)] \leq \frac{Lt_{\odot}\sqrt{\frac{2\Delta_0}{\mu}}}{\sqrt{N}}
\end{align}
\item \label{item: conv to stationarity- alg 2 iid} Assume Case \ref{case: prox_regularization-iid}. If  $\rho \leq Lt_{\odot} \leq \rho + \mu$ then
\begin{align}
    \min_{1 \leq n \leq N} \E \left[ O_f(\param_n) \right] \leq 2\sqrt{\frac{2\Delta_0Lt_{\odot}}{N}}.
\end{align}

\item \label{item: conv to stationarity- alg 2} Assume Case \ref{case: prox_regularization-markovian}. If $\rho \leq Lt_{\odot} \leq \rho + \mu$ then 
\begin{multline}
    \min_{1 \leq n \leq N} \E \left[ O_f(\param_n) \right]\\
    \leq 2\sqrt{\frac{ 2\Delta_0(Lt_{\odot} + (2 t_{\text{hit}} + 1)\log_2 (4|\mathcal{V}|))}{N}}.
\end{multline}

\item \label{item: conv to stationarity- alg 1} Assume Case \ref{case: diminishing_radius-markovian}. Let $C_N = \sum_{n = 1}^N r_n^2$. Then
\begin{multline}
\hspace{-0.3cm} \min_{1 \leq n \leq N} \E \left[O_f(\param_n) \right]\\
\leq \frac{\Delta_0 + \sqrt{\frac{2L}{\pi_{\text{min}}}C_N\Delta_0} + \Big(3 + t_{\odot}\Big)C_NL}{\sum_{n = 1}^N (1 \wedge r_{n+1})}.
\end{multline}

\end{enumerate}
\end{theorem}

To our best knowledge, the rates of convergence given in Theorem \ref{thm: convergence rates to stationarity} are entirely new for first-order algorithms with general recurrent data sampling. In contrast to the convergence result for SAG in \cite{even2023stochastic} that depends on the hitting time $t_{\text{hit}}$, Algorithm \ref{RMISO} with constant proximal regularization (case \ref{case: prox_regularization-iid}) depends  on the possibly much smaller  target time $t_{\odot}$. See Table \ref{tab: complexity_comparison} for a comparison of our results with other works concerning non-convex optimization with non-i.i.d data.

We remark that items \ref{item: conv to stationarity-miso} and \ref{item: conv to stationarity- alg 2 iid} show the potential benefit of using proximal regularization even if the surrogates are already strongly convex. For non-convex $f$, the strong convexity parameter $\mu$ of any surrogates cannot be larger than the Lipschitz constant $L$. However, we are free to choose $\rho$. So an optimal choice of $\rho$ results in dependence on $\sqrt{Lt_{\odot}}$ in \ref{item: conv to stationarity- alg 2 iid} instead of the linear dependence in \ref{item: conv to stationarity-miso}. However, it is not necessary to chose $\rho$ in the range given in items \ref{item: conv to stationarity- alg 2 iid} and \ref{item: conv to stationarity- alg 2}. We include a more general version of Theorem \ref{thm: convergence rates to stationarity}, Theorem \ref{thm: convergence rates to stationarity-extended}, in Appendix \ref{sec: remarks on main results} which shows that these convergence rates hold for arbitrary  $\rho$ and $\mu$ so long as $\rho + \mu > 0$. Overall, our theory suggests that one can improve convergence by using sampling schemes that cover the dataset most efficiently, i.e. those that minimize $t_{\odot}$ and $t_{\text{hit}}$. See the remarks in Appendix \ref{sec: remarks on main results} for more on this topic.

 We also remark that Theorem 1 in \cite{even2023stochastic} gives a lower bound in terms of $t_{\text{hit}}$. While our results depend $t_{\odot}$, they do not contradict this lower bound. See Appendix \ref{sec: remarks on main results} for more discussion.

\begin{table*}[ht]
\centering
\caption{Comparison of iteration complexity for non-convex optimization with non-i.i.d data. The notation $\tilde{O}(\cdot)$ omits logarithmic factors. Here $t_{\text{mix}}$ represents dependence on the mixing time of the Markov chain.
}
\vspace{0.1in}
\resizebox{0.9\textwidth}{!}{
\begin{tabular}{lllll}
\hline
& \rule{0pt}{1.0\normalbaselineskip} Iteration complexity & Memory \vspace{0.1cm} & Sampling & Sampling dependence \\
\hline
 \rule{0pt}{1.0\normalbaselineskip} AdaGrad \cite{AlacaogluNonconvex2023} & $\tilde{O}(\eps^{-4})$ & $O(1)$ & Markovian & $O(\sqrt{t_{\text{mix}}})$\\
\hline
\rule{0pt}{1.0\normalbaselineskip} SGD \cite{SunOMCGD, even2023stochastic} & $\tilde{O}(\eps^{-4})$ & $O(1)$ & Markovian & $O(\sqrt{t_{\text{mix}}})$\\
\hline
\rule{0pt}{1.0\normalbaselineskip} SGD \cite{mishchenko2020reshuffling, lu2022exampleselection} & $O(\eps^{-3})$ & $O(1)$ & Reshuffling & $O(\sqrt{|\mathcal{V}|})$\\
\hline
\rule{0pt}{1.0\normalbaselineskip} SAG \cite{even2023stochastic} & $O(\eps^{-2})$ & $O(|\mathcal{V}|)$ & Markovian & $O(\sqrt{t_{\text{hit}}})$\\
\hline
 \rule{0pt}{1.0\normalbaselineskip} RMISO Case \ref{case: prox_regularization-iid} & $O(\eps^{-2})$ & $O(|\mathcal{V}|)$ & Recurrent & $O(\sqrt{t_{\odot}})$ \\
 \hline
 \rule{0pt}{1.0\normalbaselineskip} RMISO Case \ref{case: prox_regularization-markovian} & $O(\eps^{-2})$ & $O(|\mathcal{V}|)$ & Recurrent & $O(\sqrt{t_{\odot} + t_{\text{hit}}})$ \\
\hline
\rule{0pt}{1.0\normalbaselineskip} RMISO Case \ref{case: diminishing_radius-markovian} & $\tilde{O}(\eps^{-2})$ & $O(|\mathcal{V}|)$ & Recurrent & $O(t_{\odot})$ \\
\hline
\end{tabular}
}
\label{tab: complexity_comparison}
\end{table*}

Our second result, Theorem \ref{thm: a.s. convergence}, concerns the asymptotic behavior of Algorithms \ref{RMISO} and \ref{RMISO_DR}. 

\begin{theorem}[Global Convergence]\label{thm: a.s. convergence}
Algorithms \ref{RMISO} and \ref{RMISO_DR} have the following asymptotic convergence properties:
\begin{enumerate}[label=(\roman{*}), itemsep=-0.1cm]
    \item \label{item: stationary_limit_points-dynamic} For Case \ref{case: prox_regularization-markovian}, we have $\lim_{n \to \infty} \E[O_f(\param_n)] = 0$ and $\lim_{n \to \infty} \E[O_f(\param_n)^2] = 0$.

    \item \label{item: stationary_limit_points} For Case \ref{case: diminishing_radius-markovian} almost surely every limit point of $(\param_n)_{n \geq 1}$ is stationary for $f$ over $\Param$.
\end{enumerate}

\end{theorem}

Theorem \ref{thm: a.s. convergence} shows that although RMISO with diminishing radius requires computing a projection at each step and has higher order dependence on $t_{\odot}$, it enjoys the strongest asymptotic guarantees. RMISO with dynamic proximal regularization is somewhere in the middle. It has lower order dependence on $t_{\odot}$ than the diminishing radius version but also depends on $t_{\text{hit}}$. 
However, we are able to show that both the first and second moments of the optimality gap converge to zero. In particular, notice that in the familiar case that $\Param = \R^p$ and $f$ is differentiable \ref{item: stationary_limit_points-dynamic} implies that $\lim_{n \to \infty} \E[\norm{\nabla f(\param_n)}] = 0$.

Though Algorithm \ref{RMISO} with constant proximal regularization is the simplest of our proposed methods and has the best dependence on the constants $L$ and $t_{\odot}$, it appears that stronger regularization schemes as in cases \ref{case: prox_regularization-markovian} and \ref{case: diminishing_radius-markovian} are needed for  obtaining asymptotic convergence guarantees.
We refer the reader to Appendix \ref{sec: remarks on main results} as well as remark \ref{rem: no convergence constant pr} in Appendix \ref{sec: asymptotic analysis} for a more detailed discussion on the technical difficulties of proving asymptotic convergence for Case \ref{case: prox_regularization-iid} as well as proving it in the almost sure sense for Case \ref{case: prox_regularization-markovian}.

\subsection{Sketch of proofs}
\label{sec: sketch of proofs}

In this section we provide a short sketch of our analysis in order to convey the main ideas. Let $\bar{h}_n = \sum_{v \in V} h_n^v \pi(v)$ be the average surrogate approximation error at step $n$. The key step in our analysis (Lemma \ref{lem: conv_proof_surrogate_error_grad_sum}) is to  prove
 \begin{align}
     \sum_{n = 1}^{N} c_n \E[\norm{\nabla \bar{h}_n(\param_n)}] = O \left( \left(\sum_{n = 1}^{N} c_n^2 \right)^{1/2} \right),
     \label{eq: sketch_gap_sum}
 \end{align}
 where $c_n$ is any non-increasing sequence. For simplicity, assume that the surrogate functions $g_n^v$ and the objective functions $f^v$ are differentiable and we are in the unconstrained setting, i.e. $\Param = \R^p$. If $\param_n$ is a minimizer of $\bar{g}_n$ then $\nabla \bar{g}_n(\param_n) = 0$. Therefore $\norm{\nabla \bar{h}_n(\param_n)} = \norm{\nabla f(\param_n)}$ and \eqref{eq: sketch_gap_sum} implies
 \begin{align}
     \sum_{n = 1}^{N} c_n\E[\norm{\nabla f(\param_n)}]  = O \left( \left(\sum_{n = 1}^N c_n^2 \right)^{1/2} \right).
 \end{align}
 If we take $c_n = 1$, we can then conclude that $\min_{1 \leq n \leq N} \E[\norm{\nabla f(\param_n)}] = O(N^{-1/2})$.  The addition of regularization introduces an added complication because we are no longer directly minimizing $\bar{g}_n$ on the entire feasible set and so do not have $\nabla \bar{g}_n(\param_n) = 0$. However, as we argue in Section \ref{sec: convergence rate analysis}, the added regularization is not too strong asymptotically.


 The main idea is to focus in the individual error gradients because each has the property $\nabla h_n^v(\param_{k^v(n)-1}) = 0$. By Definition \ref{def: surrogates}, we then have $\norm{\nabla h_n^v(\param_n)} \leq L \norm{\param_n - \param_{k^v(n) - 1}}$, so to show \eqref{eq: sketch_gap_sum} we only need to prove 
 \begin{align}
     \sum_{n = 1}^N c_n\E[\norm{\param_n - \param_{k^v(n)-1}}] = O \left( \left( \sum_{n = 1}^N c_n^2 \right)^{1/2} \right) 
 \end{align}
 The triangle inequality and monotonicity of $(c_n)$ imply
 \begin{align}
     c_n\norm{\param_n - \param_{k^v(n)-1}}  \leq \sum_{i = k^v(n)}^n c_i \norm{\param_i - \param_{i-1}},
\label{eq: sketch_schwarts_breakup}
\end{align}
so we can relate the error $\norm{\nabla h_n^v(\param_n)}$ to the sequence $(\norm{\param_n - \param_{n-1}}^2)_{n \geq 1}$. The crucial role played by the regularization or strong convexity is that we can prove the following iterate stability: $\sum_{n = 1}^{\infty} \norm{\param_n - \param_{n-1}}^2 < \infty$ a.s. This idea was also used in \cite{lyu2021BCD,lyu2022online,lyu2023srmm}.

Under Assumption \ref{assumption: process properties} one can expect $\E[n - k^v(n)] \leq M$ for some $M$. One can then intuitively view the expectation of the right hand side of \eqref{eq: sketch_schwarts_breakup} similarly to $\sum_{i = n - M}^n c_i\E[\norm{\param_i - \param_{i-1}}]$. Summing this from $n = 1$ to $N$ we conclude that for a positive constant $C$,
\begin{multline}
      \sum_{n =1}^N \E[\norm{\nabla h_n^v(\param_n)}] \approx C \sum_{n = 1}^N c_n \E[\norm{\param_n - \param_{n-1}}].
\end{multline}
By Cauchy-Schwartz and the iterate stability obtained through regularization, the right hand side is $O \left( \left( \sum_{n = 1}^{N} c_n^2 \right)^{1/2} \right)$. Full details of our analysis are given in Appendices \ref{sec: preliminary lemmas}, \ref{sec: convergence rate analysis}, and \ref{sec: asymptotic analysis}.

\section{Applications and Experiments}\label{sec: applications and experiments}

\subsection{Applications}
In this section we give some applications of our general framework. These include applications to matrix factorization as well as a double averaging version of RMISO derived by using prox-linear surrogates.

\subsubsection{Distributed Matrix Factorization}\label{sec: nmf}

Before beginning this section we define some additional notation. For a collection of matrices $\{A_v\}_{v \in \mathcal{V}} \subset \R^{n \times m}$ we let $[A_v; v \in \mathcal{V}]$ be their concatenation along the horizontal axis. For a set $\Param \subset \R^{n \times m}$ we let $\Param^{\mathcal{V}} = \{[A_v ; v \in \mathcal{V}] : A_v \in \Param \text{ for all } v\}$.

We consider the matrix factorization loss $f(W, H) = \frac{1}{2}\norm{X - WH}_{F}^2 + \alpha \norm{H}_1$ where $X \in \R^{p \times d}$ is a given data matrix to be factored into the product of dictionary $W \in \Param_W \subseteq \R^{p \times r}$ and code $H \in \Param_H \subseteq \R^{r \times d}$ with $\alpha \geq 0$ being the $L_1$-regularization parameter for $H$. Here $\Param_W$ and $\Param_H$ are convex constraint sets. 

Suppose we have a connected graph $\mathcal{G} = (\mathcal{V}, \mathcal{E})$ where each vertex stores 
a matrix $X_v \in \R^{p \times d}$. For each $v \in \mathcal{V}$ define the loss function 
\begin{align}
    &f^v(W) = \inf_{H \in \Param_H} \frac{1}{2}\norm{X^v - WH}_{F}^2 + \alpha\norm{H}_1,
\end{align}
which is the minimum reconstruction error for factorizing $X_v$ using the dictionary $W$. In this context, the empirical loss to be minimized is $\frac{1}{|\mathcal{V}|} \sum_{v \in \mathcal{V}}f^v(W)$.
Note that this problem is not convex. Indeed, letting $X = [X_{v}; v\in V]$ it is equivalent to finding $(W^*, H^*) \in \Param_W \times \Param_H^{\mathcal{V}}$ minimizing $\frac{1}{2}\norm{X - WH}_F^2 + \alpha \norm{H}_1$, 
which is a constrained non-convex optimization problem with a bi-convex loss function. 

In order to apply RMISO let $W_{n-1} \in \Param_W$ be the previous dictionary and denote 
\begin{align}
    \label{eq: code update}
    H_n^v \in \argmin_{H \in \Param_H^\mathcal{V}} \frac{1}{2} \norm{X_v - W_{n-1}H}_F^2 + \alpha \norm{H}_1
\end{align}
if $v_n = v$ and otherwise $H_n^v = H_{n-1}^v$.
Then the function $g_n^v(W) := \frac{1}{2} \norm{X_v - W_{n-1}H_n^v}_{F}^2 + \alpha\norm{H_n^v}_1$ is a majorizing surrogate of $f^v$ at $W_{n-1}$ and belongs to $\mathcal{S}_{L'}(f^v, W_{n-1})$ for some $L' > 0$ (see Ex. \ref{ex: variational surrogates}). 
Then Algorithms \ref{RMISO} and \ref{RMISO_DR} can be used with these surrogates. 


\subsubsection{Prox-Linear Surrogates}
Suppose each $f^v$ is differentiable and has $L$-Lipschitz continuous gradients. Then the functions 
\begin{align}
   \hspace{-0.3cm} g^v(\param) = f^v(\param') + \langle \nabla f^v(\param'), \param - \param' \rangle + \frac{L}{2}\norm{\param - \param'}
\end{align}
are in $\mathcal{S}_{2L, L}(f^v, \param')$ (see Example \ref{ex: prox_linear surrogates}) . Further suppose that $\pi$ is the uniform distribution. Using these surrogates, the update according to Algorithm \ref{RMISO} is 
\begin{equation}
    \begin{cases}
        \bar{\param}_{n-1} &\leftarrow \frac{1}{|\mathcal{V}|}\sum_{v \in \mathcal{V}} \param_{k^v(n)-1}\\
        \bar{\nabla}_{n-1} &\leftarrow \frac{1}{|\mathcal{V}|}\sum_{v \in \mathcal{V}} \nabla f^v(\param_{k^v(n)-1})\\
        \tilde{\param}_{n-1} &\leftarrow \frac{\rho_n}{L + \rho_n} \param_{n-1} + \frac{L}{L + \rho_n} \bar{\param}_{n-1}\\
        \param_n &\leftarrow \text{Proj}_{\Param} \left(\tilde{\param}_{n-1} - \frac{1}{L + \rho_n}\bar{\nabla}_{n-1}\right).
    \end{cases}
    \label{eq: prox_linear_update}
\end{equation}
Compared with MISO (obtained by setting $\rho_n = 0$ in \eqref{eq: prox_linear_update}) we see that the additional proximal regularization has the effect of further averaging the iterates, putting additional weight of $\frac{\rho_n}{L + \rho_n}$ on the most recent parameter $\param_{n-1}$.  

\subsection{Experiments}\label{sec: experiments}

\subsubsection{Distributed Nonnegative Matrix Factorization}\label{sec: nmf experiments}
In this section we compare the performance of the distributed matrix factorization version of RMISO from Sec. \ref{sec: nmf} against other well known optimization algorithms. We consider a randomly drawn collection of 5000 images from the MNIST \cite{mnist} dataset 
where each sample $X_v$ represents a subset of images. 
In all experiments, we set  $\alpha = \frac{1}{28}$ and $r = 15$. The dictionary $W$ is constrained to be non-negative 
and rows with euclidean norm at most one.

The set of vertices $\mathcal{V}$ is arranged in a cycle graph with $|\mathcal{V}| = 55$ with each vertex restricted to only contain samples with the same label. We consider two different sampling algorithms: the standard random walk where $t_{\odot}$ and $t_{\text{hit}}$ are both $O(|\mathcal{V}|^2)$, and cyclic where both are $O(|\mathcal{V}|)$. Our theory suggests we should expect better performance for cyclic sampling versus the random walk. We compare all three versions of RMISO: constant proximal regularization (RMISO-CPR), dynamic proximal regularization (RMISO-DPR), and diminishing radius (RMISO-DR), with MISO \cite{marial2015miso}, the online nonnegative matrix factorization (ONMF) algorithm of \cite{marialONMF}, and AdaGrad \cite{duchiAdaGrad}. 

It is not guaranteed that the surrogates $g_n^v(W) := \frac{1}{2} \lVert X_v - W_{n-1}H_n^v\rVert_{F}^2 + \alpha\lVert H_n^v \rVert_1$ are strongly convex. However, while running the experiments, we find that the Hessian of the averaged surrogate is positive definite after only a few iterations and thus the results for MISO are also supported by Theorem \ref{thm: convergence rates to stationarity} \ref{item: conv to stationarity- alg 2 iid}. This phenomenon is also discussed in Assumption B of \cite{marialONMF}.

\begin{figure}[h]
\centering
        \begin{subfigure}{0.235\textwidth}
        \includegraphics[width=\textwidth]{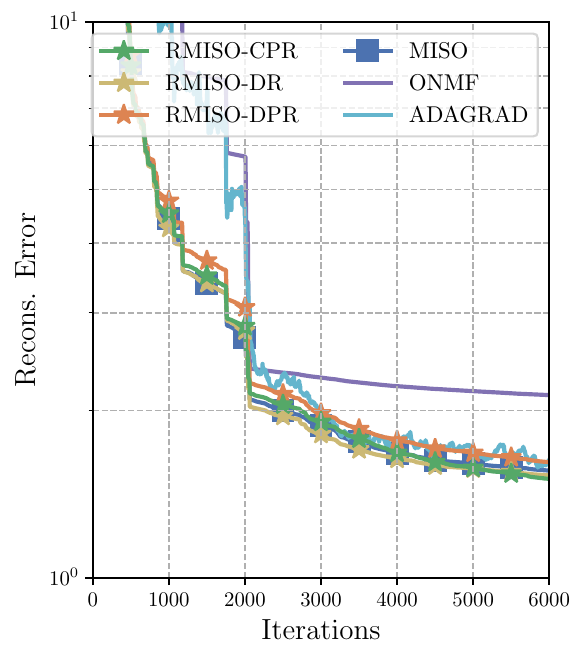}
        \caption{Random Walk}
        \end{subfigure}
        \begin{subfigure}{0.235\textwidth}
            \includegraphics[width=\textwidth]{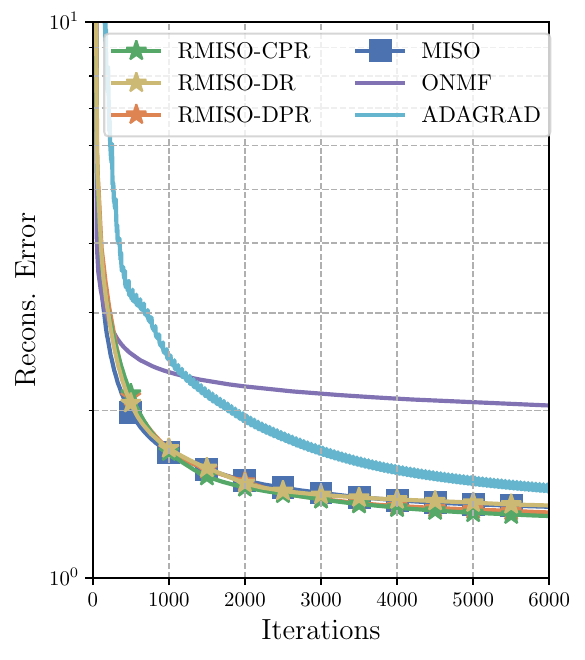}
            \caption{Cyclic}
        \end{subfigure}
        \vspace{-0.3cm}
        \caption{Plot of reconstruction error against interation number for NMF using two sampling algorithms. Results show the performance of algorithms RMISO, MISO (Algorithm \ref{RMISO} with $\rho_n = 0$), ONMF, and AdaGrad in factorizing a collection of MNIST \cite{mnist} data matrices. }
        \label{subfig: nmf}
\end{figure}


\vspace{-0.2cm}
We ran the experiment ten times with ten different random seeds and plot the average reconstruction error versus iteration number in Figure 
\ref{subfig: nmf}. We see that RMISO outperforms ONMF and shows competitive performance against AdaGrad for both sampling schemes. As expected, there is a dramatic performance improvement under cyclic sampling versus the random walk. 

\subsubsection{Logistic regression with nonconvex regularization}

We consider logistic regression with the non-convex regularization term $R(\theta) = 0.01 \cdot \sum_{i = 1}^{p} \frac{\theta_i^2}{1 + \theta_i^2}$ where $\param \in \R^p$ is the parameter to be optimized. We use the \texttt{a9a} dataset \cite{adult_2}. Here we consider the random walk on two separate graph topologies: the complete graph and the `lonely' graph as in Figure \ref{fig: lonely_node_graph}. Both graphs have $|\mathcal{V}| = 50$ and each vertex only stores data with the same label. 

\begin{figure}[h]
\centering
        \begin{subfigure}{0.235\textwidth}
        \includegraphics[width=\textwidth]{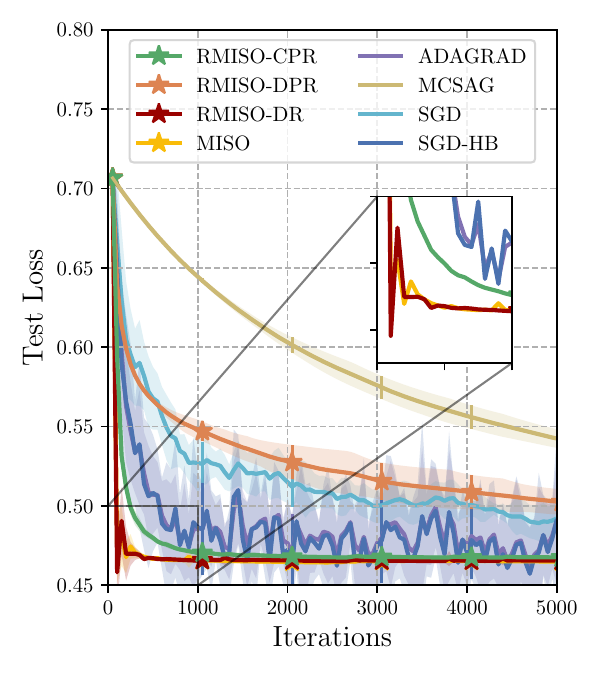}
        \caption{Lonely graph}
        \end{subfigure}
        \begin{subfigure}{0.235\textwidth}
            \includegraphics[width=\textwidth]{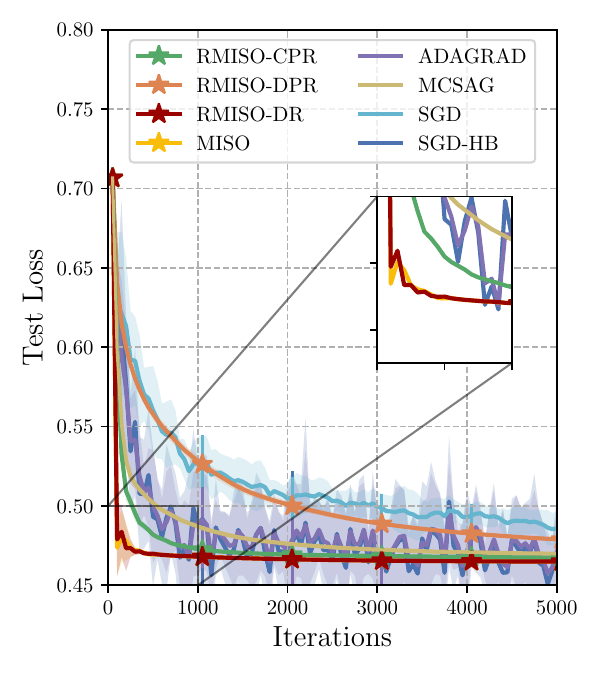}
            \caption{Complete graph}
        \end{subfigure}
        \vspace{-0.3cm}
    \caption{Plot of objective loss and standard deviation against the test dataset for \texttt{a9a} for two graph topologies and various optimization algorithms- RMISO, MISO (Algorithm \ref{RMISO} with $\rho_n =0$), AdaGrad, MCSAG, SGD, Adam, and SGD-HB}
    \label{fig: binary classification}
\end{figure}


 We compared eight different optimization algorithms: (1) the prox-linear version of Algorithm \ref{RMISO} \eqref{eq: prox_linear_update} with non-zero proximal regularization (RMISO-CPR); (2) Algorithm \ref{RMISO} with dynamic proximal regularization; (3) Algorithm \ref{RMISO_DR}; (4) MISO (Algorithm \ref{RMISO} with $\rho_n = 0$); (5) AdaGrad \cite{duchiAdaGrad}; (6) Markov Chain SAG (MCSAG) \cite{even2023stochastic}; (7) SGD with decaying step size; (8) SGD-HB (SGD with momentum).

 We ran each experiment with ten different seeds. The results plotted in Figure \ref{fig: binary classification} show the average loss against the test dataset for both graph topologies over these ten runs as well as a shaded region with boundaries given by the standard deviation. We see that RMISO-DPR and MCSAG display poorer performance on the lonely graph as $t_{\text{hit}}$ increases from $O(|\mathcal{V}|)$ to $O(|\mathcal{V}|^2)$. The performance of RMISO-CPR, RMISO-DR, and MISO are unchanged since each only depends on $t_{\odot}$, with RMISO-DR and MISO performing the best and only narrowly outperforming RMISO-CPR. 

Both SGD-HB and AdaGrad converge quickly in both settings but suffer from unstable trajectories compared to RMISO and MCSAG. A more stable algorithm may be advantageous in situations where the value of the objective function cannot easily be computed. See \cite{nesterovDoubleAverage} for an example of such a situation.  

 \section{Conclusion}

 In this paper we have established convergence and complexity results for our proposed extensions of MISO under the general assumption of recurrent data sampling. Our results show that convergence speed depends crucially on the average or supremized expected time to 
 reach other data points. In particular, the constant proximal regularization version of our algorithm depends only on the averaged target time, a potentially large improvement over the hitting time. Both our analysis and numerical experiments display the benefit of using possibly non-i.i.d or non-Markovian sampling schemes in order to accelerate convergence.

  \newpage

 \section*{Acknowledgements}

 This work was supported in part by NSF Award DMS-2023239 and DMS-2206296. The authors thank Qiaomin Xie for helpful comments.

 \section*{Impact Statement}

  This paper presents work whose goal is to advance the field of Machine Learning. There are many potential societal consequences of our work, none of which we feel must be specifically highlighted here.

\bibliography{refs}

\begin{thebibliography}{56}
\providecommand{\natexlab}[1]{#1}
\providecommand{\url}[1]{\texttt{#1}}
\expandafter\ifx\csname urlstyle\endcsname\relax
  \providecommand{\doi}[1]{doi: #1}\else
  \providecommand{\doi}{doi: \begingroup \urlstyle{rm}\Url}\fi

\bibitem[Alacaoglu \& Lyu(2023)Alacaoglu and Lyu]{AlacaogluNonconvex2023}
Alacaoglu, A. and Lyu, H.
\newblock Convergence of first-order methods for constrained nonconvex
  optimization with dependent data.
\newblock In \emph{Proceedings of the 40th International Conference on Machine
  Learning}, pp.\  458--489. PMLR, 2023.

\bibitem[Beck \& Teboulle(2009)Beck and Teboulle]{beck2009interative}
Beck, A. and Teboulle, M.
\newblock A fast iterative shrinkage-thresholding algorithm for linear inverse
  problems.
\newblock \emph{SIAM Journal on Imaging Sciences}, 2\penalty0 (1):\penalty0
  183--202, 2009.

\bibitem[Becker \& Kohavi(1996)Becker and Kohavi]{adult_2}
Becker, B. and Kohavi, R.
\newblock {Adult}.
\newblock UCI Machine Learning Repository, 1996.

\bibitem[Bertsekas(2011)]{bertsekas2011incremental}
Bertsekas, D.~P.
\newblock Incremental gradient, subgradient, and proximal methods for convex
  optimization: A survey.
\newblock \emph{Optimization for Machine Learning}, 2010\penalty0
  (1-38):\penalty0 3, 2011.

\bibitem[Beznosikov \& Tak{\'a}{\v c}(2023)Beznosikov and Tak{\'a}{\v
  c}]{Beznosikov2023RRSARAH}
Beznosikov, A. and Tak{\'a}{\v c}, M.
\newblock Random-reshuffled {SARAH} does not need full gradient computations.
\newblock \emph{Optimization Letters}, 2023.

\bibitem[Beznosikov et~al.(2023)Beznosikov, Samsonov, Sheshukova, Gasnikov,
  Naumov, and Moulines]{beznosikov2023firstordermarkovnoise}
Beznosikov, A., Samsonov, S., Sheshukova, M., Gasnikov, A., Naumov, A., and
  Moulines, E.
\newblock First order methods with {M}arkovian noise: from acceleration to
  variational inequalities.
\newblock \emph{arXiv preprint arXiv:2305:.15938}, 2023.

\bibitem[Bhandari et~al.(2018)Bhandari, Russo, and
  Singal]{bhandari2018TDLearning}
Bhandari, J., Russo, D., and Singal, R.
\newblock A finite time analysis of temporal difference learning with linear
  function approximation.
\newblock In \emph{Proceedings of the 31st Conference On Learning Theory}, pp.\
   1691--1692. PMLR, 2018.

\bibitem[Bottou(2012)]{Bottou2012sgdtricks}
Bottou, L.
\newblock \emph{Stochastic Gradient Descent Tricks}, pp.\  421--436.
\newblock Springer Berlin Heidelberg, Berlin, Heidelberg, 2012.

\bibitem[Bottou et~al.(2018)Bottou, Curtis, and
  Nocedal]{Bottou2018optimization}
Bottou, L., Curtis, F.~E., and Nocedal, J.
\newblock Optimization methods for large-scale machine learning.
\newblock \emph{SIAM Review}, 60\penalty0 (2):\penalty0 223--311, 2018.

\bibitem[Davis \& Drusvyatskiy(2019)Davis and
  Drusvyatskiy]{davis2019stochastic}
Davis, D. and Drusvyatskiy, D.
\newblock Stochastic model-based minimization of weakly convex functions.
\newblock \emph{SIAM Journal on Optimization}, 29\penalty0 (1):\penalty0
  207--239, 2019.

\bibitem[Deng(2012)]{mnist}
Deng, L.
\newblock The {MNIST} database of handwritten digit images for machine learning
  research.
\newblock \emph{IEEE Signal Processing Magazine}, 29\penalty0 (6):\penalty0
  141--142, 2012.

\bibitem[Duchi et~al.(2011)Duchi, Hazan, and Singer]{duchiAdaGrad}
Duchi, J., Hazan, E., and Singer, Y.
\newblock Adaptive subgradient methods for online learning and stochastic
  optimization.
\newblock \emph{Journal of Machine Learning Research}, 12\penalty0
  (61):\penalty0 2121--2159, 2011.

\bibitem[Even(2023)]{even2023stochastic}
Even, M.
\newblock Stochastic gradient descent under {M}arkovian sampling schemes.
\newblock In \emph{Proceedings of the 40th International Conference on Machine
  Learning}, pp.\  9412--9439. PMLR, 2023.

\bibitem[G{\"u}rb{\"u}zbalaban et~al.(2021)G{\"u}rb{\"u}zbalaban, Ozdaglar, and
  Parrilo]{Gurbuzbalaban2021reshufffling}
G{\"u}rb{\"u}zbalaban, M., Ozdaglar, A., and Parrilo, P.~A.
\newblock Why random reshuffling beats stochastic gradient descent.
\newblock \emph{Mathematical Programming}, 186\penalty0 (1):\penalty0 49--84,
  2021.

\bibitem[Horst \& Thoai(1999)Horst and Thoai]{horst1999dcprogramming}
Horst, R. and Thoai, N.~V.
\newblock Dc programming: Overview.
\newblock \emph{Journal of Optimization Theory and Applications}, 103\penalty0
  (1):\penalty0 1--43, 1999.

\bibitem[Horv{\'a}th et~al.(2022)Horv{\'a}th, Sanjabi, Xiao, Richt{\'a}rik, and
  Rabbat]{horvath2022fedshuffle}
Horv{\'a}th, S., Sanjabi, M., Xiao, L., Richt{\'a}rik, P., and Rabbat, M.
\newblock Fedshuffle: Recipes for better use of local work in federated
  learning.
\newblock \emph{Transactions on Machine Learning Research}, 2022.

\bibitem[Huang et~al.(2021)Huang, Yuan, Mao, and Yin]{huang2021improved}
Huang, X., Yuan, K., Mao, X., and Yin, W.
\newblock An improved analysis and rates for variance reduction under
  without-replacement sampling orders.
\newblock In \emph{Advances in Neural Information Processing Systems}. Curran
  Associates, Inc., 2021.

\bibitem[Huo et~al.(2023)Huo, Chen, and Xie]{Xie2023BiasExtrap}
Huo, D.~L., Chen, Y., and Xie, Q.
\newblock Bias and extrapolation in {M}arkovian linear stochastic approximation
  with constant stepsizes.
\newblock \emph{ACM SIGMETRICS Performance Evaluation Review}, 51\penalty0
  (1):\penalty0 81–82, 2023.

\bibitem[Johansson et~al.(2007)Johansson, Rabi, and
  Johansson]{Johansson2007ASP}
Johansson, B., Rabi, M., and Johansson, M.
\newblock A simple peer-to-peer algorithm for distributed optimization in
  sensor networks.
\newblock In \emph{46th IEEE Conference on Decision and Control}, pp.\
  4705--4710, 2007.

\bibitem[Johansson et~al.(2010)Johansson, Rabi, and
  Johansson]{JohanssonDistributed2010}
Johansson, B., Rabi, M., and Johansson, M.
\newblock A randomized incremental subgradient method for distributed
  optimization in networked systems.
\newblock \emph{SIAM Journal on Optimization}, 20\penalty0 (3):\penalty0
  1157--1170, 2010.

\bibitem[Johnson \& Zhang(2013)Johnson and Zhang]{Rie2013svrg}
Johnson, R. and Zhang, T.
\newblock Accelerating stochastic gradient descent using predictive variance
  reduction.
\newblock In \emph{Advances in Neural Information Processing Systems}. Curran
  Associates, Inc., 2013.

\bibitem[Karimi et~al.(2019)Karimi, Miasojedow, Moulines, and
  Wai]{karimi2019nonasymptotic}
Karimi, B., Miasojedow, B., Moulines, E., and Wai, H.-T.
\newblock Non-asymptotic analysis of biased stochastic approximation scheme.
\newblock In \emph{Proceedings of the Thirty-Second Conference on Learning
  Theory}, pp.\  1944--1974. PMLR, 2019.

\bibitem[Karimi et~al.(2022)Karimi, Wai, Moulines, and Li]{Karimi2022MISSO}
Karimi, B., Wai, H.-T., Moulines, E., and Li, P.
\newblock Minimization by incremental stochastic surrogate optimization for
  large scale nonconvex problems.
\newblock In Dasgupta, S. and Haghtalab, N. (eds.), \emph{Proceedings of The
  33rd International Conference on Algorithmic Learning Theory}, volume 167,
  pp.\  606--637. PMLR, 2022.

\bibitem[Levin \& Peres(2017)Levin and Peres]{levin2017markov}
Levin, D. and Peres, Y.
\newblock \emph{{M}arkov Chains and Mixing Times}.
\newblock MBK. American Mathematical Society, 2017.

\bibitem[Lopes \& Sayed(2007)Lopes and Sayed]{Lopes2007incremental}
Lopes, C.~G. and Sayed, A.~H.
\newblock Incremental adaptive strategies over distributed networks.
\newblock \emph{IEEE Transactions on Signal Processing}, 55\penalty0
  (8):\penalty0 4064--4077, 2007.

\bibitem[Lu et~al.(2022{\natexlab{a}})Lu, Guo, and De~Sa]{lu2022grab}
Lu, Y., Guo, W., and De~Sa, C.~M.
\newblock Grab: Finding provably better data permutations than random
  reshuffling.
\newblock In \emph{Advances in Neural Information Processing Systems}. Curran
  Associates, Inc., 2022{\natexlab{a}}.

\bibitem[Lu et~al.(2022{\natexlab{b}})Lu, Meng, and Sa]{lu2022exampleselection}
Lu, Y., Meng, S.~Y., and Sa, C.~D.
\newblock A general analysis of example-selection for stochastic gradient
  descent.
\newblock In \emph{International Conference on Learning Representations},
  2022{\natexlab{b}}.

\bibitem[Lyu(2023)]{lyu2023srmm}
Lyu, H.
\newblock Stochastic regularized majorization-minimization with weakly convex
  and multi-convex surrogates.
\newblock \emph{arXiv preprint arXiv:2201.01652}, 2023.

\bibitem[Lyu \& Li(2023)Lyu and Li]{lyu2021BCD}
Lyu, H. and Li, Y.
\newblock Block majorization-minimization with diminishing radius for
  constrained nonconvex optimization.
\newblock \emph{arXiv preprint arXiv:2012.03503}, 2023.

\bibitem[Lyu et~al.(2020)Lyu, Needell, and Balzano]{lyu2020online}
Lyu, H., Needell, D., and Balzano, L.
\newblock Online matrix factorization for markovian data and applications to
  network dictionary learning.
\newblock \emph{Journal of Machine Learning Research}, 21\penalty0
  (251):\penalty0 1--49, 2020.

\bibitem[Lyu et~al.(2022)Lyu, Strohmeier, and Needell]{lyu2022online}
Lyu, H., Strohmeier, C., and Needell, D.
\newblock Online nonnegative cp-dictionary learning for {M}arkovian data.
\newblock \emph{The Journal of Machine Learning Research}, 23\penalty0
  (1):\penalty0 6630--6679, 2022.

\bibitem[Mairal(2013)]{mairal2013stochastic}
Mairal, J.
\newblock Stochastic majorization-minimization algorithms for large-scale
  optimization.
\newblock In \emph{Advances in Neural Information Processing Systems}, pp.\
  2283--2291, 2013.

\bibitem[Mairal(2015)]{marial2015miso}
Mairal, J.
\newblock Incremental majorization-minimization optimization with application
  to large-scale machine learning.
\newblock \emph{SIAM Journal on Optimization}, 25\penalty0 (2):\penalty0
  829--855, 2015.

\bibitem[Mairal et~al.(2010)Mairal, Bach, Ponce, and Sapiro]{marialONMF}
Mairal, J., Bach, F., Ponce, J., and Sapiro, G.
\newblock Online learning for matrix factorization and sparse coding.
\newblock \emph{Journal of Machine Learning Research}, 11\penalty0
  (2):\penalty0 19--60, 2010.

\bibitem[Malinovsky et~al.(2023)Malinovsky, Sailanbayev, and
  Richt\'{a}rik]{malinovsky2023rrwithvariancereduction}
Malinovsky, G., Sailanbayev, A., and Richt\'{a}rik, P.
\newblock Random reshuffling with variance reduction: New analysis and better
  rates.
\newblock In \emph{Proceedings of the 39th Conference on Uncertainty in
  Artificial Intelligence}, pp.\  1347--1357. PMLR, 2023.

\bibitem[Mao et~al.(2020)Mao, Yuan, Hu, Gu, Sayed, and Yin]{Mao2020walkman}
Mao, X., Yuan, K., Hu, Y., Gu, Y., Sayed, A.~H., and Yin, W.
\newblock Walkman: A communication-efficient random-walk algorithm for
  decentralized optimization.
\newblock \emph{IEEE Transactions on Signal Processing}, 68:\penalty0
  2513--2528, 2020.

\bibitem[Mishchenko et~al.(2020)Mishchenko, Khaled, and
  Richtarik]{mishchenko2020reshuffling}
Mishchenko, K., Khaled, A., and Richtarik, P.
\newblock Random reshuffling: Simple analysis with vast improvements.
\newblock In \emph{Advances in Neural Information Processing Systems},
  volume~33, pp.\  17309--17320. Curran Associates, Inc., 2020.

\bibitem[Mishchenko et~al.(2022)Mishchenko, Khaled, and
  Richtarik]{mishchenko2022FedRR}
Mishchenko, K., Khaled, A., and Richtarik, P.
\newblock Proximal and federated random reshuffling.
\newblock In \emph{Proceedings of the 39th International Conference on Machine
  Learning}, pp.\  15718--15749. PMLR, 2022.

\bibitem[Mohtashami et~al.(2022)Mohtashami, Stich, and
  Jaggi]{Mohtashami2022characterizing}
Mohtashami, A., Stich, S.~U., and Jaggi, M.
\newblock Characterizing {\&} finding good data orderings for fast convergence
  of sequential gradient methods.
\newblock \emph{arXiv preprint arXiv:2202.01838}, 2022.

\bibitem[Nagaraj et~al.(2020)Nagaraj, Wu, Bresler, Jain, and
  Netrapalli]{nagaraj2020least}
Nagaraj, D., Wu, X., Bresler, G., Jain, P., and Netrapalli, P.
\newblock Least squares regression with {M}arkovian data: Fundamental limits
  and algorithms.
\newblock \emph{Advances in Neural Information Processing Systems},
  33:\penalty0 16666--16676, 2020.

\bibitem[Nesterov(2003)]{nesterov2003}
Nesterov, Y.
\newblock \emph{Introductory Lectures on Convex Optimization}.
\newblock Applied Optimization. Springer, 2003.

\bibitem[Nesterov(2013)]{nesterov2013composite}
Nesterov, Y.
\newblock Gradient methods for minimizing composite functions.
\newblock \emph{Mathematical Programming}, 140\penalty0 (1):\penalty0 125--161,
  2013.

\bibitem[Nesterov \& Shikhman(2015)Nesterov and
  Shikhman]{nesterovDoubleAverage}
Nesterov, Y. and Shikhman, V.
\newblock Quasi-monotone subgradient methods for nonsmooth convex minimization.
\newblock \emph{Journal of Optimization Theory and Applications}, 165\penalty0
  (3):\penalty0 917--940, 2015.

\bibitem[Parikh \& Boyd(2014)Parikh and Boyd]{parikh2014prox}
Parikh, N. and Boyd, S.
\newblock Proximal algorithms.
\newblock \emph{Foundations and Trends in Optimimization}, 1\penalty0
  (3):\penalty0 127--239, 2014.

\bibitem[Qian et~al.(2019)Qian, Sailanbayev, Mishchenko, and
  Richtárik]{qian2019miso}
Qian, X., Sailanbayev, A., Mishchenko, K., and Richtárik, P.
\newblock {MISO} is making a comeback with better proofs and rates.
\newblock \emph{arXiv preprint arXiv:1906.01474}, 2019.

\bibitem[Rajput et~al.(2022)Rajput, Lee, and
  Papailiopoulos]{rajput2022permutationbased}
Rajput, S., Lee, K., and Papailiopoulos, D.
\newblock Permutation-based {SGD}: Is random optimal?
\newblock In \emph{International Conference on Learning Representations}, 2022.

\bibitem[Ram et~al.(2009)Ram, Nedi\'{c}, and Veeravalli]{ram2009incremental}
Ram, S.~S., Nedi\'{c}, A., and Veeravalli, V.~V.
\newblock Incremental stochastic subgradient algorithms for convex
  optimization.
\newblock \emph{SIAM Journal on Optimization}, 20\penalty0 (2):\penalty0
  691--717, 2009.

\bibitem[Roy \& Balasubramanian(2023)Roy and Balasubramanian]{roy2023online}
Roy, A. and Balasubramanian, K.
\newblock Online covariance estimation for stochastic gradient descent under
  {M}arkovian sampling.
\newblock \emph{arXiv preprint arXiv:2308.01481}, 2023.

\bibitem[Schmidt et~al.(2017)Schmidt, Le~Roux, and Bach]{schmidtSAG2017}
Schmidt, M., Le~Roux, N., and Bach, F.
\newblock Minimizing finite sums with the stochastic average gradient.
\newblock \emph{Mathematical Programming}, 162\penalty0 (1):\penalty0 83--112,
  2017.

\bibitem[Sow et~al.(2024)Sow, Lin, Wang, and Liang]{sow2024doubly}
Sow, D., Lin, S., Wang, Z., and Liang, Y.
\newblock Doubly robust instance-reweighted adversarial training.
\newblock In \emph{The Twelfth International Conference on Learning
  Representations}, 2024.

\bibitem[Steininger et~al.(2021)Steininger, Kobs, Davidson, Krause, and
  Hotho]{steininger2021density}
Steininger, M., Kobs, K., Davidson, P., Krause, A., and Hotho, A.
\newblock Density-based weighting for imbalanced regression.
\newblock \emph{Machine Learning}, 110\penalty0 (8):\penalty0 2187--2211, 2021.

\bibitem[Sun et~al.(2018)Sun, Sun, and Yin]{SunOMCGD}
Sun, T., Sun, Y., and Yin, W.
\newblock On {M}arkov chain gradient descent.
\newblock In \emph{Advances in Neural Information Processing Systems},
  volume~31. Curran Associates, Inc., 2018.

\bibitem[Sun et~al.(2022)Sun, Li, and Wang]{sun2022Adaptive}
Sun, T., Li, D., and Wang, B.
\newblock Adaptive random walk gradient descent for decentralized optimization.
\newblock In \emph{Proceedings of the 39th International Conference on Machine
  Learning}, pp.\  20790--20809. PMLR, 2022.

\bibitem[Wang et~al.(2022{\natexlab{a}})Wang, Lei, Ying, and
  Zhou]{wang2022stability}
Wang, P., Lei, Y., Ying, Y., and Zhou, D.-X.
\newblock Stability and generalization for {M}arkov chain stochastic gradient
  methods.
\newblock In \emph{Advances in Neural Information Processing Systems}, pp.\
  37735--37748, 2022{\natexlab{a}}.

\bibitem[Wang et~al.(2022{\natexlab{b}})Wang, Xu, Liu, Li, Thuraisingham, and
  Tang]{wang2022imbanlanced}
Wang, W., Xu, H., Liu, X., Li, Y., Thuraisingham, B., and Tang, J.
\newblock Imbalanced adversarial training with reweighting.
\newblock In \emph{2022 IEEE International Conference on Data Mining (ICDM)},
  pp.\  1209--1214, Los Alamitos, CA, USA, dec 2022{\natexlab{b}}. IEEE
  Computer Society.

\bibitem[Ying et~al.(2017)Ying, Yuan, Vlaski, and Sayed]{ying2017performance}
Ying, B., Yuan, K., Vlaski, S., and Sayed, A.~H.
\newblock On the performance of random reshuffling in stochastic learning.
\newblock In \emph{2017 Information Theory and Applications Workshop}, pp.\
  1--5, 2017.

\end{thebibliography}
\bibliographystyle{icml2024}

\newpage
\appendix
\onecolumn

\section{Further remarks on main results}\label{sec: remarks on main results}

We include in this section an extended version of Theorem \ref{thm: convergence rates to stationarity} including convergence rates for arbitrary regularization parameters as well as two of its corollaries. The extended version of Theorem \ref{thm: convergence rates to stationarity} is below.

\begin{theorem}[Extended Version of Theorem \ref{thm: convergence rates to stationarity} in the main text]\label{thm: convergence rates to stationarity-extended}
Algorithms \ref{RMISO} and \ref{RMISO_DR} satisfy the following:
\begin{enumerate}[label=(\roman{*})]

\item \label{item: conv to stationarity- alg 2 iid-extended} Assume Case \ref{case: prox_regularization-iid}. Then 
\begin{align}
     \min_{1 \leq n \leq N} \E\left[ \sup_{\param \in \Param, \norm{\param - \param_n}\leq 1} -\nabla f(\param_n, \param - \param_n) \right]
     \leq \frac{\sqrt{2\Delta_0}\left( \frac{\rho}{\sqrt{\rho + \mu}} + \frac{Lt_{\odot}}{\sqrt{\rho + \mu}}\right)}{\sqrt{N}}. \label{eq: conv_to_stationarity_alg 2 obj iid-extended}
\end{align}
In particular, if $\rho$ is chosen so that $\rho \leq Lt_{\odot} \leq \rho + \mu$ then 
\begin{align}
     \min_{1 \leq n \leq N} \E\left[ \sup_{\param \in \Param, \norm{\param - \param_n}\leq 1} -\nabla f(\param_n, \param - \param_n) \right]
     \leq 2\sqrt{\frac{2\Delta_0Lt_{\odot}}{N}}.
\end{align}

\item \label{item: conv to stationarity- alg 2-extended} Assume Case \ref{case: prox_regularization-markovian}. Then 
\begin{align}
    \min_{1 \leq n \leq N} \E \left[\sup_{\param \in \Param, \norm{\param - \param_n}\leq 1} - \nabla f(\param_n, \param - \param_n) \right]
    \leq \frac{\sqrt{2\Delta_0} \left( \sqrt{\rho + (2 t_{\text{hit}} + 1)\log_2 (4|\mathcal{V}|)} + \frac{Lt_{\odot}}{\sqrt{\rho + \mu}} \right) }{\sqrt{N}}. \label{eq: conv_to_stationarity_alg 2 obj-extended}
\end{align} 
If $\rho$ satisfies the condition $\rho \leq Lt_{\odot} \leq \rho + \mu$ then
\begin{align}
\min_{1 \leq n \leq N} \E \left[ \sup_{\param \in \Param, \norm{\param - \param_n \leq 1}} - \nabla f(\param_n, \param - \param_n) \right]
\leq \frac{\sqrt{2\Delta_0}\left( \sqrt{Lt_{\odot} + (2 t_{\text{hit}} + 1)\log_2 (4|\mathcal{V}|)} + \sqrt{Lt_{\odot}} \right)}{\sqrt{N}}.
\end{align}

\item \label{item: conv to stationarity- alg 1-extended} Assume Case \ref{case: diminishing_radius-markovian}. Then
\begin{align}
    \min_{1 \leq n \leq N} \E \left[ \sup_{\param \in \Param, \norm{\param - \param_n} \leq 1} \langle - \nabla f(\param_n), \param - \param_n \rangle \right]
    \leq \frac{\Delta_0 + \sqrt{\frac{2L}{\pi_{\text{min}}}C_N\Delta_0} + \Big(3 + t_{\odot}\Big)C_NL}{\sum_{n = 1}^N 1 \wedge r_{n+1}}, \label{eq: conv_to_stationarity_alg 1 obj-extended}
\end{align}
where $C_N = \sum_{n = 1}^N r_n^2$.
\end{enumerate}
\end{theorem}

Next, Corollary \ref{cor: conv rates to stationarity-param in interior} specializes these results to the setting of unconstrained nonconvex optimization.

\begin{corollary}\label{cor: conv rates to stationarity-param in interior}
    Assume either $\Param = \R^p$ or that there exists $c \in (0, 1]$ so that $\text{dist}(\param_n, \partial \Param) \geq c$ for all $n \geq 1$. 
    \begin{enumerate}[label=(\roman{*})]

        \item \label{item: cor-conv to stationarity-interior-prox_markov} Let $(\param_n)_{n \geq 0}$ be an output of Algorithm \ref{RMISO}. Assume case \ref{case: prox_regularization-iid} or \ref{case: prox_regularization-markovian}. Then for any $N \geq 1$
        \begin{align}
             \min_{1 \leq n \leq N} \E\left[\norm{\nabla f(\param_n)}\right] = O \left(N^{-1/2} \right).
            \label{eq: cor-conv to stationarity-interior-PR}
        \end{align}
        
        \item \label{item: cor-conv to stationarity-iterior-DR} Let $(\param_n)_{n \geq 0}$ be an output of Algorithm \ref{RMISO_DR}. Assume case \ref{case: diminishing_radius-markovian}. Then for any $N \geq 1$
        \begin{align}
            \min_{1 \leq n \leq N} \E\left[\norm{\nabla f(\param_n)}\right] = O \left( \left( \sum_{n = 1}^N 1 \wedge r_n \right)^{-1} \right).
            \label{eq: cor-conv to stationarity-interior-DR}
        \end{align}

    \end{enumerate}
\end{corollary}

  Notice that we may take $r_n = \frac{1}{\sqrt{n} \log n}$ in Algorithm \ref{RMISO_DR}. Then Corollary \ref{cor: conv rates to stationarity-param in interior} implies 
  \begin{align}
         \min_{1 \leq n \leq N} \E\left[\norm{\nabla f(\param_n)}\right] = O \left( \frac{\log N}{\sqrt{N}} \right)
        \end{align}
  holds for case \ref{case: diminishing_radius-markovian}. 
  
 Finally, Corollary \ref{cor: iteration complexity} states the iteration complexity of Algorithms \ref{RMISO} and \ref{RMISO_DR}. 

\begin{corollary}[Iteration Complexity]
\label{cor: iteration complexity}
Algorithms \ref{RMISO} and \ref{RMISO_DR} have the following worst case iteration complexity:
\begin{enumerate}[label=(\roman{*})]
    \item \label{item: pr iteration complexity} Let $(\param_n)_{n \geq 0}$ be an output of Algorithm \ref{RMISO}. Assume case \ref{case: prox_regularization-iid} or \ref{case: prox_regularization-markovian}. Then $N_{\eps}(\param_0) = O(\eps^{-2})$.
    \item \label{item: dr iteration complexity} Let $(\param_n)_{n \geq 0}$ be an output of Algorithm \ref{RMISO_DR}. Assume case \ref{case: diminishing_radius-markovian}. Then $N_{\eps}(\param_0) = O(\eps^{-2}(\log \eps^{-1})^2)$.
\end{enumerate}
\end{corollary}

\begin{remark}[Comparison with the lower bound of \cite{even2023stochastic} Theorem 1]
Using our notation, the lower bound given in Theorem 1 of \cite{even2023stochastic} is
\begin{align}
    \norm{\nabla f(\param_N)}^2 = \Omega \left(L \Delta_0 \left(\frac{t_{\text{hit}}}{N}\right)^2\right).
\end{align}
Our convergence rates are given in terms of $\norm{\nabla f(\theta_n)}$ rather than $\norm{\nabla f(\theta_n)}^2$, so in our setting this is 
\begin{align}
    \norm{\nabla f(\param_N)} = \Omega \left(\sqrt{L \Delta_0} \left(\frac{t_{\text{hit}}}{N}\right)\right).
\end{align}
Notice that the rate of convergence in the lower bound is $O(N^{-1})$ while our upper bound gives a rate of convergence of $O(N^{-1/2})$. Thus, despite the dependence in our results on $t_{\odot}$, they do not contradict this lower bound.

\end{remark}

\begin{remark}[Optimal sampling and estimating $t_{\odot}$ and $t_{\text{hit}}$]
The dependence of the convergence rates on $t_{\odot}$ or $t_{\text{hit}}$ in Theorem \ref{thm: convergence rates to stationarity-extended} suggests one can accelerate convergence by choosing a sampling algorithm with the smallest values of these constants appropriate for the context. In general, an optimal sampling scheme is problem dependent. The best one can hope for, in terms of dependence on $|\mathcal{V}|$, is that both constants are $O(|\mathcal{V}|)$ which is achieved by i.i.d sampling. However, this may not be feasible in settings like decentralized optimization where communication can only occur between neighboring vertices in a graph.

Depending on the graph topology, it is likely that for the standard random walk $t_{\text{hit}}$ and $t_{\odot}$ are much larger than $|\mathcal{V}|$, especially for sparse graphs. If a cycle containing all nodes in the graph exists, our theory suggests using cyclic sampling by traversing such a spanning cycle. In this case, both $t_{\text{hit}}$ and $t_{\odot}$ have optimal order $O(|\mathcal{V}|)$. If no such cycle exists, a good way to minimized $t_{\text{hit}}$ is to find the shortest path in the graph which contains all vertices and then sample the vertices deterministically in order by walking over this path. This idea holds in a more general setting beyond optimization on graphs: a good way to minimize $t_{\text{hit}}$ is to sample data as efficiently as possible by covering the dataset with the fewest possible number of repeats.

For many specific instances, these quantities can be estimated analytically. For random walks on graphs, much about the hitting time and target time is known through classical Markov chain theory \cite{levin2017markov}. For cyclic sampling and random reshuffling respectively, one has $t_{\text{hit}} = |\mathcal{V}|$ and $t_{\text{hit}} \leq 2|\mathcal{V}|$ since each data point is visited exactly once every epoch and no re-shuffling occurs in the cyclic case. Under cyclic sampling, for fixed $n$ and for each $1 \leq k \leq |\mathcal{V}|$ there is a $v$ with $\mathbb{E}[\tau_{n, v} | \cF_n] = k$. So $t_{\odot}$ is the largest possible value of $\sum_{v \in \mathcal{V}} \sigma_v\pi(v)$ where $\sigma$ ranges over all permutations of ${1, ..., |\mathcal{V}|}.$ In particular, if $\pi$ is uniform, $t_{\odot} = \frac{|\mathcal{V}|-1}{2}$. For reshuffling with uniform $\pi$, $t_{\odot}$ is at least this large by considering a time $n$ at the beginning of an epoch. But it still holds $t_{\odot} \leq 2|\mathcal{V}|$ since $t_{\odot} \leq t_{\text{hit}}$. If these quantities cannot be easily estimated analytically, they can be approximated using Monte-Carlo. 

\end{remark}

\begin{remark}[Comparison with i.i.d sampling]
\normalfont
 If the sequence $(v_n)$ is formed by sampling vertices uniformly at random from $\mathcal{V}$ then, as previously mentioned, the return times $\tau_{n, v}$ are i.i.d geometric random variables with parameter $\frac{1}{|\mathcal{V}|}$. Then $\E[\tau_{n, v}| \cF_n] = |\mathcal{V}|$ for each $n$ and $v$ so $t_{\odot} = |\mathcal{V}|$. Substituting $|\mathcal{V}|$ for $t_{\odot}$ in the optimal bound in Theorem \ref{thm: convergence rates to stationarity} \ref{item: conv to stationarity- alg 2 iid} we recover the result given for MISO in the i.i.d setting in \cite{Karimi2022MISSO} up to a factor of two. This is in-spite of the fact that our analytical approach is necessarily different to handle general recurrent sampling and shows that our results are tight.
\end{remark}

\begin{remark}[Iterate stability and regularization]\label{rem: iterate stability and regularization}
\normalfont
Here we give some remarks on the use of diminishing radius and proximal regularization in Algorithms \ref{RMISO} and \ref{RMISO_DR}.

The diminishing radius restriction in Algorithm \ref{RMISO_DR} is a `hard' regularization technique. It bakes necessary iterate stability directly into the problem by enforcing the stronger condition $\norm{\param_n - \param_{n-1}} \leq r_n$. In comparison with proximal regularization, diminishing radius bounds the one step iterate difference by a deterministic quantity. Moreover, as is argued in the proof of Theorem \ref{thm: a.s. convergence} \ref{item: stationary_limit_points} and the preceding propositions, sufficiently often the iterate $\param_n$ obtained by minimizing the $\bar{g}_n$ over the trust region in fact minimizes $\bar{g}_n$ over the entire feasible set $\Param$. This allows us to prove that limit points of the iterates produces by Algorithm \ref{RMISO_DR} are stationary even for  general recurrent sampling schemes. However, diminishing radius introduces the drawback of needing to compute a projection at each step of the optimization process. 

Compared to diminishing radius, proximal regularization is a form of `soft' regularization. It is less restrictive than diminishing radius, but only allows us to derive a weaker form of iterate stability: we only have $\sum_{n = 1}^{\infty} \norm{\param_n - \param_{n-1}}^2 < \infty$ instead of the stronger $\norm{\param_n - \param_{n-1}} \leq r_n$ which makes some aspects of the analysis slightly more challenging. 

As mentioned in Section \ref{sec: assumptions}, the use of dynamic proximal regularization in Case \ref{case: prox_regularization-markovian} adapts to the sampling process and increasingly penalizes large values of $\norm{\param_n - \param_{n-1}}$ as the amount of time any vertex is left un-visited increases. The drawback is that we are unable to prove almost sure asymptotic convergence in Theorem \ref{thm: a.s. convergence} for Case \ref{case: prox_regularization-markovian} as we are for Case \ref{case: diminishing_radius-markovian}. If we could show $\rho_n \norm{\param_n - \param_{n-1}} \to 0$, then asymptotic convergence would follow. However $\rho_n$ is not bounded and can take arbitrarily large values, albeit with low probability. But as we show in Lemma \ref{lem: uniform bound of reg parameter}, $\E[\rho_n]$ is uniformly bounded, 
which allows us to deduce $\E[\rho_n \norm{\param_n - \param_{n-1}}] \to 0$ and leads to the $L^{1}$-convergence result in Theorem \ref{thm: a.s. convergence}. 

For case \ref{case: prox_regularization-iid} it is relatively straightforward to show $\rho_n \norm{\param_n - \param_{n-1}} \to 0$ since $\rho_n$ is constant. In this case, the difficulty lies in showing $\norm{\nabla \bar{h}_n(\param_n)} \to 0$.  We \emph{are} able to prove this for Case \ref{case: prox_regularization-markovian} however because the use of dynamic proximal regularization allows us to show that the sequence $\max_{v \in \mathcal{V}}(n - k^v(n))\norm{\param - \param_{n - 1}}^2$ is summable. More detail is given in the proofs and discussion in subsection \ref{sec: dpr a.s. convergence}.
    
\end{remark}



\section{Convergence analysis using mixing times} \label{sec: mixing time analysis}
In this section, we give an overview of the standard pipeline for analyzing stochastic optimization algorithms with Markovian data and discuss the difficulties of applying these techniques to the analysis of MISO. 

The analysis of first order methods such as SGD generally rely on conditionally unbiased gradient estimates. In the context of solving problem \eqref{eqn:main_problem}, this is to require 
\begin{align}
    \E[\nabla f^{v_n}(\param_{n-1})|\cF_{n-1}] = \nabla f(\param_{n-1})
    \label{eq: unbiased gradient}
\end{align}
where $\cF_n$ is the filtration of information up to time $n$. 
 However, in the dependent data setting \eqref{eq: unbiased gradient} does not hold which complicates the analysis significantly. Previous works (e.g \cite{SunOMCGD, bhandari2018TDLearning, nagaraj2020least, lyu2020online, lyu2022online, lyu2023srmm, AlacaogluNonconvex2023}) use a "conditioning on the distant past" argument. Specifically, for SGD, one assumes that $(v_n)$ is a Markov chain on state space $\mathcal{V}$ which mixes exponentially fast to its stationary distribution $\pi$ with parameter $\lambda$. One then considers the quantity
 \begin{align}
 \E[\nabla f^{v_n}(\param_{n-a_n})|\cF_{n-a_n}]
 \end{align}
 where $a_n$ is a slowly growing sequence satisfying $\sum_{n \geq 1} \lambda^{a_n} < \infty$. 
 By further assuming either uniformly bounded gradients as in \cite{SunOMCGD} or that the conditional expectations $\E[\norm{\nabla f^{v_{n+1}}(\param)}|\cF_n]$ are uniformly bounded as in \cite{AlacaogluNonconvex2023}, one can show that
 \begin{align}
     \left \lVert \nabla f(\param_{n- a_n}) - \E[\nabla f^{v_n}(\param_{n - a_n}) |\cF_{n - a_n}] \right \rVert = O(\lambda^{a_n}).
\label{eq: sgd_mixing_bound}
 \end{align}
 Using Lipschitz continuity of gradients, it is then established that 
 \begin{align}
     \norm{\nabla f(\param_n) - \E[\nabla f^{v_n}(\param_n)|\cF_{n-a_n}]} = O(\lambda^{a_n}) + O(\norm{\param_n - \param_{n-a_n}}),
\label{eq: sgd_mixing_lipschitz_bound}
 \end{align}
 allowing one to control the bias in the stochastic gradient estimate. Conventional analysis may then be used to prove 
 \begin{align}
     \sum_{n = 1}^{\infty} \gamma_n \E \left[ \langle \nabla f(\param_n), \nabla f^{v_n}(\param_n) \rangle \right] < \infty,
 \end{align}
 where $\gamma_{n}$ denotes the stepsize at iteration $n$. Combining this with \eqref{eq: sgd_mixing_lipschitz_bound}, it can finally be deduced that 
 \begin{align}    
 \sum_{n = 1}^{\infty} \gamma_n \E[\norm{\nabla f(\param_n)}^2] < \infty
 \end{align}
 for an appropriate stepsize $\gamma_n$ which gives the desired convergence rate. 

 This technique was further developed to analyze convergence rates of \emph{stochastic majorization minimization} (SMM) \cite{mairal2013stochastic} algorithms under Markovian sampling in \cite{lyu2023srmm}. In the context of problem \eqref{eqn:main_problem}, SMM proceeds by minimizing a recursively defined majorizing surrogate of the empirical loss function. The algorithm is stated concisely as follows:
 \begin{align}\label{eq:SMM}
 (\textbf{SMM}):
 \begin{cases}
      &\text{Sample } v_n \text{ from the conditional distribution } \pi(\cdot| \cF_{n-1})\\
     &g_n \leftarrow \text{Strongly convex majorizing surrogate of $f^{v_n}$}\\
     &\param_n \in \argmin_{\param \in \Param}\left( \bar{g}_n(\param) := (1 - w_n)\bar{g}_{n-1}(\param) + w_ng_n(\param) \right)
\end{cases}
 \end{align}
 where $(w_n)_{n \geq 1}$ is a non-increasing sequence of weights. A crucial step in the analysis of SMM is to show that 
 \begin{align}
     \sum_{n = 1}^{\infty} w_n \E[|\bar{g}_n(\param_n) - \bar{f}_n(\param_n)|] < \infty
\label{eq: srmm gap sum}
 \end{align}
where $\bar{f}_n$ is the empirical loss function satisfying the recursion $\bar{f}_n(\param) := (1 - w_n) \bar{f}_{n-1}(\param) + w_n f^{v_n}(\param)$. To do so, the problem is reduced to showing that 
\begin{align}
    \E \left[ \E[f^{v_n}(\param_{n - a_n}) - \bar{f}_n(\param_{n - a_n}) | \cF_{n - a_n}]^{+} \right] = O(w_{n - a_n}) + O(\lambda^{a_n})
\label{eq: srmm prop 8.1}
\end{align}
(see Proposition 8.1 in \cite{lyu2023srmm}) which is similar to \eqref{eq: sgd_mixing_bound}. By additionally assuming Lipschitz continuity of the individual loss functions $f^v$, it is then shown that
\begin{align}
\E\left[\E[f^{v_n}(\param_n)  - \bar{f}_n(\param_n)|\cF_{n-a_n}] \right] = O(w_{n - a_n}) + O(\lambda^{a_n}) + O(\norm{\param_n - \param_{n - a_n}})
\label{eq: srmm_mixing_lipschitz}
\end{align}
which may be compared with \eqref{eq: sgd_mixing_lipschitz_bound}. Finally, \eqref{eq: srmm gap sum} is proved by showing
\begin{align}
    \sum_{n = 1}^{\infty} w_n\E[|\bar{g}_n(\param_n) - \bar{f}_n(\param_n)|] \leq \bar{g}_1(\param_1) + \sum_{n = 1}^{\infty}w_{n} \E[ f^{v_n}(\param_n) - \bar{f}_n(\param_n)]
\end{align}
and using the bound \eqref{eq: srmm_mixing_lipschitz} to conclude that the latter sum is finite.

MISO and its extensions proposed in this paper are similar to SMM in their use of the majorization-minimization principle, but contain a few key differences. Put shortly, the original implementation of MISO is
\begin{align}\label{eq: miso sketch}
    \textbf{(MISO)}: 
    \begin{cases}
        \text{Sample } v_n \text{ from the conditional distribution } \pi(\cdot| \cF_{n-1})\\
        g_n^{v_n} \leftarrow \text{Convex surrogate of $f^{v_n}$ at $\param_{n-1}$; $g_n^v = g_{n-1}^v$ for $v \neq v_n$}\\
        \param_n \in \argmin_{\param \in \Param} \left( \bar{g}_n(\param) := \sum_{v \in \mathcal{V}} g_n^v \pi(v) \right)
    \end{cases}
\end{align}
In contrast with SMM, each surrogate defining $\bar{g}_n$ is given a constant weight. Consequently, the additional control provided by the decreasing weights $(w_n)$ in SMM is not present, which makes the adaptation of the techniques used in the analysis of SMM non-trivial. 

A key step in the original analysis of MISO given in \cite{marial2015miso} is to show
\begin{align}
\sum_{n = 1}^{\infty}  \E[\bar{h}_n(\param_n)] < \infty.
\label{eq: miso_intro gap sum}
\end{align}
 This is similar to \eqref{eq: srmm gap sum} and shows that the averaged surrogate is an asymptotically accurate approximation of the true objective at all $\param_n$s. To prove this, one needs to relate the averaged error $\bar{h}_n(\param_n)$ to another quantity proven to be summable through other means. This is, in abstract, the role that mixing rate analysis plays for SGD and SMM. Using techniques from \cite{marial2015miso} one can prove 
\begin{align}
    \sum_{n = 1}^{\infty} \E[h_n^{v_{n+1}}(\param_n)] < \infty.
\end{align}
If the $(v_n)$ are drawn i.i.d from $\pi$, or more generally if the probability of transitioning between any two vertices is uniformly bounded below by a positive quantity, then $h_n^{v_{n+1}}(\param_n)$ is a conditionally unbiased estimate of $\bar{h}_n(\param_n)$ up to a constant. Then \eqref{eq: miso_intro gap sum} follows by conditioning on the most recent information $\cF_n$. 

To adapt the analysis to a more general setting, one may attempt to use Markov chain mixing to show 
\begin{align}
    |\bar{h}_n(\param_{n - a_n}) - \E[h_n^{v_{n+1}}(\param_{n-a_n}) |\cF_{n-a_n}]| = O(\lambda^{a_n})
\end{align}
 similar to \eqref{eq: sgd_mixing_bound} and \eqref{eq: srmm prop 8.1}. However, an additional complication arises because the function $h_n^{v_{n+1}}$ conditional on $\mathcal{F}_{n-a_{n}}$ depends on both the history of data samples $v_{k}$ \textit{and} the estimated parameters $\param_{k}$ for $n-a_{n}\le k \le n$. In contrast,  $f^{v_{n}}$ for SGD (in \eqref{eq: sgd_mixing_bound}) depends only on the last data sample $v_{n}$ and $\bar{f}_{n}$ for SMM (in \eqref{eq: srmm prop 8.1}) depends  only on the histrory of data samples $v_{n-a_{n}},\dots,v_{n}$. 
 So, one cannot use the Markov property to isolate the randomness in $h_n^{v_{n+1}}(\param_{n - a_n})$ due to the Markov chain transition over the interval $[n - a_n, n+1]$. To alleviate this problem, one may attempt to control the difference
 \begin{align}
     |h_n^{v_{n+1}}(\param_{n - a_n}) - h_{n - a_n}^{v_{n+1}}(\param_{n - a_n})|
 \end{align}
 but doing so is not straightforward without access to something resembling the weights $(w_n)$ in SMM.

\section{Preliminary Lemmas}\label{sec: preliminary lemmas}

In this section we state and prove some preliminary lemmas which will be used to prove in the proofs of both Theorem \ref{thm: convergence rates to stationarity} and \ref{thm: a.s. convergence}.

We first introduce some additional notation. Throughout this section as well as the remainder of the paper we let
\begin{align}
t_{\text{cov}} := \sup_{n \geq 1} \left \lVert \E\left[\max_{v \in \mathcal{V}} \tau_{n, v} \Big|\cF_n\right] \right \rVert_{\infty}.
\label{eq: tcov}
\end{align}
We recall that the $L_{\infty}$ norm here is taken for the conditional expectation viewed as a random variable. This quantity is a generalization of the worst case expected cover time from Markov chain theory \cite{levin2017markov} and will be important in the analysis of Algorithm \ref{RMISO} with dynamic proximal regularization. 
 
Our first result, Proposition \ref{prop: return time finite exponential moments and cover time bound}, states that under Assumption \ref{assumption: process properties} the return times have finite moments of all orders and gives an upper-bound on $t_{\text{cov}}$ in terms of $t_{\text{hit}}$. The first item will is used in Section \ref{sec: asymptotic analysis} to prove asymptotic results while the second is used in the proof of iteration complexity for Case \ref{case: prox_regularization-markovian}. 

\begin{prop}[Recurrence implies finite exponential moments of return time]\label{prop: return time finite exponential moments and cover time bound}
    Let Assumption \ref{assumption: process properties} hold and let $t_{\text{hit}}$ and $t_{\text{cov}}$ be as in \eqref{eq: thit} and \eqref{eq: tcov} respectively. 
    \begin{enumerate}[label=(\roman{*})]
        \item\label{item: return time exp moments} There exists $s_0 > 0$ so that for all $0 < s < s_0$ there is a constant $C_s > 0$ with 
        \begin{align}
            \max_{v \in \mathcal{V}}\, \sup_{n \geq 1} \,\, \E\left[e^{s\tau_{n, v}} \Big|\cF_n\right] \leq C_s < \infty.
        \end{align}
         Consequently, for each $p \geq 1$, $\max_{v \in \mathcal{V}} \sup_{n \geq 1}\E[\tau_{n, v}^p |\cF_n] \leq C < \infty$ for some $C > 0$.  
        \item\label{item: tcov bound} We have the following bound on $t_{\text{cov}}$:
        \begin{align}
            t_{\text{cov}} \leq (2 t_{\text{hit}} + 1)\log_2 (4|\mathcal{V}|).
        \end{align}
    \end{enumerate}
\end{prop}

\begin{proof}
    Let $m$ be the smallest integer satisfying $m \geq 2 t_{\text{hit}}$. For any $n \geq 1$ and $v \in \mathcal{V}$ notice that if $\tau_{n, v} \geq km$ then we must have $\tau_{n + jm, v} \geq m$ for each $0 \leq j \leq k - 1$. Then
    \begin{align}
        \P (\tau_{n, v} \geq km |\cF_n) \leq \P \left( \bigcap_{j = 0}^{k-1} \{\tau_{n + jm, v} \geq m\} \Big|\cF_n \right) = \E \left[ \prod_{j = 0}^{k-1} \1(\tau_{n + jm, v} \geq m) \Big|\cF_n \right].
    \end{align}
    We have 
    \begin{align}
        \E \left[ \prod_{j = 0}^{k-1} \1(\tau_{n + jm, v} \geq m) \Big|\cF_n \right] = \E \left[\E\big[\1(\tau_{n + (k-1)m, v} \geq m)|\cF_{n + (k-1)m}\big] \prod_{j = 0}^{k-2}\1(\tau_{n + jm, v} \geq m) \Big| \cF_n\right]
    \end{align}
    where we have used that $\{\tau_{n + jm, v} \geq m\}$ is measurable with respect to $\cF_{n + (k - 1)m}$ for each $j \leq k - 2$. By Markov's inequality and Assumption \ref{assumption: process properties}
    \begin{align}
        \E\big[\1(\tau_{n + (k-1)m, v} \geq m)|\cF_{n + (k-1)m}\big] = \P\left(\tau_{n + (k-1)m, v} \geq m |\cF_{n + (k-1)m}\right) \leq \frac{t_{\text{hit}}}{m}.
    \end{align}
    So 
    \begin{align}
        \E \left[ \prod_{j = 0}^{k-1} \1(\tau_{n + jm, v} \geq m) \Big|\cF_n \right] \leq \frac{t_{\text{hit}}}{m} \E \left[ \prod_{j = 0}^{k-2} \1(\tau_{n + jm, v} \geq m) \Big |\cF_n \right].
    \end{align}
    Proceeding by induction it follows that
    \begin{align}
        \P(\tau_{n, v} \geq km |\cF_n) \leq \left( \frac{t_{\text{hit}}}{m} \right)^k \leq 2^{-k}
    \end{align}
    with the second inequality using our choice of $m$. Now,
    \begin{align}
        \E\left[e^{s \tau_{n, v}}|\cF_n\right] = \sum_{\ell = 1}^{\infty} e^{s \ell}\P(\tau_{n, v} = \ell|\cF_n) \leq \sum_{k = 0}^{\infty} e^{s(k + 1)m} \P(\tau_{n, v} \geq km |\cF_n) \leq \sum_{k = 0}^{\infty} e^{s(k + 1)m} 2^{-k}.
    \end{align}
    The latter sum is finite if $s < \frac{\log 2}{2m}$ and does not depend on $n$ or $v$ which shows \ref{item: return time exp moments}. 

    With $n$ still fixed let $\tau_{\text{cov}} = \max_{v \in \mathcal{V}} \tau_{n, v}$. We have
    \begin{align}
        \E[\tau_{\text{cov}}|\cF_n] = \sum_{\ell = 1}^{\infty} \P(\tau_{\text{cov}} \geq \ell|\cF_n) \leq \sum_{k = 0}^{\infty} m\P(\tau_{\text{cov}} \geq km |\cF_{n}) \leq (2t_{\text{hit}} + 1)\sum_{k = 0}^{\infty} \P(\tau_{\text{cov}} \geq km |\cF_n)
        \label{eq: t_cov expecation bound}
    \end{align}
    since $\P(\tau_{\text{cov}} \geq \ell |\cF_n)$ is a decreasing function of $\ell$ and $m \leq 2t_{\text{hit}} + 1$. By a union bound
    \begin{align}
        \P(\tau_{\text{cov}} \geq km |\cF_n) \leq 1 \wedge \sum_{v \in \mathcal{V}} \P(\tau_{n, v} \geq km |\cF_n) \leq 1 \wedge |\mathcal{V}|2^{-k}.
    \end{align}
    Summing a geometric series we get
    \begin{align}
        \sum_{k = 0}^{\infty} \P(\tau_{\text{cov}} \geq km |\cF_n) \leq \log_2|\mathcal{V}| + |\mathcal{V}| \sum_{k > \log_2\mathcal{|V|}} 2^{-k} \leq \log_2|\mathcal{V}| + 2.
        \label{eq: t_cov sum bound}
    \end{align}
    Combining \eqref{eq: t_cov expecation bound} and \eqref{eq: t_cov sum bound} shows \ref{item: tcov bound}. 
\end{proof}

The next proposition states some general properties of first order surrogate functions.

\begin{prop}[Properties of Surrogates]
\label{prop: surrogate properties}
Fix $\bar{\param} \in \Param$ and $f : \Param \to \R$. Let $g \in \mathcal{S}_{L}(f, \bar{\param})$ and let $\param'$ be a minimizer of $g$ over $\Param$. Then for all $\param \in \Param$
\begin{align}
    |h(\param)| \leq \frac{L}{2}\norm{\param - \bar{\param}}^2
\end{align}
\end{prop}

\begin{proof}
    This follows from using the classical upperbound for $L$-smooth functions
    \begin{align}
        h(\param) \leq h(\bar{\param}) + \langle \nabla h(\bar{\param}), \param - \bar{\param} \rangle + \frac{L}{2}\norm{\param - \bar{\param}}^2
    \end{align}
    (see Lemma \ref{lem: L-smooth_surrogate}) and noting that $h(\bar{\param})$ and $\nabla h(\bar{\param})$ are both equal to zero according to Definition \ref{def: surrogates}. 
\end{proof}

Next, we show that the surrogate objective  value $\bar{g}_{n}(\param_{n})$ evaluated at $\param_{n}$ is non-increasing. 

\begin{lemma}[Surrogate Monotonicity]\label{lem: surrogate monotonicity}
Let $(\param_n)_{n \geq 0}$ be an output of either Algorithm \ref{RMISO} or \ref{RMISO_DR} . Then $\bar{g}_n(\param_{n-1}) \leq \bar{g}_{n-1}(\param_{n-1})$ for all $n \geq 1$. Moreover, the sequence $(\bar{g}_n(\param_n))_{n \geq 0}$ is non-increasing. As a consequence, by Assumption \ref{assumption: directional derivatives} and Definition \ref{def: surrogates}, $(\bar{g}_n(\param_n))_{n \geq 0}$ is bounded below with probability one and therefore $\lim_{n \to \infty} \bar{g}_n(\param_n)$ exists almost surely.
\end{lemma}

\begin{proof}
     Since $g_n^{v_n} \in \mathcal{S}_{L}(f^{v_n}, \param_{n-1})$, Definition \ref{def: surrogates} implies that $g_n^{v_n}(\param_{n-1}) = f^{v_n}(\param_{n-1})$. Then
     \begin{align}
         \bar{g}_n(\param_{n-1}) &= \bar{g}_{n-1}(\param_{n-1}) + \Big[g_n^{v_n}(\param_{n-1}) - g_{n-1}^{v_n}(\param_{n-1})\Big]\pi(v_n)\\
         &= \bar{g}_{n-1}(\param_{n-1}) + \Big[f^{v_n}(\param_{n-1}) - g_{n-1}^{v_n}(\param_{n-1}) \Big]\pi(v_n)\\
         &\leq \bar{g}_{n-1}(\param_{n-1})
     \end{align}
     where the last inequality used $g_{n-1}^{v_n}$ is a majorizing surrogate of $f^{v_n}$.
     
     Suppose now that $(\param_n)_{n \geq 0}$ is an output of Algorithm \ref{RMISO}. Then by definition of $\param_n$,
    \begin{align*}
        \bar{g}_n(\param_n) &\leq \bar{g}_n(\param_n) + \frac{\rho_n}{2} \norm{\param_n - \param_{n-1}}^2\\
        & \leq \bar{g}_n(\param_{n-1}) + \frac{\rho_n}{2}\norm{\param_{n-1} - \param_{n-1}}^2\\
        &= \bar{g}_n(\param_{n-1})\\
        & \leq \bar{g}_{n-1}(\param_{n-1}).
    \end{align*}
  If instead $(\param_n)_{n \geq 1}$ is an output of Algorithm \ref{RMISO_DR}, then we can directly conclude $\bar{g}_n(\param_n) \leq \bar{g}_n(\param_{n-1})$ by definition of $\param_n$. The remainder of the proof is identical to the above. 
\end{proof}

The next lemma establishes the summability of the sequence $h_n^{v_{n+1}}(\param_{n})$. This was used in \cite{marial2015miso} to prove asymptotic convergence of MISO under i.i.d. sampling. We use it primarily in the analysis of Algorithm \ref{RMISO_DR}.

\begin{lemma}\label{lem:single gap convergence}
Let $(\param_n)_{n \geq 0}$ be an output of either Algorithm \ref{RMISO} or \ref{RMISO_DR}. Then almost surely
\begin{align}
\sum_{n = 1}^{\infty} h_n^{v_{n+1}}(\param_n) \leq \frac{1}{\pi_{\text{min}}} \Delta_0.
\end{align}
\end{lemma}

\begin{proof}
    By Definition \ref{def: surrogates}, for each $n$ the quantity $h_n^{v_{n+1}}(\param_n)$ is non-negative. Therefore, it suffices to show that the sequence of partial sums, $\sum_{n= 1}^N h_n^{v_{n+1}}(\param_n)$ is uniformly bounded. 
    
    Recall that
    \begin{align}
        \bar{g}_{n+1}(\param_{n+1}) \leq \bar{g}_{n+1}(\param_n)&= \bar{g}_n(\param_n) + (g_{n+1}^{v_{n+1}}(\param_n) - g_n^{v_{n+1}}(\param_n))\pi(v_{n+1})\\
        &= \bar{g}_n(\param_n) + (f^{v_{n+1}}(\param_n) - g_n^{v_{n+1}}(\param_n))\pi(v_{n+1}).
    \end{align}
    We then have
    \begin{align*}
        \sum_{n = 1}^N h_n^{v_{n+1}}(\param_{n}) &= \sum_{n = 1}^N g_{n}^{v_{n+1}}(\param_{n}) - f^{v_{n+1}}(\param_{n})\\
        &\leq \sum_{n = 1}^N \frac{1}{\pi(v_{n+1})}\left(\bar{g}_{n}(\param_{n}) - \bar{g}_{n+1}(\param_{n+1})\right)\\
        & \leq \frac{1}{\pi_{\text{min}}} \sum_{n = 1}^{N} \bar{g}_n(\param_n) - \bar{g}_{n+1}(\param_{n+1})\\
        & \leq \frac{1}{\pi_{\text{min}}} \Delta_0
    \end{align*}
    which is what we needed to show. 
\end{proof}

\begin{prop}\label{prop:iterate gap} Suppose $(\param_n)_{n \geq 0}$ is an output of Algorithm \ref{RMISO_DR}. The for all $n \geq 1$, $\norm{\param_n - \param_{n-1}} \leq r_n$.
\end{prop}

\begin{proof}
This follows directly from the definition of $\param_n$ in Algorithm \ref{RMISO_DR}.
\end{proof}

Lemma \ref{lem: iterate gap summability} establishes the iterate stability.  These results are crucially used to control the surrogate error gradient $\norm{\nabla \bar{h}_n(\param_n)}$ in Lemma \ref{lem: conv_proof_surrogate_error_grad_sum} as well as in the asymptotic analysis of RMISO in Section \ref{sec: asymptotic analysis}.

\begin{lemma}[Finite variation of iterate differences]\label{lem: iterate gap summability}
The following hold almost surely:
    \begin{enumerate}[label=(\roman{*})]
    \item \label{item: dr summability} For Case \ref{case: diminishing_radius-markovian},
    \begin{align}
        \sum_{n = 1}^{\infty} \norm{\param_n - \param_{n-1}}^2 \leq \sum_{n = 1}^{\infty} r_n^2 < \infty.
    \end{align}
    \item \label{item: prox_reg summability} In either of the Cases \ref{case: prox_regularization-iid} or \ref{case: prox_regularization-markovian}, 
    \begin{align}
        \sum_{n =1}^{\infty} \frac{\rho_n + \mu}{2} \norm{\param_n - \param_{n-1}}^2 \leq \Delta_0 < \infty.
    \end{align}
    \end{enumerate}  
\end{lemma}

\begin{proof}
    The proof of \ref{item: dr summability} can be deduced from Proposition \ref{prop:iterate gap} and Assumption \ref{assumption: sequences}. 
    
    Now assume either of the Cases \ref{case: prox_regularization-iid} or \ref{case: prox_regularization-markovian}. Define $G_n(\param) = \bar{g}_n(\param) + \frac{\rho_n}{2}\norm{\param - \param_{n-1}}$. Then $G_n$ is $\rho_n + \mu$ strongly convex. Since $\theta_n$ is a minimizer of $G_n$ over $\Param$ we get
    \begin{align}
         G_n(\param_n) + \frac{\rho_n + \mu}{2}\norm{\param_n - \param_{n-1}}^2 \leq G_n(\param_{n-1}) = \bar{g}_n(\param_{n-1}) \leq \bar{g}_{n-1}(\param_{n-1})
    \end{align}
    where the last inequality is due to Lemma \ref{lem: surrogate monotonicity}.
    So 
    \begin{align}
        \frac{\rho_n + \mu}{2} \norm{\param_n - \param_{n-1}}^2 \leq \bar{g}_{n-1}(\param_{n-1}) - G_n(\param_n) \leq \bar{g}_{n-1}(\param_{n-1}) - \bar{g}_{n}(\param_n).
    \end{align}
    Hence,
    \begin{align}
        \sum_{n = 1}^{N} \frac{\rho_n + \mu}{2}\norm{\param_n - \param_{n-1}}^2 \leq \sum_{n = 1}^N \bar{g}_{n-1}(\param_{n-1}) - \bar{g}_n(\param_n) &= \bar{g}_0(\param_0) - \bar{g}_N(\param_N) \leq \Delta_0.
    \end{align}
    Letting $N \to \infty$ shows that 
    \begin{align}
        \sum_{n = 1}^{\infty} \frac{\rho_n + \mu}{2} \norm{\param_n - \param_{n-1}}^2 \leq \Delta_0
    \end{align}
    as desired. This shows \ref{item: prox_reg summability}. 
\end{proof}

The remaining results in this section concern Algorithm \ref{RMISO_DR} and are used both in the convergence rate analysis in Section \ref{sec: convergence rate analysis} as well as the asymptotic analsyis of Section \ref{sec: asymptotic analysis}. Recall that in this case we are assuming that $\nabla f^v$ is $L$-Lipschitz continuous for each $v \in V$. Proposition \ref{prop: lipshitz surrogate} states that this assumption implies $\nabla \bar{g}_n$ is differentiable and $2L$ Lipschitz for each $n$. 

\begin{prop}\label{prop: lipshitz surrogate}
    Let $\{\param_{v}\}_{v \in \mathcal{V}}$ be a collection of $|\mathcal{V}|$ points in $\Param$. Suppose that Assumption \ref{assumption: Lsmooth} holds and that $g_n^v \in \mathcal{S}_{L}(f^v, \param_v)$ for each $v$. Then
    \begin{enumerate}[label=(\roman{*})]
     \item The gradient of the objective function $\nabla f = \sum_{v \in \mathcal{V}} \nabla f^v\pi(v)$ is $L$-Lipschitz over $\Param$. 

     \vspace{0.1cm}
     \item For each $v$, $\nabla g_n^v$ is $2L$-Lipchitz over $\Param$. In addition, $\nabla \bar{g}_n$ is $2L$-Lipchitz. 
    \end{enumerate}
\end{prop}

\begin{proof}
    Since $\pi$ is a probability distribution, (i) follows easily from the triangle inequality. 

    For (ii) note that $\nabla (g_n^v - f^v) = \nabla h_n^v$ is $L$-Lipshitz by Definition \ref{def: surrogates}. Then since $\nabla g_n^v = \nabla h_n^v + \nabla f^v$ it follows from the triangle inequality that $\nabla g_n^v$ is $2L$-Lipschitz. Then recalling that $\nabla \bar{g}_n = \sum_{v \in \mathcal{V}} \nabla g_n^v \pi(v)$ another application of the triangle inequality shows that $\nabla \bar{g}_n$ is $2L$-Lipschitz continuous. 
\end{proof}

\begin{prop}\label{prop: surogate_gradient_step_sum}
    Assume Case \ref{case: diminishing_radius-markovian} and let $(\param_n)_{n \geq 0}$ be an output of Algorithm \ref{RMISO_DR}. Then 
    \begin{align}
        \sum_{n = 1}^{N} |\langle \nabla \bar{g}_n(\param_{n-1}), \param_n - \param_{n-1} \rangle| \leq \Delta_0 + L \sum_{n = 1}^N r_n^2
    \end{align}
    almost surley.
\end{prop}

\begin{proof}
    Since $\nabla \bar{g}_n$ is $2L$-Lipschitz continuous by Proposition \ref{prop: lipshitz surrogate}, by Lemma \ref{lem: L-smooth_surrogate}
    \begin{align} 
        |\bar{g}_n(\param_n) - \bar{g}_n(\param_{n-1}) - \langle \nabla \bar{g}_n(\param_{n-1}), \param_n - \param_{n-1} \rangle | \leq L\norm{\param_n - \param_{n-1}}^2.
        \label{eq: prop_gradient_step_sum_1}
    \end{align}
    Under Algorithm \ref{RMISO_DR} we have $\bar{g}_n(\param_n) \leq \bar{g}_n(\param_{n-1})$ by the definition of $\param_n$ and $\bar{g}_n(\param_{n-1}) \leq \bar{g}_{n-1}(\param_{n-1})$ by Lemma \ref{lem: surrogate monotonicity}. These observations together with \eqref{eq: prop_gradient_step_sum_1} imply
    \begin{align}
        \left| \langle \nabla \bar{g}_n(\param_{n-1}), \param_n - \param_{n-1} \rangle| \right| &\leq |\bar{g}_{n}(\param_{n-1}) - \bar{g}_n(\param_n)| + L\norm{\param_n - \param_{n-1}}^2\\
        &= \bar{g}_n(\param_{n-1}) - \bar{g}_n(\param_n) + L\norm{\param_n - \param_{n-1}}^2\\
        & \leq \bar{g}_{n-1}(\param_{n-1}) - \bar{g}_{n}(\param_n) + L\norm{\param_{n} - \param_{n-1}}^2.
    \end{align}
    We have
    \begin{align}
        \sum_{n = 1}^{N} \bar{g}_{n-1}(\param_{n-1}) - \bar{g}_n(\param_n)  \leq \Delta_0
    \end{align}
    almost surely. Therefore
    \begin{align}
        \sum_{n = 1}^{N} \left|\langle \nabla \bar{g}_n(\param_{n-1}), \param_n - \param_{n-1} \rangle \right| & \leq \sum_{n = 1}^{N} \bar{g}_{n-1}(\param_{n-1}) - \bar{g}_n(\param_n) + L \norm{\param_n - \param_{n-1}}^2 \\
        & \leq \Delta_0 + L \sum_{n = 1}^{
        N} \norm{\param_n - \param_{n-1}}^2\\
        & \leq \Delta_0 + L \sum_{n = 1}^{N} r_n^2,
    \end{align}
    where the last line uses $\norm{\param_n - \param_{n-1}} \leq r_n$. 
\end{proof}

The next lemma is a key to establishing iteration complexity of Algorithm \ref{RMISO_DR}. A similar lemma was used to analyze block majorization-minimization with diminishing radius in \cite{lyu2021BCD}.

\begin{lemma}[Approximate first order optimality]\label{lem: surrogate_grad_inf_lowerbound}
\indent
Let $(\param_n)_{n \geq 0}$ be an output of Algorithm \ref{RMISO_DR} and let $b_n = \min\{1, r_n\}$. Then 
    \begin{align}
    b_n \sup_{\param \in \Param, \norm{\param - \param_{n-1}} \leq 1} \langle - \nabla \bar{g}_{n-1}(\param_{n-1}), \param - \param_{n-1} \rangle \leq \langle -\nabla \bar{g}_n(\param_{n-1}), \param_n - \param_{n-1} \rangle + r_n \norm{\nabla h_{n-1}^{v_n}(\param_{n-1})} + 2Lr_n^2. 
    \end{align}
\end{lemma}

\begin{proof}
     Fix $\param \in \Param$ with $\norm{\param - \param_{n-1}} \leq b_n$. By  definition of $\param_n$ we have $\bar{g}_n(\param_n) \leq \bar{g}_n(\param)$. Subtracting $\bar{g}_{n}(\param_{n-1})$ from both sides and using Proposition \ref{prop: lipshitz surrogate} and Lemma \ref{lem: L-smooth_surrogate},
    \begin{align}
        \langle \nabla \bar{g}_n(\param_{n-1}), \param_n - \param_{n-1} \rangle - L\norm{\param_n - \param_{n-1}}^2 &\leq \bar{g}_n(\param_n) - \bar{g}_n(\param_{n-1})\\
        & \leq \bar{g}_n(\param) - \bar{g}_n(\param_{n-1})\\
        & \leq \langle \nabla \bar{g}_n(\param_{n-1}), \param - \param_{n-1} \rangle + L\norm{\param - \param_{n-1}}^2. \label{eq: lem_surrogate_grad_inf_lowerbound_1}
    \end{align}
    Notice that 
    \begin{align}
        \nabla\bar{g}_n(\param_{n-1}) &= \nabla\bar{g}_{n-1}(\param_{n-1}) + \left[\nabla g_n^{v_n}(\param_{n-1}) - \nabla g_{n-1}^{v_n}(\param_{n-1})\right]\pi(v_n)\\
        &= \nabla \bar{g}_{n-1}(\param_{n-1}) + \left[\nabla f^{v_n}(\param_{n-1}) - \nabla g_{n-1}^{v_n}(\param_{n-1}\right]\pi(v_n)\\
        &= \nabla \bar{g}_{n-1}(\param_{n-1}) - \nabla h_{n-1}^{v_n}(\param_{n-1}) \pi(v_n). 
    \end{align}
    where the second line used $g_n^{v_n} \in \mathcal{S}_{L}(f^{v_n}, \param_{n-1})$ and item (ii) of Definition \ref{def: surrogates}, and the third line used the definition of $h_{n-1}^{v_n}$. Therefore, adding and subtracting $\langle \nabla \bar{g}_{n-1}(\param_{n-1}), \param - \param_{n-1} \rangle$ from the right hand side of \eqref{eq: lem_surrogate_grad_inf_lowerbound_1} we get
    \begin{align}
        \langle \nabla \bar{g}_n(\param_{n-1}), \param_n - \param_{n-1} \rangle &\leq \langle \nabla \bar{g}_{n-1}(\param_{n-1}), \param - \param_{n-1} \rangle  -  \pi(v_n)\langle \nabla h_{n-1}^{v_n}(\param_{n-1}), \param - \param_{n-1} \rangle + L\norm{\param - \param_{n-1}}^2 \\
        & \qquad \qquad \qquad \qquad + L \norm{\param_n - \param_{n-1}}^2\\
        & \leq  \langle \nabla \bar{g}_{n-1}(\param_{n-1}), \param - \param_{n-1} \rangle + \norm{\nabla h_{n-1}^{v_n}(\param_{n-1})}\norm{\param - \param_{n-1}} + L\norm{\param - \param_{n-1}}^2\\
        & \qquad \qquad \qquad \qquad + L\norm{\param_n - \param_{n-1}}^2\\
        & \leq \langle \nabla \bar{g}_{n-1}(\param_{n-1}), \param - \param_{n-1} \rangle + r_n \norm{\nabla h_{n-1}^{v_n}(\param_{n-1})} + 2Lr_n^2
    \end{align}
    where the last line used $\norm{\param_n - \param_{n-1}} \leq r_n$ and $\norm{\param - \param_{n-1}} \leq b_n \leq r_n$. 
    Since the above holds for all $\param \in \Param$ with $\norm{\param - \param_{n-1}} \leq b_n$ we obtain
    \begin{align}
        \langle \nabla \bar{g}_n(\param_{n-1}), \param_n - \param_{n-1} \rangle \leq \inf_{\param \in \Param, \norm{\param - \param_{n-1}} \leq b_n} \langle \nabla \bar{g}_{n-1}(\param_{n-1}), \param - \param_{n-1} \rangle + r_n\norm{\nabla h_{n-1}^{v_n}(\param_{n-1})} + 2Lr_n^2.
    \label{eq: surrogate_grad_inf_lowerbound-DR intermediate}
    \end{align}
    Finally notice that since $b_n \leq 1$, the convexity of $\Param$ implies that $\param_{n-1} + b_n(\param - \param_{n-1}) \in \Param$ for any $\param \in \Param$. Thus, if there exists $\param \in \Param$ with $\norm{\param - \param_{n-1}} \leq 1$ there is $\param' \in \Param$ with $\norm{\param' - \param_{n-1}} \leq b_n$ such that the direction of $\param'- \param_{n-1}$ agrees with that of $\param - \param_{n-1}$.  Therefore
    \begin{align}
        b_n\sup_{\param \in \Param, \norm{\param - \param_{n-1}}\leq 1} \langle - \nabla \bar{g}_{n-1}(\param_{n-1}), \param - \param_{n-1} \rangle &= \sup_{\param \in \Param, \norm{\param - \param_{n-1}} \leq 1} \langle - \nabla \bar{g}_{n-1}(\param_{n-1}), b_n(\param - \param_{n-1}) \rangle\\
        & \leq \sup_{\param \in \Param, \norm{\param - \param_{n-1}} \leq b_n} \langle - \nabla \bar{g}_{n-1}(\param_{n-1}), \param - \param_{n-1} \rangle. 
    \end{align}
    combining this with \eqref{eq: surrogate_grad_inf_lowerbound-DR intermediate} we complete the proof. 
\end{proof}

\section{Convergence Rate Analysis}\label{sec: convergence rate analysis}

In this subsection we prove the convergence rate guarantees of Theorem \ref{thm: convergence rates to stationarity}. 

\subsection{The key lemma}

First we state and prove Lemma \ref{lem: conv_proof_surrogate_error_grad_sum} which lies at the heart of our analysis. It allows us to relate the surrogate error gradient $\norm{\nabla \bar{h}_n(\param_n)}$ to the sequence of parameter differences $(\norm{\param_n - \param_{n-1}}^2)$ which is known to be summable by Lemma \ref{lem: iterate gap summability}. It is important to note that we only use the recurrence of data sampling Assumption \ref{assumption: process properties} and the structure of the algorithm in the proof. 

\begin{lemma}[Key lemma]\label{lem: conv_proof_surrogate_error_grad_sum}
Let $(c_n)_{n \geq 1}$ be a non-increasing sequence of positive numbers. For any of the cases \ref{case: prox_regularization-iid}-\ref{case: diminishing_radius-markovian} and any $v \in V$,
\begin{align}
\E \left[ \sum_{n = 1}^N c_n\norm{\nabla \bar{h}_n(\param_n)} \right] \leq Lt_{\odot} \left(\sum_{n = 1}^N c_n^2 \right)^{1/2} \E \left[ \left(\sum_{n=1}^N \norm{\param_n - \param_{n-1}}^2 \right)^{1/2}\right].
\end{align}
\end{lemma} 

\begin{proof}
    Fix some $v \in V$. We recall that $k^v(n)$ is the last time before $n$ that the sampling process visited $v$ and therefore the last time the surrogate $g_n^v$ was updated. We then have $g_n^v \in \mathcal{S}_{L}(f^v, \param_{k^v(n)-1})$ so by the definition of first-order surrogates (Definition \ref{def: surrogates}) $\nabla h_n^v(\param_{k^v(n)-1}) = 0$. Combining this with the Lipschitz continuity of $\nabla h_n^v$ we get
\begin{align}
    c_n\norm{\nabla h_n^v(\param_n)} = c_n\norm{\nabla h_n^v(\param_n) - \nabla h_n^v(\param_{k^v(n)-1})} \leq Lc_n\norm{\param_n - \param_{k^v(n)-1}} \leq L\sum_{i = k^v(n)}^n c_n\norm{\param_i - \param_{i-1}}.
\end{align}
Therefore by the triangle inequality,
\begin{align}
c_n\norm{\nabla \bar{h}_n(\param_n)} \leq c_n \sum_{v \in \mathcal{V}}\norm{\nabla h_n^v(\param_n)}\pi(v) \leq L \sum_{v \in \mathcal{V}} \left( \sum_{i = k^v(n)}^n c_n \norm{\param_i - \param_{i-1}} \right)\pi(v).
\end{align}
For an integer $n$, let $p^v(n) = \inf\{ j > n: v_j = v\}$ be the first time strictly after $n$ that the sampling algorithm visits $v$. Denote $a\land b :=\min(a,b)$. We have 
\begin{align}
\sum_{n = 1}^N \sum_{v \in \mathcal{V}} \left( \sum_{i = k^v(n)}^n c_n\norm{\param_i - \param_{i-1}} \right) \pi(v) &= \sum_{v \in \mathcal{V}} \left( \sum_{n = 1}^N \sum_{i = k^v(n)}^n c_n \norm{\param_i - \param_{i-1}} \right)\pi(v)\\
&=\sum_{v \in \mathcal{V}} \left(\sum_{i = 1}^N \sum_{n = i}^{N \wedge p^v(i) - 1} c_n \norm{\param_i - \param_{i-1}} \right)\pi(v)\\
& \leq \sum_{v \in \mathcal{V}} \left(\sum_{i = 1}^N c_i\norm{\param_i - \param_{i-1}}(p^v(i) - i) \right)\pi(v)\\
& = \sum_{v \in \mathcal{V}} \left(\sum_{i = 1}^N c_i\norm{\param_i - \param_{i-1}}\tau_{i, v} \right)\pi(v),
\end{align}
where the third line used that $(c_n)$ is non-increasing. So we get
\begin{align}
    \E \left[ \sum_{n = 1}^N c_n\norm{\nabla \bar{h}_n(\param_n)}\right] &\leq L \E\left[ \sum_{v \in \mathcal{V}} \left(\sum_{n = 1}^N c_n\norm{\param_n - \param_{n-1}}\tau_{n, v} \right) \pi(v)\right]\\
    &= L \E \left[ \sum_{n = 1}^N c_n \norm{\param_n - \param_{n-1}} \left(\sum_{v \in V} \tau_{n, v} \pi(v) \right)\right]\\
    &= L \E \left[ \sum_{n = 1}^N c_n\norm{\param_n - \param_{n-1}} \left(\sum_{v \in \mathcal{V}}\E[\tau_{n, v}|\cF_n]\pi(v) \right)\right]\\
    &\leq Lt_{\odot} \E\left[\sum_{n=1}^N c_n\norm{\param_n - \param_{n-1}}\right]\\
    & \leq Lt_{\odot} \left(\sum_{n = 1}^N c_n^2 \right)^{1/2}\E \left[ \left(\sum_{n = 1}^N \norm{\param_n - \param_{n-1}}^2 \right)^{1/2}\right]
\end{align}
with the second to last line using Assumption \ref{assumption: process properties} and the last using Cauchy-Schwartz.
\end{proof}

\subsection{The constant proximal regularization case \ref{case: prox_regularization-iid}}

In this section we prove Theorem \ref{thm: convergence rates to stationarity-extended} for Case \ref{case: prox_regularization-iid}.

\begin{proof}[\textbf{Proof of Theorem \ref{thm: convergence rates to stationarity-extended} for case \ref{case: prox_regularization-iid}}]

We first use the linearity of the limit and the differentiability of the average surrogate error $\bar{h}_n$ to get
\begin{align}
    \left| \nabla \bar{g}_n(\param_n, \param - \param_n) - \nabla f(\param_n, \param - \param_n) \right| = \left | \langle \nabla \bar{h}_n(\param_n), \param - \param_n \rangle \right| \leq \norm{\nabla \bar{h}_n(\param_n)}\norm{\param - \param_n} \leq \norm{\nabla \bar{h}_n(\param_n)}
\end{align}
for all $\param \in \Param$ with $\norm{\param - \param_n} \leq 1$. It the follows from the triangle inequality, taking supremums, and then expectations that
\begin{align}
    \E\left[\sup_{\param \in \Param, \norm{\param - \param_n}\leq 1} -\nabla f(\param_n, \param - \param_n) \right] \leq \E \left[\sup_{\param \in \Param, \norm{\param - \param_n} \leq 1} - \nabla \bar{g}_n(\param_n, \param -\param_n) \right] + \E\left[\norm{\nabla \bar{h}_n(\param_n)}\right].
    \label{eq: pf_thm_conv_rates_const_pr-optimality_bound}
\end{align}
Our goal will be to control the sum of right hand side. 

We first address the first term on the right hand side of \eqref{eq: pf_thm_conv_rates_const_pr-optimality_bound}. Recall that in this case we are using constant proximal regularization, so $\rho_n \equiv \rho$ for some $\rho \geq 0$. For any $\param \in \Param$
\begin{align}
    \nabla \bar{g}_n(\param_n, \param - \param_n) + \rho\langle \param_n - \param_{n-1}, \param - \param_n\rangle \geq 0
\end{align}
since $\param_n$ is a minimizer of $\bar{g}_n(\param) + \frac{\rho}{2}\norm{\param - \param_{n-1}}^2$ over $\Param$. Then 
\begin{align}
    -\nabla \bar{g}_n(\param_n, \param - \param_n) \leq \rho \langle \param_n - \param_{n-1}, \param - \param_n \rangle \leq \rho \norm{\param_n - \param_{n-1}}\norm{\param - \param_n} \leq \rho \norm{\param_n - \param_{n-1}} 
\end{align}
for any $\param \in \Param$ with $\norm{\param - \param_{n}} \leq 1$. Therefore,
\begin{align}
    \E \left[ \sum_{n = 1}^N \sup_{\param \in \Param, \norm{\param - \param_n} \leq 1} -\nabla \bar{g}_n(\param_n, \param - \param_n) \right] &\leq \rho \E \left[\sum_{n = 1}^N \norm{\param_n - \param_{n-1}}\right]\\
    & \leq \rho \sqrt{N} \E \left[ \left(\sum_{n=1}^N \norm{\param_n - \param_{n-1}}^2 \right)^{1/2} \right]
\end{align}
where we used the Cauchy-Schwartz inequality in the last line. By  Lemma \ref{lem: iterate gap summability}
\begin{align}
\left(\sum_{n=1}^N \norm{\param_n - \param_{n-1}}^2 \right)^{1/2} \leq \sqrt{\frac{2\Delta_0}{\rho + \mu}}
\end{align}
almost surely. Thus,
\begin{align}
    \E \left[ \sum_{n = 1}^N \sup_{\param \in \Param, \norm{\param - \param_n} \leq 1} -\nabla \bar{g}_n(\param_n, \param - \param_n) \right] \leq \rho\sqrt{\frac{2N\Delta_0}{\rho + \mu}}.
    \label{eq: pf_thm_conv_rates_const_pr-surrogate}
\end{align}

We now turn to the second term on the right hand side of \eqref{eq: pf_thm_conv_rates_const_pr-optimality_bound}.
By Lemma \ref{lem: conv_proof_surrogate_error_grad_sum} with $c_n = 1$ and Lemma \ref{lem: iterate gap summability}
\begin{align}
    \E\left[\sum_{n = 1}^N \norm{\nabla \bar{h}_n(\param_n)}\right] \leq \sqrt{N}Lt_{\odot} \E \left[ \left( \sum_{n = 1}^N \norm{\param_n - \param_{n-1}}^2 \right)^{1/2}\right] \leq \sqrt{\frac{2N\Delta_0}{\rho + \mu}}Lt_{\odot}.
    \label{eq:eq: pf_thm_conv_rates_const_pr-error}
\end{align}

Now, summing both sides of \eqref{eq: pf_thm_conv_rates_const_pr-optimality_bound} and using \eqref{eq: pf_thm_conv_rates_const_pr-surrogate} and \eqref{eq:eq: pf_thm_conv_rates_const_pr-error},
\begin{align}
    \sum_{n = 1}^N \E \left[ \sup_{\param \in \Param, \norm{\param - \param_n} \leq 1} -\nabla f(\param_n, \param - \param_n)\right] &= \E \left[ \sum_{n = 1}^N\sup_{\param \in \Param, \norm{\param - \param_n} \leq 1} -\nabla f(\param_n, \param - \param_n) \right]\\
    & \leq \sqrt{2N\Delta_0}\left(\frac{\rho}{\sqrt{\rho + \mu}} + \frac{Lt_{\odot}}{\sqrt{\rho + \mu}} \right).
\end{align}
This shows
\begin{align}
    \min_{1 \leq n \leq N} \E\left[ \sup_{\param \in \Param, \norm{\param - \param_n}\leq 1} -\nabla f(\param_n, \param - \param_n) \right] \leq \frac{\sqrt{2\Delta_0}\left( \frac{\rho}{\sqrt{\rho + \mu}} + \frac{Lt_{\odot}}{\sqrt{\rho + \mu}}\right)}{\sqrt{N}}.
\end{align}

\end{proof}

\subsection{The dynamic proximal regularization case \ref{case: prox_regularization-markovian}}

In this section we prove Theorem \ref{thm: convergence rates to stationarity-extended} for Case \ref{case: prox_regularization-markovian}. Recall the definition of $t_{\text{cov}}$ from \eqref{eq: tcov}. Before proving the theorem we introduce a Lemma adapted from \cite{even2023stochastic} Lemma A.5. This is used to bound the expected sum of the first $N$ dynamic regularization parameters $\rho_n$ in terms of $t_{\text{cov}}$.

\begin{lemma}
    Let $a_n = \max_{v \in V}(n - k^v(n))$. Then 
    \begin{align}
        \sum_{n = 1}^N \E[a_n] \leq Nt_{\text{cov}}.
    \end{align}
    \label{lem: sum max resid time}
\end{lemma}

\begin{proof}
    Since $a_n \leq n-1$ we have
    \begin{align}
        \sum_{n = 1}^N \E[a_n] = \sum_{n = 1}^N \sum_{i = 1}^{n-1} \P(a_n \geq i).
    \end{align}
    Let $b_n = \max_{v \in \mathcal{V}} \tau_{n, v}$. We note that if $a_n \geq i$ then there is $v \in V$ with $v_j \neq v$ for all $n - i <  j \leq n$ and so $\tau_{n-i, v} \geq i$. So we have the inclusion $\{a_n \geq  i\} \subseteq \{b_{n-i} \geq i\}$. Therefore
    \begin{align}
        \sum_{n= 1}^N \sum_{i = 1}^{n-1} \P(a_n \geq i) &= \sum_{i = 1}^N \sum_{n = i+1}^N \P(a_n \geq i)\\ &\leq \sum_{i = 1}^N \sum_{n = i+1}^N \P(b_{n-i} \geq i)\\
        &= \sum_{s = 1}^{N-1} \sum_{t = 1}^{N-s} \P(b_s \geq t)\\
        & \leq \sum_{s = 1}^{N-1} \sum_{t = 1}^{\infty} \P(b_s \geq t)\\
        &\leq N \sup_{n \geq 1} \E[b_n].
    \end{align}
    We have 
    \begin{align}
        \E[b_n] = \E\left[ \E \left[\max_{v \in \mathcal{V}} \tau_{n, v} \Big| \cF_n \right] \right] \leq t_{\text{cov}}
    \end{align}
    so we are done. 
\end{proof}

\begin{proof}[\textbf{Proof of Theorem \ref{thm: convergence rates to stationarity-extended} for Case \ref{case: prox_regularization-markovian}}]
    The proof in this case follows the same strategy as Case \ref{case: prox_regularization-iid}, but is slightly more complicated due to the randomness of the dynamic proximal regularization parameter. 

   Define $\delta_n := \bar{g}_{n-1}(\param_{n-1}) - \bar{g}_n(\param_n)$. By optimality of $\param_n$ and Lemma \ref{lem: surrogate monotonicity}, 
   \begin{align}
       \bar{g}_n(\param_n) + \frac{\rho_n}{2}\norm{\param_n - \param_{n-1}}^2 \leq \bar{g}_{n-1}(\param_{n-1})
   \end{align} 
   so $\norm{\param_n - \param_{n-1}} \leq \sqrt{2\rho_n^{-1}\delta_n}$. Using similar reasoning as in the proof for case \ref{case: prox_regularization-iid}
    \begin{align}
        -\nabla \bar{g}_n(\param_n, \param - \param_n) \leq \rho_n\norm{\param_n - \param_{n-1}} \leq \sqrt{2\rho_n \delta_n}.
    \end{align}
    We have using Cauchy-Schwartz twice,
    \begin{align}
        \sum_{n = 1}^{N} \E[\sqrt{\rho_n \delta_n}]
        & \leq \sum_{n = 1}^N (\E[\rho_n])^{1/2}(\E[\delta_n])^{1/2}\\
        & \leq \left(\sum_{n = 1}^N \E[\rho_n] \right)^{1/2} \left( \sum_{n = 1}^N \E[\delta_n] \right)^{1/2}\\
        & \leq \sqrt{N( \rho + t_{\text{cov}}) \Delta_0}.
    \end{align}
    The last inequality here uses $\rho_n = \rho + \max_{v \in V}(n - k^v(n))$ and Lemma \ref{lem: sum max resid time} as well as $\sum_{n = 1}^N \delta_n \leq \Delta_0$ a.s. It follows that
    \begin{align}
        \sum_{n = 1}^N \E \left[ \sup_{\param \in \Param, \norm{\param - \param_{n-1}} \leq 1} -\nabla \bar{g}_n(\param_n, \param - \param_n) \right] \leq \sqrt{2N(\rho + t_{\text{cov}})\Delta_0}.
    \label{eq: pf_thm_conv_rates_dynamic-surrogate}
    \end{align}

    To handle the gradient error $\norm{\nabla \bar{h}_n(\param_n)}$ we first use Lemma \ref{lem: iterate gap summability} and $\rho_n \geq \rho$ to conclude
    \begin{align}
        \sum_{n = 1}^{N} \frac{\rho + \mu}{2}\norm{\param_n - \param_{n-1}}^2 \leq \sum_{n= 1}^N \frac{\rho_n + \mu}{2}\norm{\param_n - \param_{n-1}}^2 \leq \Delta_0 
    \end{align}
    almost surely. It then follows from Lemma \ref{lem: conv_proof_surrogate_error_grad_sum} that
    \begin{align}
        \sum_{n = 1}^N \E\left[ \norm{\nabla \bar{h}_n(\param_n)}
        \right] \leq \sqrt{N}Lt_{\odot}\E \left[ \left( \sum_{n = 1}^N \norm{\param_n - \param_{n-1}} \right)^{1/2}\right] \leq \sqrt{\frac{2N\Delta_0}{\rho}}Lt_{\odot}.
    \label{eq: pf_thm_conv_rates_dynamic-error}
    \end{align}
    Finally, combining \eqref{eq: pf_thm_conv_rates_dynamic-surrogate} and \eqref{eq: pf_thm_conv_rates_dynamic-error} we get
    \begin{align}
        \sum_{n = 1}^N \E \left[ \sup_{\param \in \Param, \norm{\param - \param_n} \leq 1} -\nabla f(\param_n, \param - \param_n)\right] \leq \sqrt{2N\Delta_0} \left(\sqrt{\rho + t_{\text{cov}}} + \frac{Lt_{\odot}}{\sqrt{\rho + \mu}} \right) 
    \end{align}
    and so we deduce
    \begin{align}
        \min_{1 \leq n \leq N} \E \left[\sup_{\param \in \Param, \norm{\param - \param_n}\leq 1} - \nabla f(\param_n, \param - \param_n) \right] \leq \frac{\sqrt{2\Delta_0} \left( \sqrt{\rho + t_{\text{cov}}} + \frac{Lt_{\odot}}{\sqrt{\rho + \mu}} \right) }{\sqrt{N}}.
    \end{align}
    We complete the proof by substituting the bound for $t_{\text{cov}}$ in Proposition \ref{prop: return time finite exponential moments and cover time bound}.  
\end{proof}

\subsection{The diminishing radius case \ref{case: diminishing_radius-markovian}}

Here we prove Theorem \ref{thm: convergence rates to stationarity-extended} for Case \ref{case: diminishing_radius-markovian}.

\begin{proof}[\textbf{Proof of Theorem \ref{thm: convergence rates to stationarity-extended} for Case \ref{case: diminishing_radius-markovian}}]

Similar to the proof in Case \ref{case: prox_regularization-iid} we have
\begin{align}
    \E \left[ \sup_{\param \in \Param, \norm{\param - \param_n}\leq 1} -\langle \nabla f(\param_n), \param - \param_n \rangle \right] \leq \E \left[ \sup_{\param \in \Param, \norm{\param - \param_n} \leq 1} -\langle \nabla \bar{g}_n(\param_n), \param - \param_n \rangle \right] + \E \left[ \norm{\nabla \bar{h}_n(\param_n)}\right]
\label{eq: eq: pf_thm_conv_rates_dr-optimality_bound}
\end{align}

Let $b_n = r_n \wedge 1$. Then by Lemma \ref{lem: surrogate_grad_inf_lowerbound}
\begin{align}
    &\sum_{n = 1}^N b_{n+1} \E \left[ \sup_{\param \in \Param, \norm{\param - \param_n} \leq 1 } -\langle \nabla \bar{g}_n(\param_n), \param - \param_n \rangle \right]\\
    &\leq \sum_{n = 1}^N \E \left[\langle -\nabla \bar{g}_{n+1}(\param_{n}), \param_{n+1} - \param_{n} \rangle \right] + \sum_{n = 1}^Nr_{n+1} \E\left[\norm{\nabla h_{n}^{v_{n+1}} (\param_{n})}\right] + \sum_{n = 1}^N 2Lr_{n+1}^2.
\label{eq: pf_conf_dr_surrogate_opt_bound}
\end{align}
Because $h_n^{v_{n+1}}$ is non-negative and has $L$-Lipschitz continuous gradients $\norm{\nabla h_n^{v_{n+1}}(\param_n)} \leq \sqrt{2L h_n^{v_{n+1}}(\param_n)}$ (see Lemma \ref{lem: L-smooth_grad_upperbound}). Then
\begin{align}
    \sum_{n = 1}^N r_{n+1} \E \left[\norm{\nabla h_n^{v_{n+1}}(\param_n)} \right] &\leq \sum_{n = 1}^N r_{n+1} \E \left[ \sqrt{2L h_n^{v_{n+1}}(\param_n)}\right]\\
    & \leq \left(\sum_{n = 1}^N r_{n+1}^2 \right)^{1/2} \left(\sum_{n = 1}^{N} \E\left[2Lh_n^{v_{n+1}}(\param_n)\right] \right)^{1/2}\\
    & \leq \sqrt{\frac{2L\Delta_0}{\pi_{\text{min}}}\sum_{n =1}^N r_{n+1}^2}.
\end{align}
Here the second line used Cauchy-Schwartz and then Jensen's inequality to move the square inside the expectation and the last line used Lemma \ref{lem:single gap convergence}. From Proposition \ref{prop: surogate_gradient_step_sum},
\begin{align}
    \sum_{n = 1}^N \E \left[\langle -\nabla \bar{g}_{n+1}(\param_n), \param_{n+1} - \param_n \rangle \right] \leq \Delta_0 + L\sum_{n = 1}^N r_{n+1}^2
\end{align}
so from \eqref{eq: pf_conf_dr_surrogate_opt_bound}
\begin{align}
    \sum_{n = 1}^N b_{n+1}\E \left[ \sup_{\param \in \Param, \norm{\param - \param_n} \leq 1 } -\langle \nabla \bar{g}_n(\param_n), \param - \param_n \rangle \right] \leq \Delta_0 + \sqrt{\frac{2L\Delta_0}{\pi_{\text{min}}}\sum_{n =1}^N r_{n+1}^2} + 3L\sum_{n = 1}^N r_{n+1}^2.
\label{eq: pf_thm_conv_rates_dr-surrogate}
\end{align}
From Lemma \ref{lem: conv_proof_surrogate_error_grad_sum} with $c_n = r_{n+1}$
\begin{align}
    \sum_{n = 1}^N r_{n+1} \E\left[\norm{\nabla \bar{h}_n(\param_n)}\right] \leq \sum_{v \in V} \E \left[ \sum_{n = 1}^N r_{n+1} \norm{\nabla h_n^v(\param_n)} \right]\pi(v) \leq Lt_{\odot} \left(\sum_{n = 1}^N r_{n+1}^2 \right)^{1/2} \E \left[\left(\sum_{n = 1}^N \norm{\param_n - \param_{n-1}}^2 \right)^{1/2}\right].
\end{align}
Since $\norm{\param_n - \param_{n-1}} \leq r_n$ and $(r_n)$ is non-increasing, this bound reduces to 
\begin{align}
    \sum_{n = 1}^{N}r_{n+1} \E \left[\norm{\nabla \bar{h}_n(\param_n)}\right] \leq Lt_{\odot} \sum_{n = 1}^N r_n^2.
\label{eq: pf_thm_conv_rates_dr-error}
\end{align}
Combining \eqref{eq: pf_thm_conv_rates_dr-surrogate} and \eqref{eq: pf_thm_conv_rates_dr-error} with \eqref{eq: eq: pf_thm_conv_rates_dr-optimality_bound} we have
\begin{align}
    \sum_{n = 1}^N b_n \E\left[\sup_{\param \in \Param, \norm{\param - \param_n} \leq 1} -\langle \nabla f(\param_n), \param - \param_n \rangle \right] \leq \Delta_0 + \sqrt{\frac{2L\Delta_0}{\pi_{\text{min}}}\sum_{n = 1}^N r_n^2} + \Big(3 + t_{\odot}\Big)L\sum_{n = 1}^N r_n^2
\end{align}
so 
\begin{align}
    \min_{1 \leq n \leq N} \E\left[\sup_{\param \in \Param, \norm{\param - \param_n} \leq 1} -\langle \nabla f(\param_n), \param - \param_n \rangle \right] \leq \frac{\Delta_0 + \sqrt{\frac{2L\Delta_0}{\pi_{\text{min}}}\sum_{n = 1}^N r_n^2} + \Big(3 + t_{\odot}\Big)L\sum_{n = 1}^N r_n^2}{\sum_{n = 1}^N b_n}.
\end{align}

\end{proof}

\subsection{Proofs of Corollaries}

In this section, we prove Corollaries \ref{cor: conv rates to stationarity-param in interior} and \ref{cor: iteration complexity}.

\begin{proof}[\textbf{Proof of corollary \ref{cor: conv rates to stationarity-param in interior}}]

    Fix $N$ and let $k$ be the integer recognizing the minimum in \eqref{eq: conv_to_stationarity_alg 2 obj iid-extended}. If $\Param = \R^p$ then we may choose $\param^*$ so that $\param^* - \param_{k} = -\frac{\nabla f(\param_{k})}{\norm{\nabla f(\param_{k})}}$. Thus, 
    \begin{align}
        \min_{1 \leq n \leq N} \E[\norm{\nabla f(\param_n)}] \leq \E[\norm{\nabla f(\param_{k})}] &= \E\left[ \langle - \nabla f(\param_{k}), \param^* - \param_{k} \rangle \right]\\
        & \leq \E \left[ \sup_{\param \in \Param, \norm{\param - \param_{k}} \leq 1} \langle - \nabla f(\param_{k}), \param - \param_{k} \rangle \right] = O\left(N^{-1/2}\right).
    \end{align}
    If instead $\Param \neq \R^p$ but the second condition $\text{dist}(\param_k, \partial \Param) \geq c$ holds, we can take $\param - \param_{k} = -c\frac{\nabla f(\param_{k})}{\norm{\nabla f(\param_{k})}}$. In doing so we obtain
    \begin{align}
        \min_{1 \leq n \leq N} \E[\norm{\nabla f(\param_n)}] \leq \E[\norm{\nabla f(\param_{k})}] &= \frac{1}{c} \E\left[ \langle - \nabla f(\param_{k}), \param^* - \param_{k} \rangle \right]\\
        & \leq \frac{1}{c} \E \left[ \sup_{\param \in \Param, \norm{\param - \param_{k}} \leq 1} \langle - \nabla f(\param_{k}), \param - \param_{k} \rangle \right] = O\left( N^{-1/2} \right).
    \end{align}
    This shows \eqref{eq: cor-conv to stationarity-interior-PR}. The proof of \eqref{eq: cor-conv to stationarity-interior-DR} is similar. 
\end{proof}

\begin{proof}[\textbf{Proof of corollary \ref{cor: iteration complexity}}]
The convergence rates of Theorem \ref{thm: convergence rates to stationarity} are of order $O(N^{-1/2})$ for Algorithm \ref{RMISO}. Then we can prove \ref{item: pr iteration complexity} by choosing $N$ large enough so that $N^{-1/2} \leq  \eps$.

If we take $r_n = \frac{1}{\sqrt{n} \log n}$ in Algorithm \ref{RMISO_DR} then its corresponding convergence rate in \ref{thm: convergence rates to stationarity} is of order $O(n^{-1/2}\log n)$. Then \ref{item: dr iteration complexity} follows by using the fact that $n \geq \eps^{-2}(3\log \eps^{-1})^2$ implies $n^{-1/2}\log n \leq \eps$ for sufficiently large $\eps$. Indeed since $\frac{\log x}{\sqrt{x}}$ is decreasing for sufficiently large $x$ we have
\begin{align}
    n^{-1/2}\log n \leq \frac{\eps}{3\log \eps^{-1}} \left( 2 \log \eps^{-1} + 2 \log (3 \log \eps^{-1}) \right) \leq \eps
\end{align}
for $\eps$ sufficiently small.
\end{proof}

\section{Asymptotic Analysis}\label{sec: asymptotic analysis}

We use this section to prove Theorem \ref{thm: a.s. convergence}. Recall that by Proposition \ref{prop: return time finite exponential moments and cover time bound}, there are constants $C_1$ and $C_2$ with $\sup_{n \geq 1} \E[\tau_{n, v}^2 |\cF_n] \leq C_1$ and $\sup_{n \geq 1} \E[\tau_{n, v}^4 | \cF_n] \leq C_2$ for each $v \in \mathcal{V}$. Accordingly, we let $\mu_2 = \max_{v \in V} \sup_{n \geq 1} \left \lVert \E[\tau_{n, v}^2 |\cF_n] \right \rVert_{\infty}$ and $\mu_4 = \max_{v \in \mathcal{V}} \sup_{n \geq 1} \norm{\E[\tau_{n, v}^4 | \cF_n]}_{\infty}$. 

The first Lemma of this section states that the first and second moments of the dynamic regularization parameter $\rho_n$ in Algorithm \ref{RMISO} are uniformly bounded. While $\rho_n$ only appears in Algorithm \ref{RMISO}, the random variable $\max_{v \in V} (n - k^v(n))$ is also present in the analysis of Algorithm \ref{RMISO_DR}. Therefore, this Lemma is used in the analysis of both algorithms in this section.

\begin{lemma}\label{lem: uniform bound of reg parameter}
    Assume \ref{assumption: process properties}. Then there is a constant $C > 0$ such that
    \begin{align}
        \sup_{n \geq 1}\,\, \E[\rho_n] + \sup_{n \geq 1} \E[\rho_n^2] \leq C.
    \end{align}
\end{lemma}
\begin{proof}
    Fix $v \in \mathcal{V}$. For a positive integer $j$ we have
    \begin{align}
        \{n - k^v(n) \geq j\} = \{ k^v(n) \leq n - j\} \subseteq \{\tau_{n - j, v} \geq j\}.
    \end{align}
    Therefore, we get 
    \begin{align}
        \E[(n - k^v(n))] &= \sum_{j = 1}^{\infty} \P(n - k^v(n) \geq j) \\
        &\leq \sum_{j = 1}^{\infty} \P(\tau_{n - j, v} \geq j)\\
        & \leq \sum_{j = 1}^{\infty} \frac{\E[\tau_{n-j, v}^2]}{j^2}\\
        &\leq  \mu_2\sum_{j = 1}^{\infty} j^{-2}
    \end{align}
    since $\E[\tau_{n - j, v}^2] \leq \mu_2$. Finally,
    \begin{align}
        \E[\rho_n] = \rho + \E\left[\max_{v \in \mathcal{V}}(n - k^v(n))\right] \leq \rho +  \sum_{v \in V} \E[n - k^v(n)] \leq \rho + |\mathcal{V}|\mu_2 \sum_{j = 1}^{\infty} j^{-2}.
    \end{align}
    To bound the second moment we follow the same approach:
    \begin{align}
        \E[(n- k^v(n))^2] &= \sum_{j = 1}^{\infty} \P \left((n - k^v(n))^2 \geq j \right)\\
        & \leq \sum_{j = 1}^{\infty} \P(\tau_{n - j, v}^2 \geq j)\\
        & \leq \sum_{j = 1}^{\infty} \frac{\E[\tau_{n - j, v}^4]}{j^2}\\
        & \leq \mu_4 \sum_{j = 1}^{\infty} j^{-2}. 
    \end{align}
    The proof is completed by mimicking the last line of the proof bounding the first moment.
    
\end{proof}

\subsection{The dynamic proximal regularization case \ref{case: prox_regularization-markovian}} \label{sec: dpr a.s. convergence}

Here we prove Theorem \ref{thm: a.s. convergence} \ref{item: stationary_limit_points-dynamic}. Our first lemma, Lemma \ref{lem: gap sum dynamic prox reg}, is similar to \ref{lem: conv_proof_surrogate_error_grad_sum} and key to showing that $\norm{\nabla \bar{h}_n(\param_n)} \to 0$. The difference is that we must deal with $\norm{\param_n - \param_{k^v(n) - 1}}^2$ instead of $\norm{\param_n - \param_{k^v(n)-1}}$. In order to relate this to the sequence of one step iterate differences $(\norm{\param_n - \param_{n-1}}^2)$ we need to use the Cauchy-Schwartz inequality which introduces a dependence on $\mu_2$ as well as $t_{\text{hit}}$.

\begin{lemma}\label{lem: gap sum dynamic prox reg}
Let $(\param_n)_{n \geq 0}$ be an output of Algorithm \ref{RMISO}. Assume case \ref{case: prox_regularization-markovian}. Then
\begin{align}
    \sum_{n = 1}^{\infty} \E[\bar{h}_n(\param_n)] < \infty.
\end{align}
\end{lemma}

\begin{proof}
    Since $\bar{h}_n(\param_n) = \sum_{v \in V} h_n^v(\param_n) \pi(v)$ and $\mathcal{V}$ is finite, it is sufficient to show $\sum_{n = 1}^{\infty} \E[h_n^v(\param_n)] < \infty$ for each $v \in V$.  Before starting recall that by Lemma \ref{lem: iterate gap summability}, and $\rho_n \geq \rho > 0$
    \begin{align}
        \sum_{n = 1}^{\infty} \frac{\rho_n}{2} \norm{\param_n -\param_{n-1}}^2 \leq \Delta_0 \quad \text{ and } \quad  \sum_{n = 1}^{\infty} \norm{\param_n - \param_{n-1}}^2 \leq \frac{2}{\rho}\Delta_0
    \end{align}
    almost surely. This implies
    \begin{align}
        \E \left[ \sum_{n = 1}^{\infty} \rho_n \norm{\param_n - \param_{n-1}}^2 \right] + \E \left[ \sum_{n = 1}^{\infty} \norm{\param_n - \param_{n-1}}^2 \right] < \infty.
    \end{align}

    Fix $v \in V$. For each $n$ we have $g_n^v \in \mathcal{S}_{L}(f^v, \param_{k^v(n)-1})$. Then using Proposition \ref{prop: surrogate properties}, the triangle inequality and Cauchy Schwartz
    \begin{align}
        |h_n^v(\param_n)| \leq \frac{L}{2}\norm{\param_n - \param_{k^v(n)-1}}^2 \leq \frac{L}{2}(n - k^v(n) + 1) \sum_{i = k^v(n)}^n \norm{\param_i - \param_{i-1}}^2.
    \label{eq: pf_gap_sum_prox_cs_bound}
    \end{align}
    Let $B_n = (n - k^v(n) + 1)\sum_{i = k^v(n)}^n \norm{\param_i - \param_{i-1}}^2$. We claim that $ \E \left[\sum_{n = 1}^{\infty} B_n \right] < \infty$. As in Lemma \ref{lem: conv_proof_surrogate_error_grad_sum} let $p^v(n) = \inf\{j > n : v_j = v\}$ be the next time strictly after time $n$ the sampling algorithm visits $v$. Exchanging the order of summation we have
    \begin{align}
         \E \left[\sum_{n = 1}^{\infty} (n - k^v(n) + 1) \sum_{i = k^v(n)}^n \norm{\param_i -\param_{i-1}}^2 \right] &= \E \left[ \sum_{i = 1}^{\infty} \norm{\param_{i} - \param_{i - 1}}^2 \sum_{n = i}^{p^v(i) - 1} (n - k^v(n) + 1)\right]\\
         &= \E \left[ \sum_{i = 1}^{\infty} \norm{\param_i - \param_{i-1}}^2 \sum_{n = i}^{\infty} (n - k^v(n)+1)\1(p^v(i) > n) \right].
    \end{align}
    The equality $\{p^v(i) > n\} = \{\tau_{i, v} \geq n - i + 1\}$ holds as both are equal to the event $\{v_j \neq v: i < j \leq n + 1\}$. Moreover, $k^v(n) = k^v(i)$ on $\{p^v(i) > n\}$ since there is no visit to $v$ between times $i$ and $n$. Therefore,
    \begin{align}
        &\E \left[ \sum_{i = 1}^{\infty} \norm{\param_i - \param_{i-1}}^2 \sum_{n = i}^{\infty} (n - k^v(n) + 1)\1(p^v(i) > n)\right]\\
        &= \E \left[ \sum_{i = 1}^{\infty} \norm{\param_i - \param_{i-1}}^2 \sum_{n = i}^{\infty} (n - k^v(i) + 1)\1(\tau_{i, v} \geq n - i + 1)\right]\\
        & = \E \left[ \sum_{i = 1}^{\infty} \norm{\param_i - \param_{i-1}}^2 \sum_{n = i}^{\infty} (n - k^v(i)+1)\P(\tau_{i, v} \geq n - i + 1|\cF_i)\right]
    \end{align}
    where the last line used $\param_i, \param_{i-1}$ and $k^v(i)$ are all measurable with respect to $\cF_i$. We have
    \begin{align}
        &\sum_{n = i}^{\infty}(n - k^v(i) + 1)\P(\tau_{i, v} \geq n - i + 1|\cF_i)\\
        &= \sum_{n = i}^{\infty} (n - i + 1)\P(\tau_{i, v} \geq n - i + 1|\cF_i) + (i - k^v(i))\sum_{n = i}^{\infty} \P(\tau_{i, v} \geq n - i + 1 |\cF_i)\\
        &= \E[\tau_{i, v}^2 |\cF_i] + (i - k^v(i)) \E[\tau_{i, v}|\cF_i]\\
        & \leq \mu_2 + (i - k^v(i))t_{\text{hit}}.
    \end{align}
    Finally,
    \begin{align}
        &\E \left[ \sum_{i = 1}^{\infty} \norm{\param_i - \param_{i-1}}^2 \sum_{n = i}^{\infty} (n - k^v(i)+1)\P(\tau_{i, v} \geq n - i + 1|\cF_i)\right]\\
        & \leq \mu_2 \E \left[ \sum_{i = 1}^{\infty} \norm{\param_i - \param_{i-1}}^2 \right] + t_{\text{hit}} \E \left[ \sum_{i = 1}^{\infty} (i - k^v(i))\norm{\param_i - \param_{i-1}}^2 \right]\\
        & \leq \mu_2 +  \E \left[ \sum_{i = 1}^{\infty} \norm{\param_i - \param_{i-1}}^2 \right] + t_{\text{hit}} \E \left[ \sum_{i = 1}^{\infty} \rho_i \norm{\param_i - \param_{i-1}}^2 \right] < \infty.
    \end{align}
    This shows $\E \left[ \sum_{n = 1}^{\infty} B_n \right] < \infty$. The proof is completed by using Fubini's Theorem and \eqref{eq: pf_gap_sum_prox_cs_bound} to conclude 
    \begin{align}
        \sum_{n = 1}^{\infty} \E[h_n^v(\param_n)] = \E \left[ \sum_{n =1}^{\infty} h_n^v(\param_n) \right] \leq \frac{L}{2} \E \left[ \sum_{n = 1}^{\infty} B_n \right] < \infty.
    \end{align}
    This completes the proof.
\end{proof}

We now prove Theorem \ref{thm: a.s. convergence} \ref{item: stationary_limit_points-dynamic}.

\begin{proof}[\textbf{Proof of Theorem \ref{thm: a.s. convergence}} \ref{item: stationary_limit_points-dynamic}]

Starting as in the proofs of Theorem \ref{thm: convergence rates to stationarity}
\begin{align}
    \E \left[ \sup_{\param \in \Param, \norm{\param - \param_n} \leq 1} - \nabla f(\param_n, \param - \param_n) \right] 
    &\leq \E \left[ \sup_{\param \in \Param, \norm{\param - \param_n} \leq 1} -\nabla \bar{g}_n(\param_n, \param - \param_n)\right] + \E \left[ \norm{\nabla \bar{h}_n(\param_n)} \right]\\
\end{align}
Using the same argument as in the proof of Theorem \ref{thm: convergence rates to stationarity} for Case \ref{case: prox_regularization-markovian},
\begin{align}
    \E \left[ \sup_{\param \in \Param, \norm{\param - \param_n} \leq 1} -\nabla \bar{g}_n(\param_n, \param - \param_n)\right] \leq \E[\rho_n\norm{\param_n - \param_{n-1}}] \leq \E [\sqrt{2\rho_n \delta_n}]
\end{align}
where $\delta_n = \bar{g}_{n-1}(\param_{n-1}) - \bar{g}_n(\param_n)$. By Cauchy-Schwartz and Lemma \ref{lem: uniform bound of reg parameter}
\begin{align}
    \E[\sqrt{2\rho_n \delta_n}] \leq \sqrt{2\E[\rho_n ]\E[\delta_n]} \leq C\sqrt{\E[\delta_n]}
\end{align}
for some $C > 0$ independent of $n$. 
By Jensen's inequality and Lemma \ref{lem: L-smooth_grad_upperbound}
\begin{align}
    \E[\norm{\nabla \bar{h}_n(\param_n)}] \leq \sqrt{\E[\norm{\nabla \bar{h}_n(\param_n)}^2]} \leq \sqrt{2L \E[\bar{h}_n(\param_n)]} 
\end{align}
We have
\begin{align}
    \sum_{n = 1}^{\infty} \E[\delta_n] = \sum_{n = 1}^{\infty} \E[\bar{g}_{n-1}(\param_{n-1})] - \E[\bar{g}_n(\param_n)] \leq \Delta_0 < \infty
\end{align}
so $\E[\delta_n] \to 0$ as $n \to \infty$. Also, $\sqrt{\E[\bar{h}_n(\param_n)]} \to 0$ by Lemma \ref{lem: gap sum dynamic prox reg}. Therefore
\begin{align}
    \limsup_{n \to \infty} \E \left[ \sup_{\param \in \Param, \norm{\param - \param_n} \leq 1} - \nabla f(\param_n, \param - \param_n) \right] 
    &\leq \limsup_{n \to \infty} \left( \E \left[ \sup_{\param \in \Param, \norm{\param - \param_n} \leq 1} -\nabla \bar{g}_n(\param_n, \param - \param_n)\right] + \E \left[ \norm{\nabla \bar{h}_n(\param_n)} \right] \right)\\
    & \leq \lim_{n \to \infty} \left( C \sqrt{\E[\delta_n]} + \sqrt{2L \E[\bar{h}_n(\param_n)]} \right) = 0.
\end{align}

We follow a similar approach to show that 
\begin{align}
    \E \left[ \left( \sup_{\param \in \Param, \norm{\param - \param_n}\leq 1} - \nabla f(\param_n, \param - \param_n) \right)^2 \right] \to 0.
\end{align}
Notice that the sub-optimality measure $\sup_{\param \in \Param, \norm{\param - \param_n} \leq 1} - \nabla f(\param_n, \param - \param_n)$ is always non-negative, since we can take $\param = \param_n$. Then from the inequality
\begin{align}
    \sup_{\param \in \Param, \norm{\param - \param_n} \leq 1} -\nabla f(\param_n, \param - \param_n) \leq \sup_{\param \in \Param, \norm{\param - \param_n} \leq 1} -\nabla \bar{g}_n(\param_n, \param - \param_n) + \norm{\nabla \bar{h}_n(\param_n)}
\end{align}
and  Cauchy-Schwartz we get
\begin{align}
    \E \left[ \left( \sup_{\param \in \Param, \norm{\param - \param_n} \leq 1} - \nabla f(\param_n, \param - \param_n) \right)^2 \right] \leq 2 \E \left[ \left( \sup_{\param \in \Param, \norm{\param - \param_n} \leq 1} - \nabla \bar{g}_n(\param_n, \param - \param_n) \right)^2 \right] + 2 \E \left[ \norm{\nabla \bar{h}_n(\param_n)}^2 \right].
\end{align}
Mimicking the proof above and using Lemma \ref{lem: uniform bound of reg parameter}
\begin{align}
\E \left[ \left( \sup_{\param \in \Param, \norm{\param - \param_n} \leq 1} - \nabla \bar{g}_n(\param_n, \param - \param_n) \right)^2 \right] \leq \E[\rho_n^2\norm{\param_n - \param_{n-1}}^2] &\leq 2\E[\rho_n \delta_n]\\
& \leq 2 \sqrt{\E[\rho_n^2] \E[\delta_n^2]}.\\
& \leq C \sqrt{\E[\delta_n^2]}.
\end{align}
Since $\sum_{n = 1}^{\infty} \delta_n \leq \Delta_0$
we have $\delta_n \to 0$ almost surely. Therefore, an application of the dominated convergence theorem shows $\E[\delta_n^2] \to 0$. Using Lemmas \ref{lem: gap sum dynamic prox reg} and \ref{lem: L-smooth_grad_upperbound} again to show $\E[\norm{\nabla \bar{h}_n(\param_n)}^2] \to 0$ completes the proof.

\end{proof}

\begin{remark}\label{rem: no convergence constant pr}
\normalfont
    The proof of Lemma \ref{lem: gap sum dynamic prox reg} demonstrates one of the main difficulties in proving asymptotic convergence for constant proximal regularization. In particular, our techniques require us to show that
    \begin{align}
        \E \left[ \sum_{i = 1}^{\infty} (i - k^v(i))\norm{\param_i - \param_{i-1}}^2 \right] < \infty.
        \label{eq: no convergence pr remark-i-k^v(i) sum}
    \end{align}
    The term $i - k^v(i)$ appears as the residual when we swap $n - k^v(i) + 1$ for $n - i + 1$ in order to show
    \begin{align*}
        \sum_{n = i}^{\infty} (n - k^v(i) + 1) \P(\tau_{i, v} \geq n - i + 1 |\cF_i) \leq \mu_2 + (i - k^v(i))t_{\text{hit}}
    \end{align*}
    To avoid this, an idea is to notice that $\tau_{i, v} \geq n - i + 1$ only if $\tau_{k^v(i), v} \geq n - k^v(i) + 1$ and instead compute
    \begin{align*}
        \sum_{n = i}^{\infty} (n - k^v(i) + 1)\P(\tau_{k^v(i), v} \geq n - k^v(i) + 1 |\cF_{k^v(i)}).
    \end{align*}
    The problem here is that $k^v(i)$ is \emph{not} a stopping time so, among other things, the $\sigma$-algebra $\cF_{k^v(i)}$ may not be well defined. Intuitively, $i - k^v(i)$ represents a gap in knowledge since we must wait until time $i$ to know the last time $v$ was visited.
    
    The use of dynamic proximal regularization bakes \eqref{eq: no convergence pr remark-i-k^v(i) sum} into the algorithm. Lemmas \ref{lem: uniform bound of reg parameter} and \ref{lem: iterate gap summability} suggests that it may be true with constant proximal regularization: if $(i - k^v(n))$ and $\norm{\param_i -\param_{i-1}}^2$ were independent then 
    \begin{align*}
        \E \left[ \sum_{i = 1}^{\infty} (i - k^v(i))\norm{\param_i - \param_{i-1}}^2 \right] &= \sum_{i = 1}^{\infty} \E[i - k^v(i)]\E[\norm{\param_i - \param_{i-1}}^2]\\
        & \leq C \sum_{i = 1}^{\infty} \E [\norm{\param_i - \param_{i-1}}^2] < \infty.
    \end{align*}
    However, $k^v(i)$, $\param_i$, and $\param_{i-1}$ are all determined by the behavior of the sampling process so we do not have this independence.

    We will see in the next subsection  that diminishing radius overcomes this issue by bounding the difference $\norm{\param_i - \param_{i-1}}^2$ by a deterministic quantity. 
\end{remark}

\subsection{The diminishing radius case \ref{case: diminishing_radius-markovian}}

Here we prove Theorem \ref{thm: a.s. convergence} \ref{item: stationary_limit_points}. Lemma \ref{lem: gap sum dr} is an analogue of Lemma \ref{lem: gap sum dynamic prox reg} for the diminishing radius case. The remaining argument is similar to that used in \cite{lyu2021BCD} and \cite{lyu2023srmm} to analyze block majorization-minimization and SMM with diminishing radius respectively. 

\begin{lemma}\label{lem: gap sum dr}
    Let $(\param_n)_{n \geq 0}$ be an output of Algorithm \ref{RMISO_DR}. Assume Case \ref{case: diminishing_radius-markovian}. Then almost surely
    \begin{align}
        \sum_{n = 1}^{\infty} \bar{h}_n(\param_n) < \infty.
    \end{align}
\end{lemma}

\begin{proof}
    The strategy here is nearly the same as in Lemma \ref{lem: gap sum dynamic prox reg} except that we use $\norm{\param_n - \param_{n-1}} \leq r_n$.

    Again, it is sufficient to show $\sum_{n = 1}^{\infty} h_n^v(\param_n) < \infty$ almost surely for each $v \in \mathcal{V}$. Fixing $v$ we have $g_n^v \in \mathcal{S}_L(f^v, \param_{k^v(n) - 1})$. Then Proposition \ref{prop: surrogate properties}, the triangle inequality, and Cauchy-Schwartz give us
    \begin{align}
        |h_n^v(\param_n)| \leq \frac{L}{2}\norm{\param_n - \param_{k^v(n)-1}}^2 &\leq \frac{L}{2}(n - k^v(n)+1)\sum_{i = k^v(n)}^n \norm{\param_i - \param_{i-1}}^2\\
        & \leq \frac{L}{2}(n - k^v(n) + 1) \sum_{i = 1}^n r_i^2.
    \end{align}
    Let $B_n = (n - k^v(n) + 1) \sum_{i = k^v(n)}^n r_i^2$. We mimic the proof of Lemma \ref{lem: gap sum dynamic prox reg} with $r_i^2$ in place of $\norm{\param_i - \param_{i-1}}^2$ to conclude
    \begin{align}
        \E \left[ \sum_{n = 1}^{\infty} B_n \right] \leq \mu_2 \sum_{i = 1}^{\infty} r_i^2 + t_{\text{hit}}\E \left[ \sum_{i = 1}^{\infty} \rho_i r_i^2 \right].
    \end{align}
    The first term on the right hand side is finite by Assumption \ref{assumption: sequences}. Moreover, by Lemma \ref{lem: uniform bound of reg parameter} and Fubini's Theorem
    \begin{align}
        \E \left[ \sum_{i = 1}^{\infty} \rho_i r_i^2 \right] = \sum_{i = 1}^{\infty} \E[\rho_i]r_i^2 \leq C \sum_{i = 1}^{\infty}r_i^2 < \infty.
    \end{align}
    Hence,
    \begin{align}
        \E \left[ \sum_{n = 1}^{\infty} h_n^v(\param_n) \right] \leq \frac{L}{2} \E \left[ \sum_{n = 1}^{\infty} B_n \right] < \infty.
    \end{align}
    It then follows that $\sum_{n = 1}^{\infty} h_n^v(\param_n)$ is finite almost surely.

\end{proof}

\begin{prop}\label{prop: stationary_subsequence}
    Assume Case \ref{case: diminishing_radius-markovian}. Suppose there exists a sequence $(n_k)_{k \geq 1}$ such that almost surely either
    \begin{align}
        \sum_{k = 1}^{\infty} \norm{\param_{n_k+1} - \param_{n_k}} = \infty \quad \text{ or } \quad \liminf_{k \to \infty} \left| \left \langle \nabla \bar{g}_{n_k + 1}(\param_{n_k}), \frac{\param_{n_k+1} - \param_{n_k}}{\norm{\param_{n_k+1} - \param_{n_k}}} \right \rangle \right| = 0.
        \label{eq: prop-stationary_subsequence-conditions}
    \end{align}
    Then there exists a further subsequence $(m_k)_{k \geq 1}$ of $(n_k)_{k \geq 1}$ such that $\param_{\infty}:= \lim_{k \to \infty} \param_{m_k}$ exists almost surely and $\param_{\infty}$ is a stationary point of $f$ over $\Param$. 
\end{prop}

\begin{proof}
    By Proposition \ref{prop: surogate_gradient_step_sum},
    \begin{align}
        \sum_{k = 1}^{\infty} \norm{\param_{n_k + 1} - \param_{n_k}} \left| \left \langle \nabla \bar{g}_{n_k + 1}(\param_{n_k}), \frac{\param_{n_k+1} - \param_{n_k}}{\norm{\param_{n_k+1} - \param_{n_k}}} \right \rangle \right| < \infty \quad \text{ a.s. }
    \end{align}
    Therefore, the former condition implies the latter almost surely. So, it suffices to show the the latter condition implies the assertion. Assume the latter condition in \eqref{eq: prop-stationary_subsequence-conditions} and let $(m_k)_{k \geq 1}$ be a subsequence of $(n_k)_{k \geq 1}$, satisfying
    \begin{align}
        \lim_{k \to \infty}\left| \left \langle \nabla \bar{g}_{m_k + 1}(\param_{m_k}), \frac{\param_{m_k+1} - \param_{m_k}}{\norm{\param_{m_k+1} - \param_{m_k}}} \right \rangle \right| = 0.
    \end{align}
    Since $\norm{\param_{m_k + 1} - \param_{m_k}} \leq r_{m_k}$, it follows that
    \begin{align}
        \lim_{k \to \infty} \frac{\norm{\param_{m_{k+1}} - \param_{m_k}}}{b_{m_k + 1}} \left| \left \langle \nabla \bar{g}_{m_k+1}(\param_{m_k}), \frac{\param_{m_k+1} - \param_{m_k}}{\norm{\param_{m_k+1} - \param_{m_k}}} \right \rangle \right| = 0.
        \label{eq: prop-stationary_subsequnce-first order conv}
    \end{align} 
    where $b_n = \min\{1, r_n\}$.
    If $\param_{\infty}$ is not a stationary point of $f$ over $\Param$, then we may find $\param^{\star} \in \Param$ with $\norm{\param^{\star} - \param_{\infty}} \leq 1$ and $\eps > 0$ so that 
    \begin{align}
        \langle \nabla f(\param_{\infty}), \param^{\star}  - \param_{\infty} \rangle \leq -\eps < 0.
    \end{align}
    On the other hand by the triangle inequality and Cauchy-Schwartz
    \begin{align}
        &|\langle \nabla \bar{g}_{m_k}(\param_{m_k}), \param^{\star} - \param_{m_k} \rangle - \langle \nabla f(\param_{\infty}), \param^{\star} - \param_{\infty} \rangle|\\
        &= |\langle \nabla \bar{g}_{m_k}(\param_{m_k}) - \nabla f(\param_{m_k}), \param^{\star} - \param_{m_k} \rangle + \langle \nabla f(\param_{m_k}) - \nabla f(\param_{\infty}), \param^{\star} - \param_{m_k}\rangle + \langle \nabla f(\param_{\infty}), \param_{\infty} - \param_{m_k} \rangle |\\
        & \leq \norm{\nabla \bar{h}_{m_k}(\param_{m_k})}\norm{\param^{\star} - \param_{m_k}} + \norm{\nabla f(\param_{m_k}) - \nabla f(\param_{\infty})}\norm{\param^{\star} - \param_{\infty}} + \norm{\nabla f(\param_{\infty})}\norm{\param_{\infty} - \param_{m_k}}
    \end{align}
    Since $(\param_{m_k})_{k \geq 1}$ converges, $\sup_{k} \norm{\param^{\star} - \param_{m_k}} \leq M$ for some $M < \infty$. Furthermore, 
    \begin{align}
        \sum_{n = 1}^{\infty} \norm{\nabla \bar{h}_n(\param_n)}^2 \leq 2L \sum_{n = 1}^{\infty} \bar{h}_n(\param_n) < \infty
    \end{align}
    by Lemmas \ref{lem: gap sum dr} and \ref{lem: L-smooth_grad_upperbound} so
    $\norm{\nabla \bar{h}_{m_k}(\param_{m_k})} \to 0$ almost surely as $k \to \infty$. This, together with continuity of $\nabla f$ and $\param_{m_k} \to \param_{\infty}$, shows that right hand side above tends to zero as $k \to \infty$. Then we can choose $K$ sufficiently large so that 
    \begin{align}
        \langle \nabla \bar{g}_{m_k}(\param_{m_k}), \param^{\star} - \param_{m_k} \rangle \leq -\frac{\eps}{2}
    \end{align}
    for $k \geq K$.
     Recall that $\norm{\param_n - \param_{n-1}} \leq r_n$ and $r_n = o(1)$. Applying Lemma  \ref{lem: surrogate_grad_inf_lowerbound} we get 
    \begin{multline}
        \frac{\norm{\param_n - \param_{n-1}}}{b_n}\left \langle \nabla \bar{g}_n(\param_{n-1}), \frac{\param_n - \param_{n-1}}{\norm{\param_n - \param_{n-1}}} \right \rangle\\
        \leq \inf_{\param \in \Param, \norm{\param - \param_{n-1}} \leq 1} \left \langle \nabla \bar{g}_{n-1}(\param_{n-1}), \param - \param_{n-1}\right \rangle + \norm{\nabla h_{n-1}^{v_n}(\param_{n-1})} + Lr_n.
    \end{multline}
    It then follows that for sufficiently large $k$ 
    \begin{align}
        &\frac{\norm{\param_{m_k + 1} - \param_{m_k}}}{b_{m_k+1}}\left \langle \nabla \bar{g}_{m_k+1}(\param_{m_k}), \frac{\param_{m_k + 1} - \param_{m_k}}{\norm{\param_{m_k + 1} - \param_{m_k}}} \right \rangle\\
        &\leq \inf_{\param \in \Param, \norm{\param - \param_{m_k}} \leq 1} \left \langle \nabla \bar{g}_{m_k}(\param_{m_k}), \param - \param_{m_k}\right \rangle + \norm{\nabla h_{m_k}^{v_{m_k+1}}(\param_{m_k})} + Lr_{m_k} \\
        & \leq  \left \langle \nabla \bar{g}_{m_k}(\param_{m_k}), \param^{\star} - \param_{m_k} \right \rangle  + \norm{\nabla h_{m_k}^{v_{m_k+1}}(\param_{m_k})} + r_{m_k}\\
        & \leq -\frac{\eps}{2} + \norm{\nabla h_{m_k}^{v_{m_k + 1}}(\param_{m_k}} + r_{m_k}
    \end{align}
    Recall Lemma \ref{lem:single gap convergence} which shows 
    \begin{align}
         \sum_{n=1}^{\infty} h_n^{v_{n+1}}(\param_n) < \infty
    \end{align}
    almost surely.
    Therefore, since $h_n^{v_n}$ is non-negative, $\sqrt{h_n^{v_{n+1}}(\param_n)} \to 0$ almost surely as $n \to \infty$. Moreover, by Lemma \ref{lem: L-smooth_grad_upperbound}, $\norm{\nabla h_n^{v_{n+1}}(\param_n)} \leq \sqrt{2Lh_n^{v_{n+1}}(\param_n)}$. So letting $k \to \infty$ shows
    \begin{align}
        \limsup_{k \to \infty} \frac{\norm{\param_{m_k + 1} - \param_{m_k}}}{b_{m_k + 1}} \left \langle \nabla \bar{g}_{m_k}(\param_{m_k+1}), \frac{\param_{m_k + 1} - \param_{m_k}}{\norm{\param_{m_k+1} - \param_{m_k}} } \right \rangle \leq -\frac{\eps}{2}
    \end{align}
    contradicting \eqref{eq: prop-stationary_subsequnce-first order conv}. 
\end{proof}

Recall that under Algorithm \ref{RMISO_DR}, the one step parameter difference $\norm{\param_n - \param_{n-1}}$ is at most $r_n$. For each $n \geq 1$ we say that $\param_n$ is a \emph{long point} if $\norm{\param_n - \param_{n-1}} < r_n$ and a \emph{short point} if $\norm{\param_n - \param_{n-1}} = r_n$. The next proposition shows that if $\param_n$ is a long point, then $\param_n$ is obtained by directly minimizing $\bar{g}_n$ over the full parameter space $\Param$. It is here the we crucially use the convexity of $\bar{g}_n$ from Definition \ref{def: surrogates}. 

\begin{prop}\label{prop: long_point_global_min}
    For $n \geq 1$, suppose that $\param_n \in \argmin_{\param \in \Param \cap B_{r_n}(\param_{n-1})} \bar{g}_n(\param)$ and that $\norm{\param_n - \param_{n-1}} < r_n$. Then $\param_n \in \argmin_{\param \in \Param} \bar{g}_n(\param)$.  
\end{prop}

\begin{proof}
    By Definition \ref{def: surrogates}, $\bar{g}_n$ is convex. Thus, it suffices to verify the first order stationarity condition
    \begin{align}
        \inf_{\param \in \Param} \langle \nabla \bar{g}_n(\param_n), \param - \param_n \rangle \geq 0
    \end{align}
     to conclude $\param_n \in \argmin_{\param \in \Param} \bar{g}_n (\param)$. To this end, assume the conclusion is false. Then there is $\param^{\star} \in \Param$ with $\langle \nabla \bar{g}_n(\param_n), \param^{\star} - \param_n \rangle < 0$. Moreover, as $\param_n$ is obtained by minimizing $\bar{g}_n$ over $\Param \cap B_{r_n}(\param_{n-1})$ we must have $\norm{\param^{\star} - \param_{n-1}} > r_n$. As we are assuming $\norm{\param_n - \param_{n-1}} < r_n$, there is $\alpha \in (0, 1)$ so that $\norm{\param_n - \param_{n-1}} = \alpha r_n$. Notice that 
    \begin{align}
        \norm{\param^{\star} - \param_n} \geq \norm{\param^{\star} - \param_{n-1}} - \norm{\param_{n} - \param_{n-1}} > (1 - \alpha)r_n. 
    \end{align}
    Hence if we set $a = \frac{(1 - \alpha)r_n}{\norm{\param^{\star} - \param_n}}$ then $a \in (0, 1)$. So, the convexity of $\Param$ implies that $\tilde{\param} := a(\param^{\star} - \param_n) + \param_n \in \Param$. Furthermore, 
    \begin{align}
        \norm{\tilde{\param} - \param_{n-1}} \leq a\norm{\param^{\star} - \param_n} + \norm{\param_n - \param_{n-1}} = (1 - \alpha)r_n + \alpha r_n = r_n
    \end{align}
    and 
    \begin{align}
        \langle \nabla \bar{g}_n(\param_n), \tilde{\param} - \param_n \rangle = a \langle \nabla \bar{g}_n(\param_n), \param^{\star} - \param_n \rangle < 0.
    \end{align}
    This contradicts $\param_n \in \argmin_{\param \in \Param \cap B_{r_n}(\param_{n-1})} \bar{g}_n(\param)$ and completes the proof. 
\end{proof}

\begin{prop}\label{prop: long_point_limit_stationary}
    Assume the Case \ref{case: diminishing_radius-markovian}. If $(\param_{n_k})_{k \geq 1}$ is a sequence consisting of long points such that $\param_{\infty} := \lim_{k \to \infty} \param_{n_k}$ exist almost surely, then $\param_{\infty}$ is a stationary point of $f$ over $\Param$.
\end{prop}

\begin{proof}
    By the assumption that $\param_{n_k}$ is a long point and Proposition \ref{prop: long_point_global_min} we have $\param_{n_k} \in \argmin_{\param \in \Param} \bar{g}_{n_k}(\param)$. Therefore, for any $\param \in \Param$,
    \begin{align}
        \langle \nabla \bar{g}_{n_k}(\param_{n_k}), \param - \param_{n_k} \rangle \geq 0.
    \end{align}
    We then notice that
    \begin{align}
        \langle \nabla f(\param_{n_k}), \param - \param_{m_k} \rangle = \langle \nabla \bar{g}_{n_k}(\param_{n_k}), \param - \param_{n_k} \rangle - \langle \nabla \bar{h}_{n_k}(\param_{n_k}), \param - \param_{n_k} \rangle. 
    \end{align}
    By Lemmas \ref{lem: L-smooth_grad_upperbound} and \ref{lem: gap sum dr}, $\norm{\nabla \bar{h}_{n_k}(\param_{n_k})}^2 \leq 2L\bar{h}_{n_k}(\param_{n_k}) \to 0$ almost surely as $k \to \infty$. Therefore, by taking limits we get
    \begin{align}
        \langle \nabla f(\param_{\infty}), \param - \param_{\infty} \rangle \geq 0.
    \end{align}
    Since this holds for all $\param \in \Param$, 
    \begin{align}
        \sup_{\param \in \Param, \norm{\param - \param_{\infty}} \leq 1} \langle -\nabla f(\param_{\infty}), \param - \param_{\infty} \rangle \leq 0
    \end{align}
    which means that $\param_{\infty}$ is a stationary point of $f$ over $\Param$. 
\end{proof}

\begin{prop}\label{prop: non_stationary_nhood}
Suppose there exists a sub-sequence $(\param_{n_k})_{k \geq 1}$ such that $\lim_{k \to \infty} \param_{n_k} = \param_{\infty}$ exists almost surely and that $\param_{\infty}$ is not a stationary point of $f$ over $\Param$. Then there is $\eps > 0$ such that the $\eps$-neighborhood $B_{\eps}(\param_{\infty})$ has the following properties:
\begin{itemize}
    \item[(a)] $B_{\eps}(\param_{\infty})$ does not contain any stationary points of $f$ over $\Param$. 
    \item[(b)] There are infinitely many $n$ for which $\param_n$ is outside of $B_{\eps}(\param_{\infty})$. 
\end{itemize}
\end{prop}

\begin{proof}
    We first show that there exists $\eps > 0$ so that $B_{\eps}(\param_{\infty})$ does not contain any long points. Suppose for contradiction that for each $\eps > 0$, there is a long point in $B_{\eps}(\param_{\infty})$. Then one may construct a sequence of long points converging to $\param_{\infty}$. But then by Proposition \ref{prop: long_point_limit_stationary}, $\param_{\infty}$ is a stationary point for $f$ over $\Param$, a contradiction. 
    
    Next we show that there exists $\eps$ so that $B_{\eps}(\param_{\infty})$ satisfies (a). In fact, suppose not. Then we can find a sequence of stationary points $(\param_{\infty, k})_{k \geq 1}$ converging to $\param_{\infty}$. But then by continuity of $\nabla f$,
    \begin{align}
        \langle \nabla f(\param_{\infty}), \param - \param_{\infty} \rangle = \lim_{k \to \infty} \langle \nabla f(\param_{k, \infty}), \param - \param_{k, \infty} \rangle \geq 0
    \end{align}
    for any $\param \in \Param$. Then $\param_{\infty}$ is a stationary point of $f$ over $\Param$, contradicting our assumptions. 

    Now let $\eps > 0$ be such that $B_{\eps}(\param_{\infty})$ does not contain any long points and satisfies (a). We will show that $B_{\eps/2}(\param_{\infty})$ satisfies (b) and thus $B_{\eps/2}(\param_{\infty})$ satisfies both (a) and (b) as desired. Aiming for a contradiction, suppose there are only finitely many $n$ for which $\param_{n}$ is outside $B_{\eps/2}(\param_{\infty})$. Then there exists $N$ so that $\param_{n} \in B_{\eps/2}(\param_{\infty})$ for all $n \geq N$. Then $\param_n$ is a short point for each $n \geq N$ so $\norm{\param_n - \param_{n-1}} = r_n$ for all $n \geq N$. This, in turn, implies that $\sum_{n = 1}^{\infty} \norm{\param_n - \param_{n-1}} = \infty$. By Proposition \ref{prop: stationary_subsequence}, there exists a subsequence $(\param_{n_k})_{k \geq 1}$ such that $\param_{\infty}' = \lim_{k \to \infty} \param_{n_k}$ exists and is stationary for $f$. But since $\param_{\infty}' \in B_{\eps}(\param_{\infty})$, this contradicts (a). The proof is complete.
\end{proof}

We now prove Theorem \ref{thm: a.s. convergence} \ref{item: stationary_limit_points}.

\begin{proof}[\textbf{Proof of Theorem \ref{thm: a.s. convergence} \ref{item: stationary_limit_points}}]

Suppose for contradiction that there exists a non-stationary limit point $\param_{\infty}$ of $(\param_n)_{n \geq 0}$. By Proposition \ref{prop: non_stationary_nhood}, there is $\eps > 0$ so that $B_{\eps}(\param_{\infty})$ satisfies the conditions (a) and (b). Choose $N$ large enough so that $r_n \leq \frac{\eps}{4}$ for $n \geq N$. We call an integer interval $I := [\ell, \ell')$ a \emph{crossing} if $\param_{\ell} \in B_{\eps/3}(\param_{\infty})$, $\param_{\ell'} \notin B_{2\eps/3}(\param_{\infty})$, and no proper subset of $I$ satisfies both of these conditions. By definition, two distinct crossings have empty intersection. Fix a crossing $I = [\ell, \ell')$. It follows by the triangle inequality,
\begin{align}
    \sum_{n = \ell}^{\ell'-1} \norm{\param_{n+1} - \param_n} \geq \norm{\param_{\ell'} - \param_{\ell}} \geq \eps/3.
    \label{eqn: sum_over_crossing}
\end{align}
Note that since $\param_{\infty}$ is a limit point of $(\param_n)_{n \geq 0}$, we have $\param_n \in B_{\eps/3}(\param_{\infty})$ infinitely often. In addition, by condition (b) of Proposition \ref{prop: non_stationary_nhood}, $\param_n$ also exits $B_{\eps}(\param_{\infty})$ infinitely often. Therefore, there must be infinitely many crossings. Let $n_k$ be the $k$-th smallest integer that appears in some crossing, noting importantly that $\param_{n_k} \in B_{2\eps/3}$ for $k \geq 1$. Then $n_k \to \infty$ as $k \to \infty$ and by \eqref{eqn: sum_over_crossing},
\begin{align}
    \sum_{k = 1}^{\infty} \norm{\param_{n_k+1} - \param_{n_k}} \geq (\# \text{ of crossings} ) \frac{\eps}{3} = \infty.
\end{align}
Then by Proposition \ref{prop: stationary_subsequence}, there is a further subsequence $(\param_{m_k})_{k \geq 1}$ of $(\param_{n_k})_{k \geq 1}$ so that $\param_{\infty}' = \lim_{k \to \infty} \param_{m_k}$ exists and is stationary. However, since $\param_{n_k} \in B_{2\eps/3}(\param_{\infty})$ the stationary point $\param_{\infty}'$ is in $B_{\eps}(\param_{\infty})$. This contradicts property (a) of Proposition \ref{prop: non_stationary_nhood} which shows the assertion.
\end{proof}

\subsection{Details for numerical experiments}\label{sec: experiment details}

\subsubsection{Distributed Nonnegative Matrix Factorization}\label{sec: nmf experiments-details}

The MNIST samples $X_v$ at each node were formed by concatenating a collection of images $\{X_i\}_{i = 1}^k \subset \R_{+}^{28 \times 28}$ along the horizontal axis so that $X_v \in \R_{+}^{28 \times 28k}$. We selected 5000 images from the full dataset at random and divided them into groups based on class label. New nodes were formed by adding batches of 100 images from each group until fewer than 100 images remained. Then a final node was added for the remaining images.   

We include here a list of hyperparamters used for the NMF experiments.

For AdaGrad we used constant step size parameter $\eta=0.5$. For both RMISO-DPR and RMISO-CPR we set $\rho=2500$ for the random walk and $\rho=50$ for cyclic sampling. For the diminishing radius version RMISO-DR we set $r_n = \frac{1}{\sqrt{n}\log(n+1)}$. 

Figure \ref{subfig: nmf time} displays the results of these experiments vs compute time.

\begin{figure}[h]
\centering
        \begin{subfigure}{0.4\textwidth}
        \includegraphics[width=\textwidth]{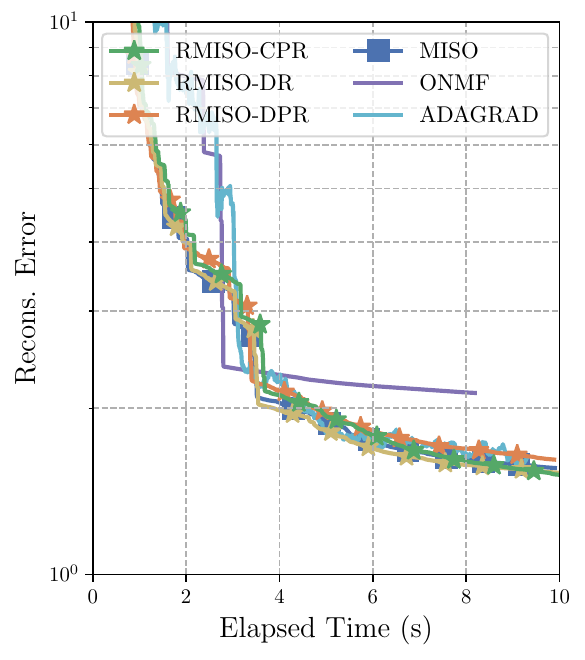}
        \caption{Random Walk}
        \end{subfigure}
        \begin{subfigure}{0.4\textwidth}
            \includegraphics[width=\textwidth]{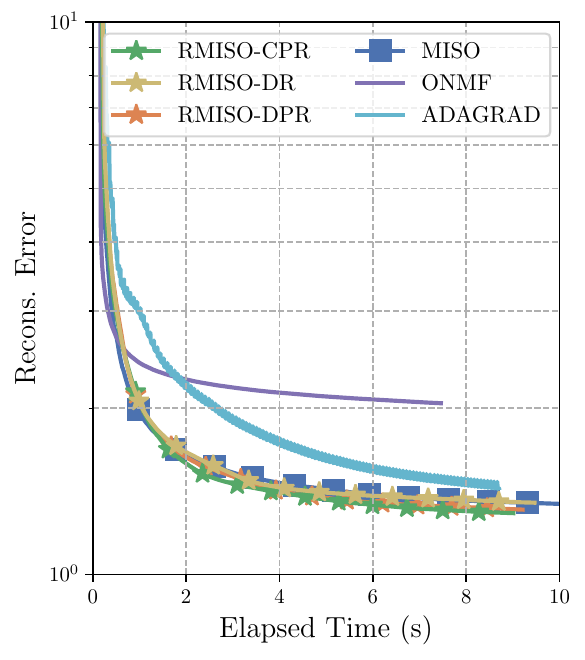}
            \caption{Cyclic}
        \end{subfigure}
        \caption{Plot of reconstruction error against compute time for NMF using two sampling algorithms. Results show the performance of algorithms RMISO, MISO (Algorithm \ref{RMISO} with $\rho_n = 0$), ONMF, and AdaGrad in factorizing a collection of MNIST \cite{mnist} data matrices. }
        \label{subfig: nmf time}
\end{figure}

\subsubsection{Logistic Regression with nonconvex regularization}
\label{sec: binary classificatin experiment-details}

The hyperparameters for the logistic regression experiments were chosen as follows. For MCSAG and RMISO/MISO we took $L = 2/5$. The random walk on the complete graph has $t_{\text{hit}} = O(|\mathcal{V}|)$ while $t_{\text{hit}} = O(|\mathcal{V}|^2)$ for the lonely graph but $t_{\odot} = O(|\mathcal{V}|)$ for both. Accordingly for MCSAG we set the hitting time parameter in the step size $t_{\text{hit}} = 50$ for the complete graph and $t_{\text{hit}} = 2500$ for the lonely graph. For RMISO we set $\rho = 50$ for both the constant proximal regularization version and the dynamic proximal regularization version. We ran SGD with a decaying step size of the form $\alpha_n = \frac{\alpha}{n^{\gamma}}$ where $\alpha = 0.1$ and $\gamma = 0.5$. For SGD-HB and AdaGrad we used step sizes $\alpha=0.05$ and SGD-HB momentum parameter $\beta = 0.9$.   

Figure \ref{fig: binary classification time} shows the results of our experiments plotted vs compute time. 

\begin{figure}[h]
\centering
        \begin{subfigure}{0.4\textwidth}
        \includegraphics[width=\textwidth]{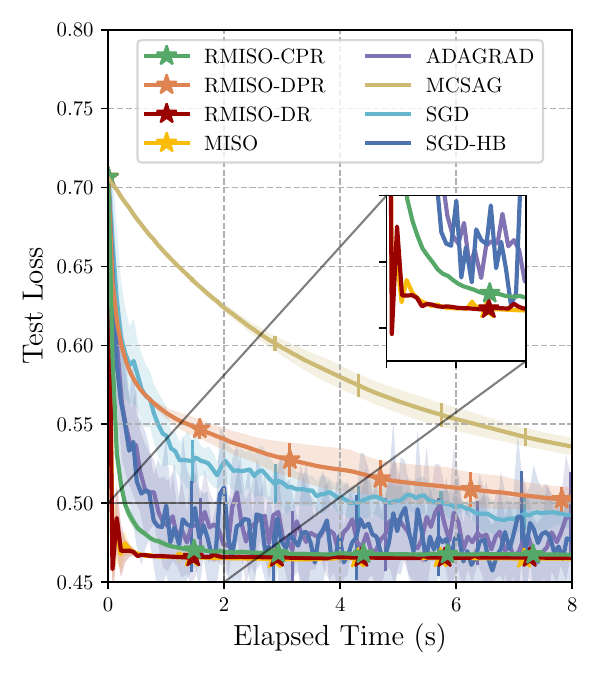}
        \caption{Lonely graph}
        \end{subfigure}
        \begin{subfigure}{0.4\textwidth}
            \includegraphics[width=\textwidth]{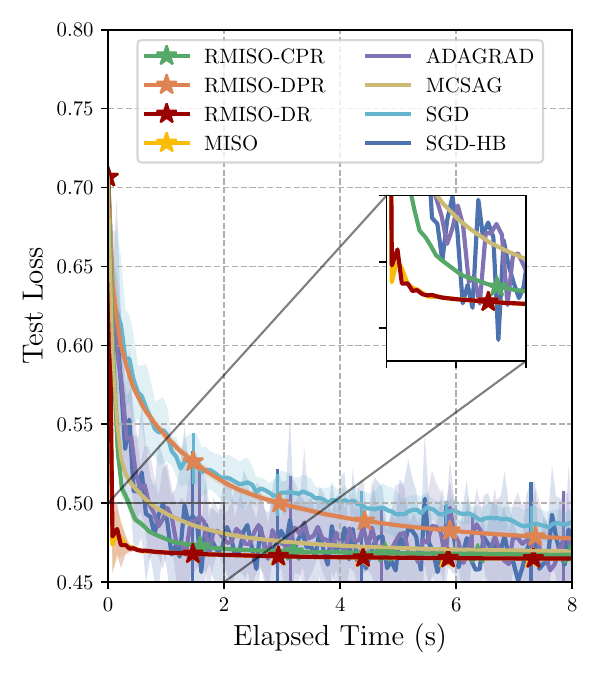}
            \caption{Complete graph}
        \end{subfigure}
    \caption{Plot of objective loss and standard deviation vs compute time for \texttt{a9a} for two graph topologies and various optimization algorithms- RMISO, MISO (Algorithm \ref{RMISO} with $\rho_n =0$), AdaGrad, MCSAG, SGD, Adam, and SGD-HB}
    \label{fig: binary classification time}
\end{figure}

\section{Auxiliary Lemmas}\label{sec: Auxiliary Lemmas}

\begin{lemma}\label{lem: L-smooth_surrogate}
Let $f : \R^p \to \R$ be a continuously differentiable function with $L$-Lipschitz continuous gradient. Then for all $\param, \param' \in \R^p$,
\begin{align}
    |f(\param') - f(\param) - \langle \nabla f(\param), \param' - \param \rangle | \leq \frac{L}{2}\norm{\param - \param'}^2.
\end{align}
\end{lemma}

\begin{proof}
    This is a classical lemma. See \cite{nesterov2003} Lemma 1.2.3.
\end{proof}

\begin{lemma}\label{lem: L-smooth_grad_upperbound}
Let $f : \R^p \to [0, \infty)$ be a continuously differentiable function with $L$-Lipschitz continuous gradient. Then for all $\param \in \R^p$, it holds $\norm{\nabla f(\param)}^2 \leq 2Lf(\param)$. 
\end{lemma}
\begin{proof}
    Fix $\param \in \R^p$. By Lemma \ref{lem: L-smooth_surrogate} we have
    \begin{align}
       \inf_{\param' \in \R^p} f(\param') \leq \inf_{\param' \in \R^p} \left\{f(\param) + \langle \nabla f(\param), \param' - \param \rangle + \frac{L}{2} \norm{\param' - \param}^2\right\}.
    \end{align}
    It is easy to compute that 
    \begin{align}
        \inf_{\param' \in \R^p} \left\{f(\param) + \langle \nabla f(\param), \param' - \param \rangle + \frac{L}{2} \norm{\param' - \param}^2\right\} = f(\param) - \frac{1}{2L}\norm{\nabla f(\param)}_2^2.
    \end{align}
    Therefore
    \begin{align}
        \norm{\nabla f(\param)}_2^2 \leq 2L(f(\param) - \inf_{\param' \in \R^p}f(\param')) \leq 2Lf(\param)
    \end{align}
    since $\inf_{\param' \in \R^p}f(\param')\geq 0$.
\end{proof}

\section{Examples of Surrogate Functions}\label{sec: examples of Surrogate Functions}

\begin{example}[Proximal surrogates for $L$-smooth functions]
\normalfont
Suppose $f$ is continuously differentiable with $L$-Lipschitz continuous gradients. Then $f$ is $L$-weakly convex, meaning $\param \mapsto f(\param) + \frac{L}{2}\norm{\param}^2$ is convex (see \cite{lyu2023srmm} Lemma C.2). For each $\gamma \geq L$, the following function belongs to $\mathcal{S}_{L + \gamma}(f, \param^*)$:
\begin{align}
    g: \param \mapsto f(\param) + \frac{\gamma}{2}\norm{\param - \param^*}^2
\end{align}
Indeed, $g \geq f$, $g(\param^*) = f(\param^*)$, $\nabla h(\param^*) = 0$, and $\nabla h$ is $(L + \rho)$ Lipschitz. Minimizing the above function over $\Param$ is equivalent to applying a proximal mapping of $f$ where the resulting estimate is denoted $\text{prox}_{f/\rho}(\param^*)$ (see \cite{parikh2014prox, davis2019stochastic}).
    
\end{example}

\begin{example}[Prox-linear surrogates]
\label{ex: prox_linear surrogates}
\normalfont
If $f$ is $L$-smooth, then the following quadratic function $g$ belongs to $\mathcal{S}_{2L}(f, \param^*)$:
\begin{align}
    g : \param \mapsto f(\param^*) + \langle \nabla f(\param^*), \param - \param^* \rangle + \frac{L}{2}\norm{\param - \param^*}^2.
\end{align}
Indeed, $g(\param^*) = f(\param^*)$, $\nabla g(\param^*) = \nabla f(\param^*)$. Moreover,
\begin{align}
    \norm{\nabla h(\param) - \nabla h(\param')} = \norm{\nabla f(\param') - \nabla f(\param) +  L(\param - \param')} \leq 2L\norm{\param - \param'}
\end{align}
since $f$ is $L$-smooth.
\end{example}

\begin{example}[Prox-linear surrogates]
    \normalfont
    Suppose $f = f_1 + f_2$ where $f_1$ is differentiable with $L$-Lipschitz gradient and $f_2$ is convex over $\Param$. Then the following function $g$ belongs to $\mathcal{S}_{2L}(f, \param^*)$:
    \begin{align}
        g : \param \mapsto f_1(\param^*) + \langle \nabla f_1(\param^*), \param - \param^* \rangle + \frac{L}{2}\norm{\param - \param^*}^2 + f_2(\param).
    \end{align}
    Minimizing $g$ over $\Param$ amounts to performing a proximal gradient step \cite{beck2009interative, nesterov2013composite}.
\end{example}

\begin{example}[DC programming surrogates]
\normalfont
Suppose $f = f_1 + f_2$ where $f_1$ is convex and $f_2$ is concave and differentiable with $L_2$-Lipschitz gradient over $\Param$. One can also write $f = f_1 - (-f_2)$ which is the difference of convex (DC) functions $f_1$ and $-f_2$. Then the following function $g$ belongs to $\mathcal{S}_{2L}(f, \param^*)$:
\begin{align}
    g: \param \mapsto f_1(\param) + f_2(\param^*) + \langle \nabla f_2(\param^*), \param - \param^* \rangle.
\end{align}
Such surrogates are important in DC programming \cite{horst1999dcprogramming}.
    
\end{example}

\begin{example}[Variational Surrogates]
\label{ex: variational surrogates}
\normalfont
Let $f: \R^p 
\times \R^q \to \R$ be a two-block multi-convex function and let $\Param_1 \subseteq \R^p$ and $\Param_2 \subseteq \R^q$ be two convex sets. Define a function $f_* : \inf_{H \in \Param_2} f(\param, H)$. Then for each $\param^* \in \Param$, the following function 
\begin{align}
    g : \param \to f(\param, H^*), \quad H^* \in \argmin_{H \in \Param_2} f(\param^*, H)
\end{align}
is convex over $\Param_1$ and satisfies $g \geq f$ and $g(\param^*) = f(\param^*)$. Further, assume that 
\begin{itemize}
    \item[(i)] $\param \mapsto f(\param, H)$ is differentiable for all $H \in \Param_2$ and $\param \to \nabla_{\param}f(\param, H)$ is $L'$-Lipschitz for all $H \in \Param_2$;
    \item[(ii)] $H \mapsto \nabla_{\param}f(\param, H)$ is $L$-Lipschitz for all $\param \in \Param_1$;
\end{itemize}
Then $g$ belongs to $\mathcal{S}_{L}(f_*, \param^*)$ for some $L'' > 0$. When $f$ is jointly convex, then $f_{*}$ is also convex and we can choose $L'' = L$. 

\end{example}

\section{Matrix factorization algorithms}\label{sec: nmf algorithm statements}

Here we formally state the non-negative matrix factorization algorithms derived in Section \ref{sec: nmf}. They may be compared to the celebrated online nonnegative matrix factorization algorithm in \cite{marialONMF} which is a special case of SMM. 

With surrogates $g_n(W)$ as defined in Section \ref{sec: nmf}, one can show that minimizing the averaged surrogate $\bar{g}_n(W) = \frac{1}{|\mathcal{V}|} \sum_{v \in V} g_n^v(W)$ is equivalent to minimizing 
\begin{align}
\text{tr}(W A_n W^T) - 2\text{tr}(WB_n), 
\end{align}
with $A_n$ and $B_n$, 
defined recursively as 
\begin{align}
    A_n &:= A_{n-1} + \frac{1}{|\mathcal{V}|} \Big[ H_n^{v_n} (H_n^{v_n})^T - H_{n-1}^{v_n} (H_{n-1}^{v_n})^T \Big]\\
    B_n &:= B_{n-1} + \frac{1}{|\mathcal{V}|} \Big[H_n^{v_n}X_v^T - H_{n-1}^{v_n}X_v^T\Big].
\end{align}
With this, we state the full algorithms below. 

\begin{algorithm}[H]
    \small
    \caption{Distributed Matrix Factorization with Proximal Regularization}
    \label{DNMF}
    \begin{algorithmic}
    \STATE \textbf{Input:} $(X^v)_{v \in \mathcal{V}}$ (Data matrices in $\R^{p \times d}$); $W_0 \in \Param_W$ (initial dictionary); $N$ (number of iterations); $\rho > 0$ (regularization parameter)
    \STATE \textbf{Option: } $Regularization \in \{Dynamic, Constant\}$
    \STATE Compute initial codes $H_0^v \in \argmin_{H \in \Param_H^{\mathcal{V}}} \frac{1}{2}\norm{X^v - W_0H}_F^2 + \alpha \norm{H}_1$ for each $v \in \mathcal{V}$
    \FOR {$n = 1$ {\bfseries to} $N$}
    \STATE sample an index $v_n$
    \STATE update $H_n^v \in \argmin_{H \in \Param_H^{\mathcal{V}}} \frac{1}{2} \norm{X_v - W_{n-1}H}_F^2 + \alpha \norm{H}_1$; $H_n^v = H_{n-1}^v$ for $v \neq v_n$
    \STATE $A_n \leftarrow A_{n-1} + \frac{1}{|\mathcal{V}|} \left[ H_n^{v_n} (H_n^{v_n})^T - H_{n-1}^{v_n}(H_{n-1}^{v_n})^T \right]$
    \STATE $B_n \leftarrow B_{n-1} + \frac{1}{|\mathcal{V}|} \left[ H_n^{v_n} (X^{v_n})^T - H_{n-1}^{v_n} (X^{v_n})^T \right]$
    \IF{$Regularization=Dynamic$}
    \STATE $\rho_n \leftarrow \rho + \max_{v \in \mathcal{V}} (n - k^v(n))$
    \ELSE 
    \STATE $\rho_n \leftarrow \rho$
    \ENDIF
    \STATE update dictionary $W_n$:
    \begin{align}
        &W_n \in \argmin_{W \in \Param_W}\Big[ \text{tr}(WA_nW^T) - 2\text{tr}(WB_n) + \frac{\rho_n}{2}\norm{W - W_{n-1}}_F^2 \Big]
    \end{align}
    \ENDFOR
    \STATE \textbf{output:} $\param_N$
    \end{algorithmic}
\end{algorithm}

\begin{algorithm}[H]
    \small
    \caption{Distributed Matrix Factorization with Diminishing Radius}
    \label{DNMF_DR}
    \begin{algorithmic}
    \STATE \textbf{Input:} $(X^v)_{v \in \mathcal{V}}$ (Data matrices in $\R^{p \times d}$); $W_0 \in \Param_W$ (initial dictionary); $N$ (number of iterations); $(r_n)_{n \geq 1}$ (diminishing radius search constraints)
    \STATE Compute initial codes $H_0^v \in \argmin_{H \in \Param_H^{\mathcal{V}}} \frac{1}{2}\norm{X^v - W_0H}_F^2 + \alpha \norm{H}_1$ for each $v \in \mathcal{V}$
    \FOR {$n=1$ {\bfseries to} $N$}
    \STATE sample an index $v_n$
    \STATE update $H_n^v \in \argmin_{H \in \Param_H^{\mathcal{V}}} \frac{1}{2} \norm{X_v - W_{n-1}H}_F^2 + \alpha \norm{H}_1$; $H_n^v = H_{n-1}^v$ for $v \neq v_n$
    \STATE $A_n \leftarrow A_{n-1} + \frac{1}{|\mathcal{V}|} \left[ H_n^{v_n} (H_n^{v_n})^T - H_{n-1}^{v_n}(H_{n-1}^{v_n})^T \right]$
    \STATE $B_n \leftarrow B_{n-1} + \frac{1}{|\mathcal{V}|} \left[ H_n^{v_n} (X^{v_n})^T - H_{n-1}^{v_n} (X^{v_n})^T \right]$
    \STATE update dictionary $W_n$:
    \begin{align}
        &W_n \in \argmin_{W \in \Param_W \cap B_{r_n}(W_{n-1})}\Big[ \text{tr}(WA_nW^T) - 2\text{tr}(WB_n) \Big] 
    \end{align}
    \ENDFOR
    \STATE \textbf{output:} $\param_N$
    \end{algorithmic}
\end{algorithm}


\end{document}